\documentclass[12pt,reqno]{amsart}

\usepackage{amssymb,url}
\usepackage[all]{xy}
\usepackage{graphicx}

\newtheorem{theorem}{Theorem}[section]
\newtheorem{proposition}[theorem]{Proposition}
\newtheorem{lemma}[theorem]{Lemma}
\newtheorem{corollary}[theorem]{Corollary}
\theoremstyle{definition}
\newtheorem{definition}[theorem]{Definition}
\newtheorem{example}[theorem]{Example}
\newtheorem{remark}[theorem]{Remark}
\numberwithin{equation}{section}

\newtheorem{alphatheorem}{Theorem}

\newcommand{\g}{\mathfrak{g}}
\newcommand{\C}{\mathbb{C}}
\newcommand{\Z}{\mathbb{Z}}

\newcommand{\X}{\mathfrak{X}}

\newcommand{\hm}{\mathrm{Hom}}

\newcommand{\Ad}{{\mathrm{Ad}_\rho}}
\newcommand{\git}{/\!\!/}
\newcommand{\SL}{\mathrm{SL}}
\newcommand{\GL}{\mathrm{GL}}
\newcommand{\tr}{\mathrm{tr}}
\newcommand{\ti}[1]{t_{(#1)}}
\newcommand{\q}{\mathfrak{q}}

\newcommand{\lieg}{\mathfrak{g}}
\newcommand{\SU}{\mathrm{SU}}

\makeatletter
\@namedef{subjclassname@2020}{\textup{2020} Mathematics Subject Classification}
\makeatother

\begin{document}
\baselineskip=15.5pt

\title[Poisson maps between character varieties]{Poisson maps between character varieties: gluing and capping}

\author[I. Biswas]{Indranil Biswas}

\address{School of Mathematics, Tata Institute of Fundamental
Research, Homi Bhabha Road, Mumbai 400005, India}

\email{indranil@math.tifr.res.in}

\author[J. Hurtubise]{Jacques Hurtubise}

\address{Department of Mathematics, McGill University, Burnside
Hall, 805 Sherbrooke St. W., Montreal, Que. H3A 2K6, Canada}

\email{jacques.hurtubise@mcgill.ca}

\author[L. C. Jeffrey]{Lisa C. Jeffrey}

\address{Department of Mathematics,
University of Toronto, Toronto, Ontario, Canada}

\email{jeffrey@math.toronto.edu}

\author[S. Lawton]{Sean Lawton}

\address{Department of Mathematical Sciences, George Mason University, 4400
University Drive, Fairfax, Virginia 22030, USA}

\email{slawton3@gmu.edu}

\subjclass[2020]{Primary 14M35, 53D30; Secondary 14L24}


\keywords{character variety, Poisson structure, gluing, capping}

\date{\today}

\begin{abstract}
Let $G$ be a compact Lie group or a complex reductive affine algebraic group.  We explore induced mappings between $G$-character varieties of surface groups by mappings between corresponding surfaces. It is shown that these mappings are generally Poisson. We also given an effective algorithm to compute the Poisson bi-vectors when $G=\SL(2,\C)$.  We demonstrate this algorithm by explicitly calculating the Poisson bi-vector for the 5-holed sphere, the first example for an Euler characteristic $-3$ surface.
\end{abstract}

\maketitle

\section{Introduction}\label{se1}
Suppose that $\Sigma_{n,g}$ is a surface of genus $g$ with $n$ boundary circles (or $n$ punctures), and $G$ is either a complex reductive affine algebraic group or compact Lie group.  The moduli space of representations of the fundamental group $\pi_1(\Sigma_{n,g})$ into $G$, the {\it $G$-character variety of $\pi_1(\Sigma_{n,g})$}, has a natural Poisson structure. This structure was given by Goldman \cite{Go} in the closed case, and extended in \cite{L3} to the case of surfaces with boundary. The representation space also has an equivalent interpretation as a space of flat connections, and from this point of view one can define the Poisson structure with an approach pioneered by Atiyah 
and Bott \cite{AB} when the surface is compact, and extended by Jeffrey in \cite{J} to punctured surfaces (see also \cite{BJ}). In this paper we consider the effect of mappings between surfaces. Given an appropriate mapping between two surfaces $f\,:\,\Sigma_1\,\longrightarrow \,\Sigma_2$, there is a natural morphism between their character varieties:
\[
\Phi\, :\, \X_{\Sigma_2}(G) \,\longrightarrow\, \X_{\Sigma_1}(G) \, .
\]
We will see that this is a Poisson map, both from the point of view of representations (Theorem \ref{thm:poisson}) and from the point of view of flat connections (Theorem \ref{thm1}).  Precisely, we prove:

\begin{alphatheorem}\label{thmA}Let $G$ be either a compact Lie group, or a complex reductive affine algebraic group. Let $q\,:\,\Sigma_1\,\longrightarrow\, \Sigma_2$ be a continuous map between compact orientable surfaces that preserves transversality of based loops, and double points. Then the induced map $$\q^*\, :\X_{\Sigma_2}(G)\, \longrightarrow\, \X_{\Sigma_1}(G)$$ is Poisson whenever $q$ preserves orientation, and is anti-Poisson if $q$ reverses orientation.
\end{alphatheorem}

Next, we review past computations of Poisson bi-vectors on character varieties and show that the bi-vector determines the underlying surface (Theorem \ref{thm:bivec}). In particular, we prove:

\begin{alphatheorem}
There is a homeomorphism between compact, connected orientable surfaces $\Sigma_{n_1,g_1}\,\cong\, \Sigma_{n_2,g_2}$ if and only if there is an equivalence of Poisson varieties $\X_{n_1,g_1}(G)\,\cong\, \X_{n_2,g_2}(G)$.
\end{alphatheorem}

We then extend (Theorem \ref{thm:5holed}) the known examples of explicitly computed bi-vectors on character varieties by determining the case of the 5-holed sphere $\Sigma_{5,0}$ and $G=\SL(2,\C)$.

\begin{alphatheorem}\label{thmC}
For any $n,g\geq 0$, there is an effective algorithm to compute the Poisson bi-vector of $\X_{n,g}(\SL(2,\C))$.  The bi-vector of $\X_{5,0}(\SL(2,\C))$ is:
 \begin{eqnarray*}
&&\mathfrak{a}_{5,0}(\SL(2,\C))\,=\,\mathfrak{a}_{1324}\frac{\partial}{\partial t_{\{1,3\}}}
\wedge \frac{\partial}{\partial t_{\{2,4\}}}+\Sigma_1\left(\mathfrak{a}_{1214}
\frac{\partial}{\partial t_{\{1,2\}}}\wedge\frac{\partial}{\partial t_{\{1,4\}}}\right)\\
&+&\Sigma_2\left(\mathfrak{a}_{12314}
\frac{\partial}{\partial t_{\{1,2,3\}}}\wedge \frac{\partial}{\partial t_{\{1,4\}}}+
\mathfrak{a}_{12324}\frac{\partial}{\partial t_{\{1,2,3\}}}\wedge
\frac{\partial}{\partial t_{\{2,4\}}}+\mathfrak{a}_{12334}
\frac{\partial}{\partial t_{\{1,2,3\}}}\wedge \frac{\partial}{\partial t_{\{3,4\}}}\right)\\
&+&\Sigma_2\left(\mathfrak{a}_{123124}\frac{\partial}{\partial t_{\{1,2,3\}}}\wedge
\frac{\partial}{\partial t_{\{1,2,4\}}}\right)+\Sigma_3\left(\mathfrak{a}_{123134}
\frac{\partial}{\partial t_{\{1,2,3\}}}\wedge \frac{\partial}{\partial t_{\{1,3,4\}}}\right),
\end{eqnarray*}
where $\Sigma_i$ are symmetry operators defined by the mapping class group of the surface and $\mathfrak{a}_x$ are explicit polynomials $($see Section \ref{5holed}$)$. 
\end{alphatheorem}

\begin{remark}
We give only one new example in Theorem \ref{thmC} since all cases of less complexity exist in the literature (references in Subsection \ref{subsec:examples}) and this is the first (and most tractable) example of a bi-vector for an Euler characteristic -3 surface.  Other examples are possible to compute by following the general algorithm we describe, however such examples are significantly more onerous to determine in detail.
\end{remark}

In Sections \ref{se2} and \ref{se3}, we provide an analytic point of view on Theorem \ref{thmA}.  In Section \ref{capping} we give simple examples, in which we ``cap'' some of the boundary circles of a surface $\Sigma_1$ (with a disk, a cylinder, a genus one surface, etc.) to obtain a surface $\Sigma_2$, and see what we obtain with $\Phi$. By capping all the boundaries, we obtain symplectic character varieties mapping to Poisson character varieties. In Section \ref{symquot}, we discuss gluing maps via symplectic quotients. The final section (Section \ref{5holed}) is devoted to the computer aided proof of Theorem \ref{thm:5holed}.

\subsection*{Acknowledgments}
Biswas is partially supported by  a J. C. Bose Fellowship.  Hurtubise and Jeffrey are each partially supported by an NSERC Discovery grant.  Lawton is partially supported by a Collaboration grant from the Simons Foundation, and thanks IHES for hosting him in 2021 when this work was completed.  We thank the referees for helpful suggestions.

\section{Poisson Structure on Character Varieties: Betti Point-of-View}

\subsection{Reductive Groups}

Let $G$ be a connected reductive affine algebraic group over $\C$. By the central isogeny theorem, $G\,\cong\, 
DG\times_F T$, where $DG\,=\,[G,\,G]$ is the derived subgroup, $T\,\cong\, (\C^*)^s$ is the maximal central torus, 
and $F\,=\,T\cap DG$. The group $G$ acts on itself by conjugation and the
geometric invariant theoretic quotient $G\git G$ is isomorphic to 
$\mathbf{T}/ W$, where $\mathbf{T}\, \subset\, G$ is a maximal torus and $W$ is the Weyl group 
$N_G(\mathbf{T})/\mathbf{T}$. By potentially enlarging $F$, we can assume $DG$ is simply connected. With that 
assumption made, by results of Steinberg \cite{St65}, we can say more: $$G\git G\,\cong\, \mathbf{T}\git W\,\cong\,\C^r\times_F 
(\C^*)^s,$$ where $r$ is the rank of $DG$.
Therefore, the coordinate ring $\C[G\git G]$ is isomorphic to $$\C[t_1,\,\dots ,\,t_r,\,d_1,\,\dots,
\,d_{s+1}]^F/(d_1\cdots d_{s+1}-1).$$ 
We denote points in $G\git G$ by $(\tau_1,\,\dots ,\,\tau_r,\,\delta_1,\,\dots ,\,\delta_s)$.
 
 \begin{example}
If $G$ is $\SL(n,\C)$, then $G\git G\,\cong\, \C^{n-1}$. Therefore, if $G\,=\,\GL(n,\C)$, then $G\git G\,\cong\,
\C^{n-1}\times_{\mathbb{Z}/n\mathbb{Z}}\C^*$, where $\mathbb{Z}/n\mathbb{Z}$ arises as the group of scalar matrices
of determinant 1. The coordinates describing points are the coefficients of the characteristic polynomial
which can take any value freely, except the determinant which can take any value except $0$.
\end{example}

\subsection{Character Varieties of Surfaces}

Let $\Sigma_{n,g}$ be a compact connected orientable surface of genus $g\,\geq\, 0$ with $n\,\geq\,0$
boundary components; we assume that $n\,\geq\, 2$ if $g\,=\,0$ since otherwise the surface is simply-connected
and the moduli spaces we will consider are trivial. Pick a base point $*$ in the interior of $\Sigma_{n,g}$.
The fundamental group of $\Sigma_{n,g}$ admits the presentation:
$$\pi_1(\Sigma_{n,g},\,*)\,\cong\, \langle a_1,\,b_1,\,\dots,\,a_g,\,b_g,\,c_1,\,\dots,\,c_n\ \Big\vert\ 
\prod_{i=1}^g[a_i,b_i]\prod_{j=1}^nc_j\,=\,1\rangle,$$ where $[x,\,y]\,=\,xyx^{-1}y^{-1}$ is the commutator.

The set of homomorphisms $\hm(\pi_1(\Sigma_{n,g},\,*),\,G)$ is naturally an affine algebraic subvariety of 
$G^{2g+n}$ by evaluating a homomorphisms at generators. The group $G$ acts rationally on 
$\hm(\pi_1(\Sigma_{n,g},\,*),\,G)$ by conjugation, that is, $g\cdot \rho\,=\,g\rho g^{-1}$. The geometric invariant 
theoretic quotient of this action is denoted
$$\X_{n,g}(G)\,:=\,\hm(\pi_1(\Sigma_{n,g},\,*),\,G)\git G,$$
and is called the {\it $G$-character variety of $\Sigma_{n,g}$}. More generally, if $(\Sigma,\,x_0)$ is any 
(pointed) compact orientable surface we will denote the $G$-character variety of $\pi_1(\Sigma,\,x_0)$ by 
$\X_{\Sigma,x_0}(G)$.

Note that the conjugation action of the center $Z(G)$ of $G$ is trivial and thus it suffices to consider the conjugation 
action of $PG\,:=\,G/Z(G)$; making it an effective action. The following lemma overlaps with \cite[Proposition 
49]{SikCharVar}.

\begin{lemma}\label{lem:dim}Assume $G$ is non-abelian.
The $($complex$)$ dimension of $\X_{n,g}(G)$ is $$-\chi(\Sigma_{n,g})\dim G +\zeta_{n,g},$$ where
\begin{itemize}
\item $\zeta_{n,g}\, =\,s$ if $n\,>\,0$ and $2g+n\,\geq \,3$,

\item $\zeta_{n,g}\,=\,r+s$ if $n\,>\,0$ and $2g+n\,=\,2$,

\item $\zeta_{n,g}\,=\,2s$ if $n\,=\,0$ and $g\,\geq\, 2$, and

\item $\zeta_{n,g}\,=\,2(r+s)$ if $n\,=\,0$ and $g\,=\,1$.
\end{itemize}
\end{lemma}

\begin{proof}
Recall our standing assumption that $n\,\geq\, 2$ if $g\,=\,0$, which rules out only the 2-sphere and the disk (both 
simply-connected).

If $n\,>\,0$, then $\Sigma_{n,g}$ deformation retracts to a 1-complex whose fundamental group is free of rank 
$2g+n-1$. In this case $\chi(\Sigma_{n,g})\,=\,1-(2g+n-1)\,=\,2-2g-n$ and so the rank of $\pi_1(\Sigma_{n,g},\,*)$ is 
$1-\chi(\Sigma_{n,g})$. If $n\,=\,0$, then the Euler characteristic is $1-2g+1\,=\,2-2g$. Either way, 
$\chi(\Sigma_{n,g})\,=\,2-2g-n$.

When $n\,>\,0$ then $\pi_1(\Sigma_{n,g},\,*)$ is a free group of rank $2g+n-1$ which is greater than or equal to 2 if 
and only if $g\,\geq\, 1$ or $n\,\geq\, 3$. Thus, $\pi_1(\Sigma_{n,g},\,*)$ surjects onto a rank 2 free group
$F_2$, and 
so $\hm(F_2,\, G)$ injects into $\hm(\pi_1(\Sigma_{n,g},\,*),\, G)$. The generic dimension of a $PG$-conjugation 
stabilizer of $\hm(F_2,\,G)$ is 0 since a generic pair of elements in $G$ generates a Zariski dense subgroup; hence 
the same is true for $\hm(\pi_1(\Sigma_{n,g},\,*),\, G)$.

Thus, since $\hm(\pi_1(\Sigma_{n,g},\,*),\, G)\,\cong\, G^{2g+n-1}$ the dimension of $\X_{n,g}(G)$ is
$(2g+n-1)\dim G-\dim PG\,=\,(2g+n-1)\dim G-\dim G +\dim Z(G)\,=\,-\chi(\Sigma_{n,g})\dim(G)+s$.

Now, still assuming $n\,>\,0$, if $\pi_1(\Sigma_{n,g},\,*)$ is a free group of rank 1 (only occurring when $g\,=\,0$ and 
$n\,=\,2$), then the character variety is isomorphic to $G\git G$ which we have already seen is of dimension $r+s$ 
which is also the dimension of a generic $G$-conjugation stabilizer.

Likewise, if $n\,=\,0$ and $g\,\geq \,2$ then $\pi_1(\Sigma_{n,g},\,*)$ surjects onto a free group of rank 2. 
Thus, we have that generic dimension of a stabilizer is the dimension of $Z(G)$, which is $s$. Moreover, the 
commutator map $G^{2g}\,\longrightarrow\, DG$ (defining the relation in $\pi_1(\Sigma_{0,g},\,*)$) is dominant 
(and hence a generic submersion), and therefore $\dim \hm(\pi_1(\Sigma_{0,g},\,*),G)\,=\,2g\dim G -\dim DG$ which 
then implies $\dim\X_{0,g}(G) \,=\,2g\dim G -\dim DG -\dim PG\, =\,-\chi(\Sigma_{0,g})\dim G +2\dim 
Z(G)\,=\,-\chi(\Sigma_{0,g})\dim G +2s$.

Lastly, when $n\,=\,0$ and $g\,=\,1$, then the identity component of $\X_{0,1}(G)$ satisfies $\X^0_{0,1}(G)\,\cong
\,\mathbf{T}^2/W$ which has dimension $2(r+s)$, see \cite{FlLa4}. We note that if $G$ is simply connected then 
$\X_{0,1}(G)$ is connected but otherwise it has smaller dimensional components. Regardless, in this case 
$\chi(\Sigma_{0,1})\,=\,0$, and so we have established the formula.
\end{proof}

\begin{remark}
If $G$ is abelian then the conjugation action is trivial, and so \begin{eqnarray*}\X_{n,g}(G)&=&\hm(\pi_1(\Sigma_{n,g},\,*),\, G)\\
&=&\hm(\pi_1(\Sigma_{n,g},\,*)/[\pi_1(\Sigma_{n,g},\,*),
\,\pi_1(\Sigma_{n,g},\,*)],\,G)\\&=&\hm(\Z^{\epsilon_{n,g}},\,G)\,=\,G^{\epsilon_{n,g}},\end{eqnarray*} where
$\epsilon_{0,g}\,=\,2g$ and $\epsilon_{n,g}\,=\,2g+n-1$ if $n\,>\,0$.
The dimension in these cases is obvious. In the cases when $g\,=\,0$ and $n\,=\,0$ or $1$, then the character variety is a point and so has dimension 0.
\end{remark}

The algebraic structure of $\X_{n,g}(G)$, up to biregular mappings, does not depend on the presentation of the 
fundamental group of $\Sigma_{n,g}$. In fact, it only depends on the Euler characteristic $\chi(\Sigma_{n,g})$.

\begin{proposition}\label{prop:euler}
There is a biregular morphism $\X_{n_1,g_1}(G)\,\cong\, \X_{n_1,g_2}(G)$ if and only if both $n_1,\, n_2$ are either
positive or 0, and $\chi(\Sigma_{n_1,g_1})\,=\,\chi(\Sigma_{n_2,g_2})$.
\end{proposition}

\begin{proof}
Assume first that $\X_{n_1,g_1}(G)\,\cong\, \X_{n_2,g_2}(G)$. Then their dimensions are equal. Suppose $n_1\,>\,0$ and 
$n_2\,=\,0$. Then $\X_{n_1,g_1}(G)$ strong deformation retracts onto $\X_{n_1,g_1}(K)$, where $K$ is a maximal compact 
subgroup of $G$, while $\X_{0,g_2}(G)$ is not even homotopy equivalent to $\X_{0,g_2}(K)$ as long as $g_2\,\geq\, 2$ 
by \cite{FlLa1, BiFl, FlLa5}. Thus, if $\X_{n_1,g_1}(G)\,\cong\, \X_{n_1,g_2}(G)$, then both surfaces 
are open or both are closed. From Lemma \ref{lem:dim} we conclude that 
$\chi(\Sigma_{n_1,g_1})\,=\,\chi(\Sigma_{n_2,g_2}).$ Note that this deduction holds {\it without} the assumptions on 
$n_1$ and $n_2$.

Conversely, assume that $\chi(\Sigma_{n_1,g_1})\,=\,\chi(\Sigma_{n_2,g_2}).$ There are three cases to
consider: (1) $n_1n_2\,>\,0$, (2) $n_1\,=\,0\,=\,n_2$, and (3) $n_1n_2\,=\,0$ but $n_1+n_2\,\not=\,0$.

In Case (1), both the fundamental groups $\pi_1(\Sigma_{n_1,g_1},\,*)$ and $\pi_1(\Sigma_{n_2,g_2},\,*)$
are free of the same rank, and hence the character varieties are isomorphic (note 
that the surfaces need not be homeomorphic, but they will be homotopic).

In Case (2), the Euler characteristics being equal implies the surfaces are homeomorphic and hence their fundamental
groups are isomorphic. Hence the character varieties too are isomorphic.

Case (3) does occur since the Euler characteristics can be equal (with one surface open and the other closed)
since $2-2g_1-n_1\,=\,2-2g_2$ requires only that $n_1\,=\,2(g_1-g_2)$. In this situation (without loss of generality
assume $n_1\,>\,0$ and $n_2\,=\,0$), as noted in the first paragraph of this proof, $\X_{n_1,g_1}(G)$ is homotopic
to $\X_{n_1,g_1}(K)$, while $\X_{0,g_2}(G)$ is not homotopic to $\X_{0,g_2}(K)$ for $g_2\,\geq\, 2$. So the
converse (without additional assumptions on $n_1$ and $n_2$) does not hold.

To exhaust the possibilities with Case (3), suppose $g_2\,=\,1$, then $2-2g_1-n_1\,=\,0$ and hence $2g_1+n_1\,=\,2$
and so we have that $\X_{n_1,g_1}(G)\,\cong\, G\git G$ of dimension $s+t$ while $\X_{0,g_2}(G)$ has dimension
$2(s+t)$, and so they are not isomorphic. In short, every time Case (3) arises we have simultaneously
that the Euler characteristics are equal yet the character varieties are not isomorphic.
\end{proof}

\begin{remark}
Proposition \ref{prop:euler} is non-trivial in the sense that in general if two character varieties are isomorphic
it does not imply the underlying groups are isomorphic. As a simple example of this observe that
$$\hm(\Z/2\Z,\,\SL(2,\C))\git \SL(2,\C)\,\cong\,\hm(\Z/3\Z,\,\SL(2,\C))\git \SL(2,\C)$$ as each is a set of two
points.
\end{remark}

\subsection{Relative Character Varieties}

When $n\,>\,0$, for every $1\, \leq\, i\, \leq\, n$ define the boundary map
$$\mathfrak{b}_i\,:\,\X_{n,g}(G)\,\longrightarrow\, G\git G$$ by sending a representation class $[\rho]$ to $[\rho_{|_{c_i}}]$. Subsequently, we define $$\mathfrak{b}_{n,g}\,:=\,(\mathfrak{b}_1,\,\dots,\,\mathfrak{b}_n)
\,:\,\X_{n,g}(G)\,\longrightarrow\, (G\git G)^{ n}.$$ We emphasize that the map $\mathfrak{b}_{n,g}$ depends on the surface, not only its fundamental group.

Let $\tau\,\in\, \mathfrak{b}_{n,g}\left(\X_{n,g}(G)\right)\,\subset\, (G\git G)^n$ be a point in the image of the 
boundary map and define $\mathfrak{L}_\tau\,=\,\mathfrak{b}_{n,g}^{-1}(\tau).$ The singular locus of $\X_{n,g}(G)$ is 
a proper closed sub-variety; denote its complement by $\mathcal{X}_{n,g}(G)$. So $\mathcal{X}_{n,g}(G)$ is a complex 
manifold that is dense in $\X_{n,g}(G)$. Since $\mathfrak{b}_{n,g}$ is dominant, its regular values are generic.
Thus, at such a point, $\mathcal{L}_\tau\,:=\,\mathfrak{L}_\tau\cap\mathcal{X}_{n,g}(G)$ is a submanifold of dimension 
$$\chi(\Sigma_{n,g})\dim G +\zeta_{n,g}-n(r+s).$$

It is shown in \cite{L3} that $\cup_\tau\mathcal{L}_\tau$ foliate $\mathcal{X}_{n,g}(G)$ by complex symplectic 
submanifolds, making $\mathcal{X}_{n,g}(G)$ a complex Poisson manifold. This structure continuously extends over 
all of $\X_{n,g}(G)$ making it a Poisson variety; a variety whose sheaf of regular functions is a sheaf of Poisson 
algebras (see \cite{BLR} for details).

We now review the explicit definition of this structure.

\subsection{Poisson Structure}

For an affine variety $V$ defined over $\C$, a Poisson structure on $V$ is a Lie bracket operation $\{\, ,\,\}$ on its coordinate ring $\C[V]$ that acts as a formal derivation (satisfies the Leibniz rule).

The smooth stratum of $V$, denoted $\mathcal{V}$, is a complex Poisson manifold in the usual sense by the Stone-Weierstrass Theorem. For any holomorphic function $f$ on $\mathcal{V}$, there is a Hamiltonian vector field $H_f\,:=\,\{f,\cdot \}$ on $\mathcal{V}$ defined in terms of the Poisson bracket. There exists an exterior bi-vector field $\mathfrak{a}\,\in\, H^0(\mathcal{V},\,\bigwedge^{2} T \mathcal{V})$ whose restriction to symplectic
leaves (with $(2,0)$-form $\omega$) is given by $\{f,\,g\}\,=\,\omega(H_g,\,H_f)$.
Let $f,\,g\,\in\, \C[V]$. Then with respect to interior multiplication $\{f,\,g\}
\,=\,\mathfrak{a}\cdot (df\wedge dg)\,=\, (df\wedge dg)(\mathfrak{a}).$ In local coordinates $(z_1,\, \dots,\,z_k)$ it takes the form 
$$\mathfrak{a}\,=\,\sum_{i,j}\mathfrak{a}_{i,j}\frac{\partial}{\partial z_i}\land 
\frac{\partial}{\partial z_j}$$ and so 
\begin{align*}
\{f,g\}&\,=\,\sum_{i,j}\left(\mathfrak{a}_{i,j}\frac{\partial}{\partial z_i}\land \frac{\partial}{\partial z_j}\right)\cdot\left(\frac{\partial f}{\partial z_i} dz_i \wedge \frac{\partial g}{\partial 
z_j}dz_j\right)\\
&\,=\,\sum_{i,j}\mathfrak{a}_{i,j}\left(\frac{\partial f}{\partial z_i}\frac{\partial g}{\partial z_j}-
\frac{\partial f}{\partial z_j}\frac{\partial g}{\partial z_i}\right).
\end{align*}

Any reductive $G$ has a symmetric, non-degenerate bilinear form $\mathfrak{B}$ on its Lie algebra $\mathfrak g$ that is invariant under the adjoint representation. Fix such an invariant form $\mathfrak{B}\,:\,\mathfrak{g}\times 
\mathfrak{g}\,\longrightarrow\, \C$. If $G$ is semisimple $\mathfrak{B}$ is a multiple of the Killing form. 

Returning to our varieties $\X_{n,g}(G)$, in \cite{GHJW} it is established that $\omega$, in the following  commutative diagram, defines a symplectic form on the leaf $\mathcal{L}_\tau$:
$$
\xymatrix{
H^1(\Sigma_{n,g},\,\partial \Sigma_{n,g};\,\g_\Ad) \times H^1(\Sigma_{n,g};\,\g_\Ad) \ar[r]^-\cup &
H^2(\Sigma_{n,g},\,\partial \Sigma_{n,g};\,\g_\Ad\otimes \g_\Ad) \ar[d]^{\mathfrak{B}_*}\\
& H^2(\Sigma_{n,g},\,\partial \Sigma_{n,g};\,\C) \ar[d]^{\cap [Z]}\\
H^1_{\mathrm{par}}(\Sigma_{n,g};\,\g_\Ad)\times H^1_{\mathrm{par}}(\Sigma_{n,g};\,\g_\Ad)\ar[uu] \ar[r]^-\omega &
H_0(\Sigma_{n,g};\,\C)\,\cong\,\C. }
$$
Note that $H^1_{\mathrm{par}}(\Sigma_{n,g};\,\g_\Ad)$ is a model for the tangent space at a class $[\rho]$ in $\mathcal{L}_\tau$.  

With respect to this 2-form, in \cite{L3}, it is shown that Goldman's proof \cite{Go,G86} of the Poisson bracket in the closed surface case generalizes directly to relative and parabolic cohomology and establishes a Poisson bracket on the coordinate ring $\C[\X_{n,g}(G)]$.

Let $\alpha,\, \beta \,\in\, \pi_1(\Sigma_{n,g},\,*)$. Up to homotopy, we can always arrange for $\alpha$ and $\beta$ to intersect at worst in transverse double points. Let $\alpha\cap\beta$ be the set of (transverse) double point intersections of $\alpha$ and $\beta$. Let $\epsilon(p,\,\alpha,\, \beta)$ be the oriented intersection number at $p\,\in \,\alpha\cap\beta$ and let $\alpha_p\,\in\, \pi_1(\Sigma,\,p)$ be the curve $\alpha$ based at $p$.

For a given $f\,\in \,\C[G\git G]$ we obtain $f_\alpha\,:\,\X_{n,g}(G)\,\longrightarrow\, \C$ defined by $f_\alpha([\rho])\,=\,f(\rho(\alpha))$. Define the variation $F$ of an invariant function $f$ by $$\mathfrak{B}(F(\mathbf{A}),\,X)\,=\,\frac{d}{dt}\Big|_{t=0}f(\mathrm{exp}(tX)\mathbf{A}).$$
In special cases, $F$ can be computed explicitly; see \cite{G86} for further details.
In these terms the bracket is defined on $\mathbb{C}[\mathfrak{X}_{n,g}(G)]$ by:
\begin{align}\label{eq:poisson}
\{f_\alpha([\rho]),\,g_\beta([\rho])\}&\,=\,\sum_{p\in
\alpha\cap\beta}\epsilon(p,\,\alpha,\,\beta)\mathfrak{B}(F_{\alpha_p}([\rho]),\, G_{\beta_p}([\rho])).
\end{align}
See \cite[Sections 3 and 4]{L3} for further details when $n\,>\,0$, \cite{G86} when $n\,=\,0$ and $g\,\geq\, 2$, and \cite{SikAb} for $n\,=\,0$ and $g\,=\,1$. We will denote the bi-vector associated to this Poisson bracket on $\X_{n,g}(G)$ 
by $\mathfrak{a}_{n,g}(G)$. 

Note that when $\alpha$ represents one of the boundary curves in $\Sigma_{n,g}$, it can be chosen to not intersect any of the other generators of the fundamental group. Consequently, Formula \eqref{eq:poisson} implies that $f_\alpha$ Poisson commutes with all other functions; such functions are called {\it Casimirs}.

\subsection{Bi-vectors on Character Varieties}\label{subsec:examples}

In contrast to Proposition \ref{prop:euler}, the Poisson bi-vector completely determines the isomorphism class of the underlying surface.

\begin{theorem}\label{thm:bivec}
There is a homeomorphism $\Sigma_{n_1,g_1}\,\cong\, \Sigma_{n_2,g_2}$ if and only if there is an equivalence of Poisson varieties $\X_{n_1,g_1}(G)\,\cong\, \X_{n_2,g_2}(G)$.
\end{theorem}

\begin{proof}
The forward direction is obvious. We break the converse direction into three cases.

First, assume $n_1,\,n_2\,>\,0$. Then the Casimir subalgebra of $\C[\X_{n_1,g_1}(G)]$ differs from the Casimir subalgebra of $\C[\X_{n_1,g_1}(G)]$ unless $n_1\,=\,n_2$. In that case, the Euler characteristic, which is read off the dimension of $\X_{n,g}(G)$, determines the genus and so $g_1\,=\,g_2$ and we are done.

Second, assume $n_1\,=\,0$ but $n_2\,>\,0$ (which implies that the two surfaces are not isomorphic). In that case, $\X_{n_1,g_1}(G)$ is symplectic but $\X_{n_2,g_2}(G)$ is not (non-trivial Casimirs). Thus, their bi-vectors could not be equivalent either.

Lastly, assume that $n_1\,=\,0\,=\,n_2$. Then Proposition \ref{prop:euler} tells us that since the dimensions of $\X_{n_1,g_1}(G)$ and $\X_{n_2,g_2}(G)$ are the same the Euler 
characteristics of the surfaces are the same. But since each surface is closed, they are 
isomorphic.
\end{proof}

We next consider some examples. Since the character variety is a point when $(n,\,g)\,=\,(0,\,0)$ or $(1,\,0)$ these cases are trivial for any $G$. Likewise, for any $G$ the $(2,\,0)$ case has 0 dimensional symplectic leaves and so the bi-vector is trivial.

The next simplest example is the 3-holed sphere. For $G\,=\,\SL(2,\C)$,
we have $\X_{3,0}(G)\,\cong\, \C^3$ with coordinates $\tr_{c_1}, \tr_{c_2},$ and $\tr_{c_1c_2}$ (see \cite{ABL} for a proof). Since the boundary curves are 
disjoint, they have no intersections and thus the Poisson bracket is trivial. Alternatively, the symplectic leaves are the level sets obtained by fixing the three boundary invariants. But since each point in $\X_{3,0}(G)$ is uniquely determined by $\tr_{c_1}, \tr_{c_2},$ and $\tr_{c_1c_2}$, each symplectic leaf is a point.

When $G\,=\,\SL(3,\C)$ the bi-vector for $\Sigma_{3,0}$ was worked out in \cite{La06, L3}. Unlike the case of $\SL(2,\C)$ where the symplectic leaves are 0 dimensional, the symplectic leaves in $\X_{3,0}(\SL(3,\C))$ are 2 dimensional.

To describe it we need to briefly review the structure of $\hm(F_2,\,\SL(3,\C))\git \SL(3,\C)$ from \cite{La06,La07}. The $\SL(3,\C)$-character variety of a free group $F_2$ of rank 2 is a hypersurface in $\C^9$, which is a branched double cover of $\C^8$ under projection. The coordinate ring is generated by 9 trace functions of simple closed curves (in the 1-holed torus) denoted $\{\ti{\pm 1},\, \dots ,\,\ti{\pm 4},\ti{5}\}$, and satisfies a single relation of the form $\ti{5}^2-P\ti{5}+Q$ where $P,\,Q\,\in\, \C[\ti{\pm 1},\,\dots,\,\ti{\pm 4}]$. Let $\mathfrak{a}_{i,j}\,=\,\{\ti{i},\ti{j}\}$. In these terms, the Poisson bi-vector is:
$$\mathfrak{a}_{3,0}(G)\,=\,(P-2\ti{5})\frac{\partial}{\partial \ti{4}}\land 
\frac{\partial}{\partial \ti{-4}} +(1-\mathfrak{i})\left(\mathfrak{a}_{4,5}
\frac{\partial}{\partial \ti{4}}\land \frac{\partial}{\partial \ti{5}} \right),$$
where $\mathfrak{a}_{4,5}\,=\,\frac{\partial}{\partial \ti{-4}}(Q-\ti{5}P)$,
and $\mathfrak{i}\,\in\, \mathrm{Out}(F_2)$ is the outer automorphism of $F_2\,=\,
\langle a,\,b\rangle$ defined by $a\,\longmapsto \,a^{-1}$ and $b\,\longmapsto\, b^{-1}$.

The other surface with Euler $-1$ is the 1-holed torus. In this case, the 
bi-vector for $\SL(2,\C)$ is computed in \cite{G06}, and for $\SL(3,\C)$ it is computed in \cite{L3,L4}. Additionally, in \cite{G06} the bi-vector is computed for $\SL(2,\C)$ and the Euler characteristic $-2$ open surfaces: the 4-holed sphere and the 2-holed torus.  No other examples have been computed. 

In Section \ref{5holed} we add to the known examples by computing the bi-vector for the 5-holed sphere (one of three Euler characteristic $-3$ orientable surfaces). This computation uses the computational program in \cite{ABL} that allows one to compute the generators and relations of {\it any} $\SL(2,\C)$-character variety. There are 45 required computations and diagrams. In fact, the algorithm we describe and use is {\it effective}, that is,  the algorithm terminates after a finite number of steps that can in principle be done ``by hand," and always produces a correct answer if the steps are correctly followed.  Here is the theorem:

\begin{theorem}\label{thm:5holed}
For any $n,g\geq 0$, there is an effective algorithm to compute the Poisson bi-vector of $\X_{n,g}(\SL(2,\C))$
The bi-vector of $\X_{5,0}(\SL(2,\C))$ is:
 \begin{eqnarray*}
&&\mathfrak{a}_{5,0}(\SL(2,\C))\,=\,\mathfrak{a}_{1324}\frac{\partial}{\partial t_{\{1,3\}}}
\wedge \frac{\partial}{\partial t_{\{2,4\}}}+\Sigma_1\left(\mathfrak{a}_{1214}
\frac{\partial}{\partial t_{\{1,2\}}}\wedge\frac{\partial}{\partial t_{\{1,4\}}}\right)\\
&+&\Sigma_2\left(\mathfrak{a}_{12314}
\frac{\partial}{\partial t_{\{1,2,3\}}}\wedge \frac{\partial}{\partial t_{\{1,4\}}}+
\mathfrak{a}_{12324}\frac{\partial}{\partial t_{\{1,2,3\}}}\wedge
\frac{\partial}{\partial t_{\{2,4\}}}+\mathfrak{a}_{12334}
\frac{\partial}{\partial t_{\{1,2,3\}}}\wedge \frac{\partial}{\partial t_{\{3,4\}}}\right)\\
&+&\Sigma_2\left(\mathfrak{a}_{123124}\frac{\partial}{\partial t_{\{1,2,3\}}}\wedge
\frac{\partial}{\partial t_{\{1,2,4\}}}\right)+\Sigma_3\left(\mathfrak{a}_{123134}
\frac{\partial}{\partial t_{\{1,2,3\}}}\wedge \frac{\partial}{\partial t_{\{1,3,4\}}}\right),
\end{eqnarray*}
where $\Sigma_i$ are symmetry operators defined by the mapping class group of the surface and $\mathfrak{a}_x$ are explicit polynomials; both are described in detail in Section \ref{5holed}. Moreover, the polynomial coefficients do not exhibit any further mapping class group symmetry from boundary permutation.
\end{theorem}

We remark again that this is the first explicit example of the Poisson structure on a 
character variety of an Euler characteristic -3 surface.  

\subsection{Mappings Between Surfaces}

Let $\Sigma_1$ and $\Sigma_2$ be compact orientable surfaces (possibly with boundary), and 
$G$, as before, is a reductive affine algebraic group over $\C$. If $q\,:\,\Sigma_1\,
\longrightarrow\, \Sigma_2$ is a continuous map and $q(x)\,=\,y$, then there is an induced homomorphism $$q_\#\,:\,\pi_1(\Sigma_1,\,x)\,\longrightarrow\, \pi_1(\Sigma_2,\,y).$$ In turn, we have an induced {\it continuous} map $$q^*\,:\,\hm(\pi_1(\Sigma_2,y),\, G)\,\longrightarrow\,\hm(\pi_1(\Sigma_1,\,x),\,G)$$ given by $q^*(\rho)\,=\,\rho\circ q_\#$. This function is equivariant with respect to $G$-conjugation, and thus there is a morphism $$\mathfrak{q}^*\,:\,\hm(\pi_1(\Sigma_2,\,y),\,G)\git G\,\longrightarrow\,
\hm(\pi_1(\Sigma_1,\,x),\,G)\git G$$ given by $\mathfrak{q}^*([\rho])\,=\,[\rho\circ q_\#].$ Lastly, we have an algebra morphism between coordinate rings $$\mathfrak{q}_*\,:\,\C[\hm(\pi_1(\Sigma_1,\,x),\,G)]^G\,\longrightarrow\, 
\C[\hm(\pi_1(\Sigma_2,\,y),\,G)]^G$$ given by $\mathfrak{q}_*(f)(\rho)\,=\,f(\rho\circ q_\#).$

\begin{theorem}\label{thm:poisson}
Let $q\,:\,\Sigma_1\,\longrightarrow\, \Sigma_2$ be a continuous map between compact orientable
surfaces that preserves transversality of based loops, and double points.
Then the induced algebra morphism of coordinate rings $\mathfrak{q}_*\,:\,
\C[\X_{\Sigma_1,x}(G)]\,\longrightarrow\,\C[\X_{\Sigma_2,y}(G)]$ is a morphism of
Poisson algebras if $q$ preserves orientation and is an anti-Poisson morphism if $q$ reverses orientation.  Regardless, the image of $\mathfrak{q}_*$ is a Poisson subalgebra.
\end{theorem}

\begin{proof}
The first part of the theorem implies the second, so we only prove that. Since 
$\mathfrak{q}_*$ is an algebra morphism and the bracket is a derivation, it is enough to 
verify the claim on all generators of the algebra.

Since $q$ preserves transversality of based loops, double points, and either preserves or 
reverses (globally) orientation, it follows that for any two based loops $\alpha$ and 
$\beta$ in $\Sigma_1$ used in computing the bi-vector $\mathfrak{a}_{\Sigma_1}(G)$ we have 
from Equation \eqref{eq:poisson}:
\begin{align}\label{qbracket}
&\q_*\left(\{f_\alpha([\rho]), \,g_\beta([\rho])\}\right)\\
= &\sum_{q(p)\in q(\alpha)\cap q(\beta)}\!\!\!\!\!\!\epsilon(q(p),\,q(\alpha),\,q(\beta))\mathfrak{B}(F_{q(\alpha)_{q(p)}}([q_\#(\rho)]), \,G_{q(\beta)_{q(p)}}([q_\#(\rho)]) \nonumber \\
=\, &\pm \{f_{q(\alpha)}([q_\#(\rho)]),\,g_{q(\beta)}([q_\#(\rho)])\} \nonumber \\
=\, &\pm \{\q_*(f_\alpha([\rho])),\,\q_*(g_\beta([\rho]))\}.\nonumber
\end{align}
However, the intersection numbers $\epsilon(q(p),\,q(\alpha),\,q(\beta))$ and
$\epsilon(p,\,\alpha,\,\beta)$ will be reversed if $q$ reverses orientation and will be preserved if $q$ preserves orientation.
\end{proof}

\begin{corollary}\label{cor:poisson}
The map between character varieties induced by $q$ in Theorem \ref{thm:poisson}
$$\q^*\, :\X_{\Sigma_2,y}(G)\, \longrightarrow\, \X_{\Sigma_1,x}(G)$$
is Poisson whenever $q$ preserves transversality of based loops, double points and orientation.
\end{corollary}

\begin{proof}
Whenever there is a morphism of Poisson algebras, the dual map, between spaces, is always
Poisson \cite[Chapter 1]{LPV}. Thus, the corollary follows from Theorem \ref{thm:poisson}. 
\end{proof}

\begin{remark}
For $K$ a compact Lie group, $\X_{n,g}(K)\,=\,\hm(\pi_1(\Sigma_{n,g}),K)/K$ is semi-algebraic. As 
the complexification of $K$ is a reductive algebraic group $G$, $\X_{n,g}(K)$ naturally embeds into the 
$\mathbb{R}$-locus of $\X_{n,g}(G)$ by \cite[Theorem 4.3]{FlLa2}. Consequently, Equation \eqref{eq:poisson} defines a Poisson bracket on the 
real coordinate ring of $\X_{n,g}(K)$ by restriction of scalars. Thus, Theorem \ref{thm:poisson} and its corollary 
remain valid in this context as well. The proof is exactly the same.
\end{remark}

\begin{remark}
More generally, for a real form $H$ of $G$, the map $\X_{n,g}(H)\to \X_{n,g}(G)$ need not be an embedding (as it is for the compact real form), but it will be a finite map by the paragraph following \cite[Proposition 6.1]{CFLO}.  It would be interesting to explore if the above theorems remain valid in this context.
\end{remark}

\begin{example}
For any two surfaces $\Sigma_{n_1,g_1}$ and $\Sigma_{n_2,g_2}$ with $n_1>n_2>0$ and $\chi(\Sigma_{n_1,g_1})=\chi(\Sigma_{n_2,g_2})$, there is a quotient mapping $q:\Sigma_{n_1,g_1}\longrightarrow\Sigma_{n_2,g_2}$ identifying one or more pairs of boundary components. Both $\pi_1(\Sigma_{n_1,g_1},*)$ and $\pi_1(\Sigma_{n_2,g_2},q(*))$ are isomorphic to a free group of rank $1-\chi(\Sigma_{n_1,g_1})$. Since $q$ satisfies the conditions of Theorem \ref{thm:poisson}, the induced gluing map $$\mathfrak{q}^*:\X_{n_2,g_2}(G)\longrightarrow\X_{n_1,g_1}(G)$$ is Poisson. Since $\X_{n_2,g_2}(G)\,\cong\, \X_{n_1,g_1}(G)$ we have Poisson morphisms between isomorphic varieties with {\it different} Poisson structures. For a detailed example of this phenomena see \cite{L4}. 
\end{example}

\begin{example}
Another natural example that satisfies the conditions of Theorem \ref{thm:poisson} is the inclusion $\iota:\Sigma_1\hookrightarrow\Sigma_2$ of a subsurface $\Sigma_1$ into a surface $\Sigma_2$.  In this case, the induced map on character varieties $\mathfrak{q}^*:\X_{\Sigma_2,y}(G)\longrightarrow\X_{\Sigma_1,x}(G)$, which is Poisson, is the restriction map.  In the case when $\Sigma_2$ is closed and $\Sigma_1$ has boundary, we have a Poisson map whose domain is symplectic.
\end{example}

\section{Poisson structure on character varieties: De Rham Point-of-View}\label{se2}

Let, almost as before, $\Sigma_{n,g}$ be a connected orientable surface of genus $g\,\geq\, 0$ 
with $n\,\geq\, 0$ punctures, the differences being that:
\begin{enumerate}
\item the surface is $C^\infty$, not just topological, and

\item unlike in the previous case where we had surfaces with boundary, we only
consider here their interior, so that we have open surfaces when $n\,>\,0$.
\end{enumerate}
Let now $G$ be connected Lie group such that either:
\begin{itemize}
\item $G$ is compact, or

\item $G$ is a reductive affine algebraic group defined over the field of complex numbers.
\end{itemize}
Note that the reductive affine algebraic group defined over $\mathbb C$
are complexifications of compact connected Lie groups.

Fix a base point $x_0\, \in\, \Sigma_{n,g}$. When $G$ is a reductive affine algebraic group defined over $\mathbb C$,
a homomorphism $\rho\, :\, \pi_1(\Sigma_{n,g},\, x_0)\, \longrightarrow\, G$ is called \textit{reductive} if the
Zariski closure of $\rho(\pi_1(\Sigma_{n,g},\, x_0))$ in $G$ is a reductive subgroup. We note that $\rho$ is
reductive if and only if satisfies the following condition: if $P$ is a parabolic subgroup of $G$ such that
$\rho(\pi_1(\Sigma_{n,g},\, x_0))\, \subset\, P$, then $\rho(\pi_1(\Sigma_{n,g},\, x_0))$ is contained in a Levi
factor of $P$ (see \cite[11.2]{Bo} and \cite[p.~184]{Hu} for parabolic subgroups and their Levi factors).

If $G$ is a compact Lie group, then all homomorphisms $\rho\, :\, \pi_1(\Sigma_{n,g},\, x_0)\, \longrightarrow\, G$
are reductive.

Let
$$
{\mathcal R}^r_G(\Sigma_{n,g})\, \subset\, \text{Hom}(\pi_1(\Sigma_{n,g},\, x_0),\, G)
$$ be the space of 
all reductive homomorphisms. The adjoint action of $G$ on itself produces an action of $G$ on ${\mathcal 
R}^r_G(\Sigma_{n,g})$. The corresponding quotient space
\begin{equation}\label{e1}
\X_{n,g}(G)\, :=\, {\mathcal R}^r_G(\Sigma_{n,g})/G
\end{equation}
is again the character variety, this time with $G$ either reductive over the 
complex numbers or compact. Note that $\X_{n,g}(G)$ is actually independent of the choice of the base point 
$x_0$. When $G$ is a reductive affine algebraic group defined over $\mathbb C$, then, as we have seen,
$\text{Hom}(\pi_1(\Sigma_{n,g},\, x_0),\, G)$ is a complex affine variety, because $G$ is a complex affine variety 
and the group $\pi_1(\Sigma_{n,g},\, x_0)$ is finitely generated. In this case, restricting to reductive 
representations and taking the ordinary quotient coincides with the geometric invariant theoretic quotient 
$\text{Hom}(\pi_1(\Sigma_{n,g},\, x_0),\, G)\git G$ (in the analytic topology) by \cite[Theorem 2.1]{FlLa4}, and 
so inherits a natural algebraic structure. It is also homotopic to the non-Hausdorff quotient
$\text{Hom}(\pi_1(\Sigma_{n,g},\, x_0),\, G)/G$ by \cite[Proposition 3.4]{FLR}.

A $G$--connection on $\Sigma_{n,g}$ is a $C^\infty$ principal $G$--bundle on $\Sigma_{n,g}$ equipped with a connection.
If the curvature of a connection vanishes identically, then it is called a flat $G$--connection.
A flat $G$--connection on $\Sigma_{n,g}$ is called reductive if the corresponding monodromy homomorphism
is reductive. The character variety is identified with the moduli space of reductive flat $G$--connections
on $\Sigma_{n,g}$. This identification sends a flat $G$--connection to the monodromy homomorphism corresponding to
the flat connection.

Let $E_G$ be a $C^\infty$ principal $G$--bundle on $\Sigma_{n,g}$, and let $\nabla$ be a reductive flat
connection on $E_G$. So $(E_G,\, \nabla)$ gives a point
\begin{equation}\label{e2}
(E_G,\, \nabla)\, :=\, z\, \in\, \X_{n,g}(G)\, ,
\end{equation}
where $\X_{n,g}(G)$ is constructed in \eqref{e1}.

As before, the Lie algebra of $G$ will be denoted by $\mathfrak g$. The adjoint bundle
$$
\text{ad}(E_G)\, =\, E_G\times^G\mathfrak g
$$
is the vector bundle over $\Sigma_{n,g}$ associated to $E_G$ for the adjoint action of $G$ on $\mathfrak g$. So
the fibers of $\text{ad}(E_G)$ are Lie algebras identified with $\mathfrak g$ uniquely up to automorphisms
of $\mathfrak g$ given by conjugations. Fix a nondegenerate $G$--invariant symmetric bilinear form
\begin{equation}\label{B}
B\, \in\, \text{Sym}^2({\mathfrak g}^*)^G
\end{equation}
on $\mathfrak g$; the assumptions on $G$ ensure that such a form $B$ exists. Since $B$ in \eqref{B} is
$G$--invariant, it produces a $C^\infty$ pairing
\begin{equation}\label{B1}
{\mathcal B}\, :\, \text{ad}(E_G)\otimes \text{ad}(E_G)\, \longrightarrow\, \Sigma_{n,g}\times k\, ,
\end{equation}
where $k\,=\, \mathbb R$ (respectively, $k\,=\, \mathbb C$) when $G$ is compact (respectively,
complex reductive). Since $B$ in \eqref{B} is also nondegenerate, the pairing ${\mathcal B}$ in
\eqref{B1} is fiberwise nondegenerate. Therefore, ${\mathcal B}$ produces an isomorphism of vector bundles
\begin{equation}\label{e3}
\text{ad}(E_G)\,\stackrel{\sim}{\longrightarrow}\,\text{ad}(E_G)^*\, .
\end{equation}

The flat connection $\nabla$ on $E_G$ induces a flat connection on $\text{ad}(E_G)$; this induced
connection on $\text{ad}(E_G)$ will be denoted by $\nabla^{\rm ad}$.
We note that $\nabla^{\rm ad}$ is self-dual with respect to the isomorphism in
\eqref{e3}; this means that the connection on $\text{ad}(E_G)^*$ given by $\nabla^{\rm ad}$ using the
isomorphism in \eqref{e3} coincides with the connection on $\text{ad}(E_G)^*$ induced by
$\nabla^{\rm ad}$ using the duality pairing.

Let $\underline{\text{ad}}(E_G)$
be the local system on $\Sigma_{n,g}$ given by the sheaf of flat sections of $\text{ad}(E_G)$ for the connection
$\nabla^{\rm ad}$. We have
\begin{equation}\label{e4a}
T_z\X_{n,g}(G)\,=\, H^1(\Sigma_{n,g},\, \underline{\text{ad}}(E_G))\, ,
\end{equation}
where $z$ is the point in \eqref{e2} (for example \cite[Section 1.8]{Go}). Therefore, the Poincare--Verdier duality gives that
\begin{equation}\label{e4b}
T^*_z\X_{n,g}(G)\,=\, H^1(\Sigma_{n,g},\, \underline{\text{ad}}(E_G))^*\,=\,
H^1_c(\Sigma_{n,g},\, \underline{\text{ad}}(E_G))\, ,
\end{equation}
where $H^i_c$ denotes the compactly supported $i$-th cohomology. We have a natural homomorphism
$$
T^*_z\X_{n,g}(G)\,=\, H^1_c(\Sigma_{n,g},\, \underline{\text{ad}}(E_G))\, \longrightarrow\,
H^1(\Sigma_{n,g},\, \underline{\text{ad}}(E_G))\,=\,T_z\X_{n,g}(G)\, .
$$
As $z$ moves over $\X_{n,g}(G)$, these point-wise homomorphisms together produce a $C^\infty$
homomorphism
\begin{equation}\label{e4}
\Theta\, :\, T^*\X_{n,g}(G) \, \longrightarrow\,T\X_{n,g}(G)\, .
\end{equation}

A Poisson structure on a $C^\infty$ manifold $A$ is a section
$\theta_A\, \in\, C^\infty(A,\, \bigwedge^2 TA)$ 
such that the Schouten--Nijenhuis bracket
$[\theta_A,\, \theta_A]$ vanishes identically \cite{Ar}. The condition that
$[\theta_A,\, \theta_A]\,=\, 0$ is equivalent to the following condition:
given a pair of $C^\infty$ locally defined functions $f_1$ and $f_2$ on $A$, consider the
locally defined $C^\infty$ function
$$
\{f_1,\, f_2\}_{\theta_A}\,:=\, \theta_A((df_1)\wedge (df_2))\, ;
$$
then $$\{f_1,\, \{f_2,\, f_3\}_{\theta_A}\}_{\theta_A}+ \{f_2,\, \{f_3,\, f_1\}_{\theta_A}\}_{\theta_A} +
\{f_3,\, \{f_1,\, f_2\}_{\theta_A}\}_{\theta_A}\,=\, 0$$
for all $C^\infty$ locally defined functions $f_1,\, f_2,\, f_3$.

The $C^\infty$ homomorphism $\Theta$ in \eqref{e4} is a Poisson structure on $\X_{n,g}(G)$
(see \cite{BJ}).

When $\Sigma_{n,g}$ is a compact oriented surface ($n\,=\,0$), then $\Theta$ is an isomorphism, so the Poisson structure given by it
is actually a symplectic structure. In that case, it coincides with the symplectic structure constructed
by Atiyah--Bott \cite{AB} and Goldman \cite{Go}.

\section{Open subsets and Poisson structure}\label{se3}

Let $\Sigma_1\,= \,\Sigma_{n_1, g_1}$ be embedded as a connected open subset of $\Sigma_2\,= \,\Sigma_{n_2, g_2}$.  Take the base point $x_0\, \in\, \Sigma_2$ such that $x_0\,\in\, \Sigma_1$.

Restricting the flat $G$--connections on $\Sigma_2$ to the open subset $\Sigma_1$ we obtain a map
\begin{equation}\label{e5}
\Phi\, :\, \X_{\Sigma_2}(G) \,\longrightarrow\, \X_{\Sigma_1}(G) \, .
\end{equation}
Indeed, as above, the natural homomorphism $\pi_1(\Sigma_1,\, x_0)\, \longrightarrow\,
\pi_1(\Sigma_2,\, x_0)$ produces a map
$$
\text{Hom}(\pi_1(\Sigma_2,\, x_0),\, G)\, \longrightarrow\,\text{Hom}(\pi_1(\Sigma_1,\, x_0),\, G)\, ;
$$
it in turn gives the map $\Phi$ in \eqref{e5} by taking geometric invariant theoretic quotient for the actions of
$G$ on $\text{Hom}(\pi_1(\Sigma_2,\, x_0),\, G)$ and $\text{Hom}(\pi_1(\Sigma_1,\, x_0),\, G)$.

As in \eqref{e2}, take any $(E_G,\, \nabla)\, :=\, z\, \in\, \X_{\Sigma_2}(G)$.
Consider the local system $\underline{\text{ad}}(E_G)$ on $\Sigma_2$ (see \eqref{e4a}, \eqref{e4b}).
Its restriction to $\Sigma_1$ will be denoted by $\underline{\text{ad}}(E_G)_{\Sigma_1}$.
The inclusion map $\Sigma_1\, \hookrightarrow\, \Sigma_2$ produces homomorphisms
\begin{equation}\label{e6}
\beta\, :\,H^1(\Sigma_2,\, \underline{\text{ad}}(E_G))\, \longrightarrow\, 
H^1(\Sigma_1,\, \underline{\text{ad}}(E_G)_{\Sigma_1})\, ,
\end{equation}
and
\begin{equation}\label{e7}
\gamma\, :\,H^1_c(\Sigma_1,\, \underline{\text{ad}}(E_G)_{\Sigma_1})\, \longrightarrow\, 
H^1_c(\Sigma_2,\, \underline{\text{ad}}(E_G))
\end{equation}
(see \eqref{e4b}). We note that $\beta$ is the pullback by the inclusion map of $\Sigma_1$ in $\Sigma_2$, while
$\gamma$ is the push-forward by the inclusion map.

For the map $\Phi$ in \eqref{e5}, let
$$
d\Phi (z)\, :\, T_z\X_{\Sigma_2}(G)\,\longrightarrow\, T_{\Phi (z)} \X_{\Sigma_1}(G)
$$
be its differential at the point $z$ in \eqref{e2}. Let
$$
(d\Phi)^* (z)\, :\, T^*_{\Phi (z)} \X_{\Sigma_1}(G)\,\longrightarrow\, T^*_z \X_{\Sigma_2}(G)
$$
be its dual homomorphism.

\begin{proposition}\label{prop1}
Invoke the isomorphism in \eqref{e4a} $($respectively, \eqref{e4b}$)$ for $T_z \X_{\Sigma_2}(G)$ and
$T_{\Phi (z)} \X_{\Sigma_1}(G)$ $($respectively, $T^*_z \X_{\Sigma_2}(G)$ and 
$T^*_{\Phi (z)}\X_{\Sigma_1}(G)$ $)$.
\begin{enumerate}
\item The homomorphism $d\Phi (z)$ coincides with the homomorphism $\beta$ in \eqref{e6}.

\item The homomorphism $(d\Phi)^*(z)$ coincides with the homomorphism $\gamma$ in \eqref{e7}.
\end{enumerate}
\end{proposition}

\begin{proof}
The first statement is standard.

To prove the second statement, take any $\text{ad}(E_G)$--valued one-form
$$\omega\, \in\, C^\infty(\Sigma_2,\, \text{ad}(E_G)\otimes T^*\Sigma_2)$$ such that
$\nabla^{\rm ad}(\omega)\,=\, 0$, where $\nabla^{\rm ad}$, as before, is the connection on
$\text{ad}(E_G)$ induced by the connection $\nabla$ (see \eqref{e2}). Also take a
compactly supported $\text{ad}(E_G)$--valued one-form
$\omega^c\, \in\, C^\infty_c(\Sigma_1,\, (\text{ad}(E_G)\vert_{\Sigma_1})\otimes T^*\Sigma_1)$ on $\Sigma_1$ such that
$\nabla^{\rm ad}(\omega^c)\,=\, 0$. Let $\iota\, :\, \Sigma_1\, \hookrightarrow\, \Sigma_2$ be the inclusion map.
Then we have
\begin{equation}\label{e8}
\int_{\Sigma_1} {\mathcal B}(\omega^c\wedge (\iota^*\omega))\,=\,\int_{\Sigma_2}{\mathcal B}((\iota_*\omega^c)\wedge\omega)\, ,
\end{equation}
where $\mathcal B$ is the pairing in \eqref{B1}, and $\iota_*\omega^c$ is the
push-forward of the compactly supported form $\omega^c$ using the inclusion
map $\iota$; note that both ${\mathcal B}(\omega^c\wedge (\iota^*\omega))$
and ${\mathcal B}((\iota_*\omega^c)\wedge\omega)$ are compactly supported $2$-forms on $\Sigma_1$, and moreover
they coincide. The second statement in the proposition follows from \eqref{e8}.
\end{proof}

A smooth map $F\, :\, A\, \longrightarrow\, B$ between Poisson manifolds $(A,\, \theta_A)$ and
$(B,\, \theta_B)$ is called \textit{Poisson} if
$$F\circ \{f,\, g\}_{\theta_B}\,=\,\{F\circ f,\, F\circ g\}_{\theta_A}$$
for all locally defined $C^\infty$ functions $f$ and $g$ on $B$.  This is equivalent to our usage in Corollary \ref{cor:poisson}.  We note that $F$ is Poisson if and only if the following diagram
is commutative
\begin{equation}\label{F}
\begin{matrix}
T^*_{F(x)} B & \xrightarrow{(dF)^*(x)=dF(x)^*} & T^*_x A\\
\,\,\,\,\,\,\,\,\,\,\,\,\,\,\,\,\,\,\,\,\Big\downarrow \theta_B(F(x)) &&
\,\,\,\,\,\,\,\,\,\,\,\,\,\,\Big\downarrow \theta_A(x)\\
T_{F(x)} B & \xleftarrow{dF(x)} & T_x A
\end{matrix}
\end{equation}
for every point $x\, \in\, A$, where $dF\, :\, TA\, \longrightarrow\, F^*TB$ is the differential
of the map $F$ while $(dF)^*$ is its dual. 

\begin{theorem}\label{thm1}
The map $\Phi$ in \eqref{e5} is Poisson.
\end{theorem}

\begin{proof}
In view of Proposition \ref{prop1} it is straight-forward to check that the diagram in
\eqref{F} for $F\,=\, \Phi$ commutes. To explain this, take any
$$(E_G,\, \nabla)\, :=\, z\, \in\, \X_{\Sigma_1}(G)$$ as in \eqref{e2}. Take any compactly
supported form
$$\omega\, \in\, C^\infty_c(\Sigma_1,\, (\text{ad}(E_G)\vert_{\Sigma_1})\otimes T^*\Sigma_1)$$ such that
$\nabla^{\rm ad}(\omega)\,=\, 0$, where $\nabla^{\rm ad}$ is the connection on
$\text{ad}(E_G)$ induced by the connection $\nabla$ (see \eqref{e2}). Let $\iota_*\omega\, \in\,
C^\infty_c(\Sigma_2,\, \text{ad}(E_G)\otimes T^*\Sigma_2)$ be the push-forward of $\omega$ using the inclusion map
$\iota\, :\, \Sigma_1\, \hookrightarrow\, \Sigma_2$. Now consider $\iota_*\omega$ as an element of
$C^\infty(\Sigma_2,\, \text{ad}(E_G)\otimes T^*\Sigma_2)$;
finally, restrict this element of $C^\infty(\Sigma_2,\, \text{ad}(E_G)\otimes T^*\Sigma_2)$ to $\Sigma_1$. This
restriction is evidently $\omega$ itself. Hence the diagram in \eqref{F} commutes for $F\,=\, \Phi$.
\end{proof}

The theorem was stated for one surface embedded in another; it holds more generally for suitable ramified covers.
Let $\overline{\Sigma_1}$ and $\overline{\Sigma_2}$ be compact connected oriented $C^\infty$ surfaces and
$$
\varphi_1\, :\, \overline{\Sigma_1}\, \longrightarrow\, \overline{\Sigma_2}
$$
a possibly ramified covering map which is oriented. Let $$S_2\, \subset\,\overline{\Sigma_2}
\ \ \text{ and }\ \ S_1\, \subset\,\overline{\Sigma_1}$$ be finite subsets such that $\varphi^{-1}_1(S_2)
\, \subset\, S_1$. Define
$$
\Sigma_2\, :=\, \overline{\Sigma_2}\setminus S_2 \ \ \text{ and }\ \ {\Sigma_1} \, :=\, \overline{\Sigma_1}\setminus S_1\, ;
$$
let
$$
\varphi\, :=\, \varphi_1\vert_{\Sigma_1}\, :\, \Sigma_1\, \longrightarrow\, \Sigma_2
$$
be the restriction of $\varphi_1$ to the open subset $\Sigma_1$.

Consider the character varieties $\X_{\Sigma_1}(G)$ and $\X_{\Sigma_2}(G)$. Let
\begin{equation}\label{Psi}
\Psi\, :\, \X_{\Sigma_2}(G) \, \longrightarrow\, \X_{\Sigma_1}(G)
\end{equation}
be the map that sends any flat principal $G$--bundle $(E_G,\, \nabla)$ on $\Sigma_2$ to the
flat principal $G$--bundle $(\varphi^*E_G,\, \varphi^*\nabla)$ on $\Sigma_1$. This map $\Psi$
coincides with map of character varieties given by the homomorphism
$$
\varphi_*\, :\, \pi_1(\Sigma_1,\, y_0)\, \longrightarrow\, \pi_1(\Sigma_2,\, \varphi(y_0))
$$
induced by $\varphi$.

\begin{proposition}\label{prop2}
The map $\Psi$ in \eqref{Psi} is Poisson.
\end{proposition}

\begin{proof}
Take any
$$(E_G,\, \nabla)\, :=\, z\, \in\, \X_{\Sigma_2}(G)\, .$$
Let $\nabla^{\rm ad}$ be the flat connection on $\text{ad}(E_G)$ induced by $\nabla$. Let
$\underline{\text{ad}}(E_G)$ be the local system on $\Sigma_2$ given by the sheaf of flat sections of
$\text{ad}(E_G)$ for the connection $\nabla^{\rm ad}$. From \eqref{e4a} we know that
$T_z\X_{\Sigma_2}(G)\,=\, H^1(\Sigma_2,\, \underline{\text{ad}}(E_G))$ and
$$
T_{\Psi(z)}\X_{\Sigma_1}(G)\,=\, H^1(\Sigma_1,\, \underline{\text{ad}}(\varphi^*E_G))
\,=\, H^1(\Sigma_1,\, \varphi^*\underline{\text{ad}}(E_G))\, ,
$$
where $\underline{\text{ad}}(\varphi^*E_G)$ is the local system on $\Sigma_1$ given by the sheaf of flat sections of
$\text{ad}(\varphi^* E_G)$ for the flat connection on $\text{ad}(\varphi^* E_G)$
induced by the flat connection $\varphi^*\nabla$ on $\varphi^* E_G$. Note that this induced flat connection
on $\text{ad}(\varphi^* E_G)$ coincides with the flat connection $\varphi^*\nabla^{\rm ad}$ on
$\text{ad}(\varphi^* E_G)\,=\, \varphi^*\text{ad}(E_G)$.

The differential $d\Psi$ at $z$
$$
d\Psi (z)\, :\, 
H^1(\Sigma_2,\, \underline{\text{ad}}(E_G))\,=\, T_z\X_{\Sigma_2}(G)\, \longrightarrow\,
T_{\Psi(z)}\X_{\Sigma_1}(G)\,=\, H^1(\Sigma_1,\, \varphi^*\underline{\text{ad}}(E_G))
$$
coincides with the homomorphism $H^1(\Sigma_2,\, \underline{\text{ad}}(E_G))\, \longrightarrow\,
H^1(\Sigma_1,\, \varphi^*\underline{\text{ad}}(E_G))$ that sends any cohomology class $\mu$ to
its pullback $\varphi^*\mu$.

The dual homomorphism
$$
(d\Psi)^* (z)\, :\, T^*_{\Psi(z)}\X_{\Sigma_1}(G)\,=\,
H^1_c(\Sigma_1,\, \varphi^*\underline{\text{ad}}(E_G))\, \longrightarrow\,
H^1_c(\Sigma_2,\, \underline{\text{ad}}(E_G))\,=\, T^*_z\X_{\Sigma_2}(G)
$$
(see \eqref{e4b}) coincides with the trace map. To explain the trace map, take any
compactly supported $\varphi^*\text{ad}(E_G)$--valued $1$--form
\begin{equation}\label{fm}
\omega\, \in\, C^\infty_c(\Sigma_1,\, \varphi^*\text{ad}(E_G)\otimes T^*\Sigma_1)\,=\,
C^\infty_c(\Sigma_1,\, \text{ad}(\varphi^* E_G)\otimes T^*\Sigma_1)
\end{equation}
such that $\varphi^*\nabla^{\rm ad}(\omega)\,=\, 0$. Now construct
$$
\widehat{\omega}\, \in\, C^\infty(\Sigma_2,\, \text{ad}(E_G)\otimes T^*\Sigma_2)
$$
as follows: For any $x\, \in\, \Sigma_2$ such that $\varphi$ is unramified over $x$,
take any $v\, \in\, T_x \Sigma_2$. Define
$$
\widehat{\omega}(x)(v) \, =\, \sum_{y\in \varphi^{-1}(x)} \omega (y)((d\varphi (y))^{-1}(v))
\, \in\, \text{ad}(E_G)_x\, ,
$$
where $d\varphi (y)$ is the differential of the map $\varphi$ at $y$. It is straight-forward
to check the following:
\begin{itemize}
\item The above form $\widehat{\omega}$ extends to entire $\Sigma_2$ as a $C^\infty$ section of
$\text{ad}(E_G)\otimes T^*\Sigma_2$; this section of $\text{ad}(E_G)\otimes T^*\Sigma_2$ over $\Sigma_2$ will
be denoted by $\widehat{\omega}$.

\item The form $\widehat{\omega}$ on $\Sigma_2$ is compactly supported. In fact, the support of
$\widehat{\omega}$ is contained in the image of the support of $\omega$ under the map $\varphi$.
This implies that $\widehat{\omega}$ is compactly supported.

\item $\nabla^{\rm ad}(\widehat{\omega})\,=\, 0$.
\end{itemize}

The trace map $$H^1_c(\Sigma_1,\, \varphi^*\underline{\text{ad}}(E_G))\, \longrightarrow\,
H^1_c(\Sigma_2,\, \underline{\text{ad}}(E_G))$$ mentioned earlier is constructed by sending any
$\omega$ as in \eqref{fm} to $\widehat{\omega}$ constructed above from it.

{}From the above construction of $\widehat{\omega}$ we conclude that $\widehat{\omega}$ has the following
property: For any $\omega'\, \in\, C^\infty(\Sigma_2,\, \text{ad}(E_G)\otimes T^*\Sigma_2)$,
\begin{equation}\label{e9}
\int_{\Sigma_2}{\mathcal B}(\widehat{\omega}\wedge \omega')\,=\, \int_{\Sigma_1}{\mathcal B}(\omega\wedge (\varphi^*\omega'))\, ,
\end{equation}
where ${\mathcal B}$ is the pairing in \eqref{B1}.

{}Using \eqref{e9} it is straightforward to check that
the diagram in \eqref{F} commutes for $F\,=\, \Psi$. Hence the map $\Psi$ is Poisson.
\end{proof}

\section{Capping: Symplectic and Poisson extensions of Poisson character varieties}\label{capping}

For a given $\Sigma_1\,=\, \Sigma_{n_1, g_1}$ with a non-empty boundary, we consider different $\Sigma_2$ 
we can obtain by gluing onto the boundary components of $\Sigma_1$. This yields Poisson maps of the character varieties:
\begin{equation}
\Phi\, :\, \X_{\Sigma_2}(G) \,\longrightarrow\, \X_{\Sigma_1}(G) \, ,
\end{equation}
and we will examine the images and fibres of these maps. When $\Sigma_2$ is closed, this gives us in some sense symplectic completions of the character varieties.

Number the $n_1$ boundary components of $\Sigma_1\,=\,\Sigma_{n_1, g_1}$.
Choose a base point $x_1$ on the first boundary component $\gamma_1$ of $\Sigma_1$.
For every $2\, \leq\, i\, \leq\, n_1$, fix a path $p_i$ from $x_1$ to base point $x_i$ on the $i$-th boundary
component $\gamma_i$ of $\Sigma_1$. With 
parametrizations $\mu_i$ respecting the orientations of the $\gamma_i$ and the right choice of paths, we get some 
standard generators $c_1 \,=\, \mu_1, c_i \,=\, p_i\mu_ip_i^{-1}$ of the fundamental group $\pi_1(\Sigma_1,\, x_1)$.

\subsection{Case 1: Capping with disks.} We consider first the case when $\Sigma_2$ is obtained from $\Sigma_1\,=\, 
\Sigma_{n_1, g_1}$ by gluing in $k\,\leq\, n_1$ disks, identifying the boundaries $\widetilde{\gamma}_i$ of disks $D_i$ 
with $\gamma_i$, so that $\Sigma_2 $ has $k$ of the holes of $\Sigma_1$ filled in. In this case, it is 
straightforward to see that $\Phi$ is an injection, with image the representations $\rho$ with $\rho(c_i)\, =\, 1.$ 
This is a union of a family of symplectic leaves, and if $k\,=\, n_1$, a single symplectic leaf.

\subsection{Case 2: Capping with a cylinder.} Now consider the situation where one glues in a cylinder, attaching 
the boundary circles $\widetilde{\gamma}_1$,\, $\widetilde{\gamma}_2$ of a cylinder to boundary circles 
$\gamma_1$,\, $\gamma_2$ of $\Sigma_1$. The resulting $\Sigma_2$ has two less boundary components and genus one 
more.

One can again ask what the image and fibres of $\Phi$ are in this case. For a flat connection on $\Sigma_1$, 
choose a flat trivialization along the path $p_2$. We then, for the circles $\gamma_i$, have holonomies $C_i$. 
For these to lie in the image of the flat connections on $\Sigma_2$, we need to be able to glue a flat connection 
on $\Sigma_1$ to a flat connection on the cylinder.

On the cylinder, choose base points on the circles $\widetilde{\gamma}_1$, $\widetilde{\gamma}_2$, and a path
$\beta$ from $\widetilde{\gamma}_1$ to $\widetilde{\gamma}_2$. On the cylinder, we have the relation in the
fundamental groupoid:
$$\widetilde{\gamma}_1\beta^{-1}\widetilde{\gamma}_2^{-1}\beta\,= \,1.$$
Once one trivializes at the two base points, a flat connection determines corresponding holonomies $\widetilde{C}_i,\,
B$, satisfying $$\widetilde{C}_1 B^{-1} \widetilde{C}_2^{-1}B\,=\,1.$$

The gluing, on the level of connections matches $C_i$ with $\widetilde{C}_i$. We then must have: $$ C_1 B^{-1} 
C_2^{-1} B\,=\,1$$ for some matrix $B$, which is the image of an extra cycle created by the gluing. Thus the image 
in $\X_{\Sigma_1}(G)$ of $\X_{\Sigma_2}(G)$, is the union of symplectic leaves for which the conjugacy class along 
$\gamma_1$ is the inverse of the conjugacy class along $\gamma_2$, while the fibre is isomorphic to the stabilizer 
of $C_2$ under conjugation. Note that $\X_{\Sigma_1}(G)$ and $\X_{\Sigma_2}(G)$ have the same dimension.

\subsection{Case 3: Capping with a $k$-holed sphere.} We now consider the case of gluing a $k$-holed 
sphere ($k\, \geq\, 3$) with boundary circles $\widetilde{\gamma}_i$, $i\, =\, 1,\,\dots,\,k$, to the boundary circles
$\gamma_i$, $i\, =\, 1,\,\dots,\,k$, of 
$\Sigma_1$. The resulting $\Sigma_2$ will have $k$ less boundary components, and genus $k-1$ more.

On the sphere, choose base points on the circles $\widetilde \gamma_i$, $i \,= \,1,\,\dots,\,k$,
and paths $\beta_i$,\, $i \,= \,1,\,\dots,\,k$, from $\widetilde{\gamma}_1$ to $\widetilde{\gamma}_i$. We
have the relation in the fundamental group:
$$\widetilde{\gamma}_1^{-1}\beta_2^{-1}\widetilde{\gamma}_2^{-1}\beta_2\cdots
\beta_i^{-1}\widetilde{\gamma}_i^{-1}\beta_i\cdots \beta_k^{-1}
\widetilde{\gamma}_k^{-1}\beta_k\,=\, 1.$$
A flat connection on the punctured sphere, once one trivializes at the base points, gives corresponding
holonomies $\widetilde{C}_i,\, B_i$, satisfying:
$$\widetilde{C}_1^{-1} B_2^{-1} \widetilde{C}_2^{-1}B_2\cdots
B_i^{-1} \widetilde{C}_i^{-1}B_i\cdots
B_k^{-1} \widetilde{C}_k ^{-1}B_k\,=\,1.$$
The gluing, on the level of connections matches $C_i$ with $\widetilde{C}_i^{-1}$, and so 
\begin{equation}
C_1 B_2^{-1}C_2B_2\dots
B_i^{-1}C_iB_i\dots
B_k^{-1}C_k B_k\,=\,1.
\label{conjclass}\end{equation}
So one needs to be able to find matrices $B_i$ which make this relation true. In short, given $k$ conjugacy
classes, we have to be able to find elements in them whose product is one. There are choices for which this is not the case; for example, if one takes $C_1\,=\,\cdots \,=\,C_{k-1}\, = \,1,\, C_k \,=\, -1$.

The general question of when Equation \eqref{conjclass} has a solution is known as the Deligne-Simpson Problem, 
which is only solved when $G$ is of type $A_n$ \cite{Simp, Ko, Cr}.

The map $\Phi\, :\, \X_{\Sigma_2}(G) \,\longrightarrow\, \X_{\Sigma_1}(G)$ is not surjective.

If $g$ is at least two, $\X_{\Sigma_2}(G)$ has dimension $(k-2)\dim (G)$ greater than $\X_{\Sigma_1}(G)$. This 
would then be the dimension of the generic fibres when there is a solution to Equation \eqref{conjclass}.

\subsection{Case 4: Capping one circle with a genus $1$ curve, $n_1\geq 2$}

Now let us glue a punctured genus $1$ curve with one boundary circle $\widetilde \gamma$ to the boundary circle 
$\gamma_1$ of $\Sigma_1$, when $n_1\,\geq\, 2$. The resulting $\Sigma_2$ will have $1$ less boundary components, and 
genus $1$ more.

On the (closed) genus $1$ curve, take a standard basis $\alpha,\, \beta$ of the fundamental group. In the 
fundamental group of the punctured curve: $$\widetilde \gamma \alpha\beta \alpha^{-1} \beta^{-1} \,=\, 1.$$

\begin{lemma}\label{lem1}
Take $G$ as before.
Every element of the group $[G,\, G]$ can be expressed as $$[A,\, B]\, :=\, ABA^{-1} B^{-1}$$ for some
$A,\, B\, \in\, [G,\, G]$. Thus elements of the commutator subgroup can be written as a single commutator.
\end{lemma}

\begin{proof}
First assume that $G$ is compact. Then $[G,\, G]$ is a connected compact semisimple Lie group. A theorem
of Got\^o says that for every element $C\, \in\, [G,\, G]$, there are elements $A,\, B\, \in\, [G,\, G]$
such that $C\,=\, [A,\, B]\, :=\, ABA^{-1}B^{-1}$ \cite[p.~270, Lemma]{Got}.

Next assume that $G$ is a reductive affine algebraic group defined over $\mathbb C$. Then
$[G,\, G]$ is a connected semisimple affine algebraic group defined over $\mathbb C$.
Then for every element $C\, \in\, [G,\, G]$, there are elements $A,\, B\, \in\, [G,\, G]$
such that $C\,=\, [A,\, B]\, :=\, ABA^{-1}B^{-1}$ \cite[p.~908]{PW}, \cite[p.~457]{Re}.
\end{proof}

We then have:

\begin{proposition}
The image in $\X_{\Sigma_1}(G)$ of $\X_{\Sigma_2}(G)$ consists of the representations with the image of
$\gamma_1$ in $[G,\,G]$. In particular, if $G$ is semisimple, meaning $G\,=\, [G,\,G]$, then the map
$\X_{\Sigma_2}(G)\, \longrightarrow\, \X_{\Sigma_1}(G)$ is surjective.
\end{proposition}

If $G\, =\, [G,\,G]$, then by capping successively the boundaries of $\Sigma_1$ we get a symplectic manifold 
$\X_{\Sigma_2}(G)$ and a Poisson map $\X_{\Sigma_2}(G)\,\longrightarrow\, \X_{\Sigma_1}(G)$. If $g$ is at least two, each 
cap increases the genus by one and diminishes the number of punctures by one, and so adds $\dim(G)$ dimensions.

\subsection{Case 5: Capping with an $n$-punctured genus $1$ curve.}

We now glue a genus $1$ curve $C$ with $n$ boundary circles to the $n$ boundary circles of $\Sigma_1$ so that the resulting curve 
$\Sigma_2$ has no punctures, but the genus is increased by $n$. Again, for $G$ semisimple, $g_1\geq 2$, this increases the dimensions of the representation space by $n\dim{G}$.

We have the fundamental groups
$$
\pi_1(\Sigma_1)\,=\, \langle a_1,\, b_1,\, \dots,\, a_{g_1},\, b_{g_1},\, c_1,\, \dots,\, c_n\ \Big\vert\ \prod_{i=1}^{g_1} [a_i,\, b_i]\prod_{j=1}^n c_j\rangle $$
where $g_1$ is the genus of $\Sigma_1$, and 
$$
\pi_1(C)\,=\, \langle a,\, b,\, c_1,\, \dots,\, c_n\ \Big\vert\ [a,\, b]\prod_{j=1}^n c_j\rangle.$$
We then have similar relations for their representations into $G$.
 
\begin{proposition}
The map $\Phi\,:\, \X_{\Sigma_2}(G)\,\longrightarrow\,\X_{\Sigma_1}(G)$ is surjective and Poisson.
\end{proposition}

\begin{proof}
For the surjectivity, for a representation of $\pi_1(\Sigma_1)$ into $G$, we have images $A_i,\, B_i,\, C_j$ of the 
generators $a_i,\, b_i,\, c_j$ giving an element $ \prod_{i=1}^{g_0} [A_i,\, B_i]$ of the commutator subgroup of $G$. 
Now Lemma \ref{lem1} tells us that this is equal to a single commutator $[A,\,B]$, giving a representation $A,\,B,
\,C_j$ of $\pi_1(C)$. Inverting, this can be glued to the representation of $\pi_1(\Sigma_1)$ to obtain a
representation of $\pi_1(\Sigma_2)$. Consequently, the map $\Phi\,:\, \X_{\Sigma_2}(G)
\,\longrightarrow\,\X_{\Sigma_1}(G)$ is surjective. We have already seen that $\Phi$ is Poisson.
\end{proof}

Thus we have a symplectic ``completion'' of $\X_{\Sigma_1}(G)$.

\subsection{Case 6: Capping with a mirror image.}

We note that we can also glue a copy of $\Sigma_1$ to itself. Again, the resulting $\Sigma_2$ will give a 
symplectic character variety mapping surjectively onto that of $\Sigma_1$, albeit with an enormous redundancy; the 
dimension gets doubled. We note that for the representations on one of the copies of $\Sigma_2$, we must invert 
the matrices $A_i,\, B_i,\, C_j$ before gluing. This involves changing the generating set of the fundamental group 
somewhat, but it can be done.

\section{Gluing via symplectic quotients}\label{symquot}

The symplectic extensions obtained by capping boundary components are, as we have seen, often somewhat inefficient in terms of the dimensions they add. We close by recalling two classical constructions which accomplish the same task, essentially by adding a trivialisation to our connections on a base point on each boundary circle. The first, due to Alekseev, Malkin and Meinrenken \cite{AMM}, takes us outside of the symplectic domain, into quasi-Hamiltonian territory; the second involves the extended moduli spaces of Jeffrey \cite{J}. The reduction to the flat connection spaces are then group quotients. We note that there is an ``imploded'' version of the construction of \cite{AMM}, considered in \cite{HJ}.

\subsection{q-Hamiltonian spaces}

$q$-Hamiltonian spaces were defined by Alekseev, Malkin and Meinrenken in \cite{AMM}. They play a role analogous to the symplectic quotient of the space of all connections
by the based gauge group (in other words the extended moduli space \cite{J}).

We quote the following definitions from \cite{AMM}.

\begin{definition}[{\cite[Definition 2.2]{AMM}}]
Let $G$ be a compact Lie group. Assume further that $G$ is connected and simply 
connected. A quasi-Hamiltonian (or $q$-Hamiltonian) $G$-space is a $G$-manifold 
$M$ together with an invariant 2-form $\omega \,\in\, \Omega^2(M)^G$ and a $G$-equivariant map $\mu\,\in\,C^\infty(M,\,G)$ (where $G$ acts on itself by conjugation) such that:
\begin{enumerate}
\item[(B1)] The differential of $\omega$ is given by
$$ d \omega \,=\, - \mu^* \chi,$$ where $\chi$ is the closed bi-invariant 3-form on $G$
given by $\chi\,  =\, \frac{1}{12}(\theta,\, [\theta,\, \theta]).$
Here  $\theta \in \Omega^1(G) \otimes \lieg$  is the left invariant Maurer-Cartan form (where $\lieg$ is the Lie algebra of $G$) and
$\overline{\theta}$ is the right invariant Maurer-Cartan form. 
 This is often denoted $\theta = g^{-1} dg$ if $g: U \to G $ is a coordinate on a coordinate chart $U$ for  $G$.
 Similarly $\bar{\theta}$ is denoted $dg g^{-1}$.

\item[(B2)] The map $\mu$ satisfies:
 $$i(v_\xi) \omega \,= \, \frac{1}{2} \mu^* (\theta + \overline{\theta},\, \xi),$$ where $v_\xi$ is the fundamental vector field on $G$ associated to an element $\xi$ in $\lieg$.

We have
$$i(v_\xi) \omega \,= \, \frac{1}{2} \mu^* (\theta + \overline{\theta},\, \xi).$$

\item[(B3)] At each $x\,\in\,M$, the kernel of $\omega_x$ is given by:
$${\rm ker} ( \omega_x) \,=\, {v_\xi,\,\ \xi \,\in\, {\rm Ker}
 ({\rm Ad}_{\mu(x)} +1)}. $$
 \end{enumerate}
\end{definition}

We will refer to $ \mu$ as a moment map.

The $q$-Hamiltonian space whose $q$-Hamiltonian quotient is the 
character variety for a 2-manifold of genus $g$ with no
boundary 
is the space denoted $M_g$, which is isomorphic to $G^{2g}$ .
The moment map for the diagonal action of 
$G$ on $M_g$ by conjugation is the product of commutators
\begin{equation} \label{moment1}\mu(x_1,\, x_2,\, \dots,\, x_{2g})\,=\, \prod_{j = 1}^g x_{2j-1} x_{2j} x_{2j-1}^{-1}
x_{2j}^{-1}. \end{equation}
Here $x_j \in G$.

The $q$-Hamiltonian space whose
 $q$-Hamiltonian quotient is the character variety for a $2$-manifold which has genus $g$ and $r$ boundary components is
the space
\begin{equation} \label{mgrdef} M_{g,r}:= G^{2g+2r}. \end{equation}  See Section \ref{s:fusion} below.

If $X_1$ and $X_2$ are two $q$-Hamiltonian $G$-spaces with moment maps $\mu_1$ and $\mu_2$,
then the fusion product $X_1 \times X_2$ is also a $q$-Hamiltonian $G$-space with moment map $\mu_1 \cdot \mu_2$, where
$\cdot$ is multiplication in $G$. See for example \cite{Mein}, \S3.2.
This is analogous to the fact that the product of two Hamiltonian $G$-spaces $Y_1$ and $Y_1$ is also a Hamiltonian $G$-space and its moment map is the sum of the 
moment maps for the Hamiltonian $G$ actions on $Y_1$ and $Y_2$.

It follows that if $M_g$ is the $q$-Hamiltonian space
associated to a genus $g$ surface $\Sigma_g$ 
and $M_h$ is the $q$-Hamiltonian space associated with a 
genus $h$ surface $\Sigma_h$, then 
the  $q$-Hamiltonian quotient of $M_g \times M_h$ by the diagonal 
action of $G$ is a symplectic manifold. It is the character
variety of a surface of genus $g+h$ without boundary.

This quotient construction is defined in Section \ref{s:fusion} below.  For completeness, we include the following material from Section 6 of \cite{AMM}. 
The 2-form on the double $D(G):=G\times G$ is given in \cite[Section 3.2]{AMM}:
$$ \omega_D \,=\, \frac{1}{2} (a^* \theta, \,
b^* \overline{\theta}) + \frac{1}{2} (a^* \overline{\theta},\, b^* {\theta}), $$
where $a,\, b\,: \, G\times G \,\longrightarrow\,G$ are  projections to the first and second factors respectively.
Here $\theta, \bar{\theta}$  were introduced in Definition 6.1.

Now let us consider the construction for a surface of genus $g$ and 
$r+1$ boundary components ($r \ge 0 $).
This construction creates a 2-form on a space called
the internal  fusion of the double and denoted  ${\bf D}(G) \,=\, G \times G $ (\cite[Example 6.1]{AMM}):
$$ \omega\, =\, \frac{1}{2} ( a^* \theta,\, b^* \overline{\theta}) + \frac{1}{2} ( a^* \overline{\theta},\, b^* {\theta})
+\frac{1}{2} ( (ab)^* \theta,\, (a^{-1}b^{-1})^* \overline{\theta}). $$

The space $M_{g,r} \,=\, G^{2g+2r}$ with coordinates $ (a_i,\, b_i,\, u_j ,\, v_j)$
for $i\, =\,1,\, \dots,\, g$ and $j \,= 1,\, \dots,\, r$.
The action of $(z_0,\, \dots ,\, z_r)\,\in\, G^{r+1} $ is given by (see \cite{AMM}, Equation (38))
\begin{equation} \label{thirty-eight} a_i\,\longmapsto\, {\rm Ad}_{z_0} a_i, \end{equation}
$$ b_i\,\longmapsto\, {\rm Ad}_{z_0} b_i ,$$
$$u_j \,\longmapsto\, z_0 u_j z_j^{-1} $$
$$v_j \,\longmapsto\, {\rm Ad}_{z_j} v_j$$
(for $i,j$ in the ranges listed above).
Here, we often use $a$ to  denote the tuple $(a_1, \dots, a_g)$
(similarly for $b$). We also use $u$ to denote $(u_1, \dots, u_r)$ (similarly for $v$).

Our earlier notation $M_g$    is an abbreviation for $M_{g,1}$.

The components of the moment map $\mu$ are (see \cite{AMM}, Equation (39))
\begin{equation} \label{thirty-nine} \mu_j (a,\, b,\, u\,, v)\, = \,(v_j)^{-1} ~~~ (j \,= \,1,\, \dots,\, r)
\end{equation}
$$ \mu_0(a,\, b,\, u,\, v) \,= \,{\rm Ad}_{u_1} (v_1)\cdots
{\rm Ad}_{u_r} (v_r) [a_1,\, b_1] \cdots [a_g,\, b_g].
$$

We point out to  the reader that the capping constructions  described in the previous section are special cases
of the quotient of $q$-Hamiltonian spaces described in this section (see Remarks (\ref{r1}) and (\ref{r2})).

For example:
\begin{enumerate}
\item \S5.2: Case 2 corresponds to the $q$-Hamiltonian quotient of $M_{g,2}  \times M_{0,2}$ by the diagonal action of $G \times G$.
\item \S5.3: Case 3 corresponds to the $q$-Hamiltonian quotient of $M_{g,k} \times M_{0,k} $ by the diagonal action of $G^k$.
\item \S5.4: Case 4 corresponds to the $q$-Hamiltonian quotient of $M_{g,1} \times M_{1,1} $ by the diagonal action of $G$.
\item \S5.5: Case 5 corresponds to the $q$-Hamiltonian quotient of $M_{g,n} \times M_{1,n} $ by the diagonal action of $G^n$.
\item \S5.6: Case 6 corresponds to the $q$-Hamiltonian quotient of $M_{g,n} \times M_{g,n} $ by the diagonal action of $G^n$.

\end{enumerate}
\subsection{Fusion} \label{s:fusion}

We continue to assume that $G$ and $H$ are compact, connected, simply connected Lie groups.

\begin{theorem}[{\cite[Theorem 6.1]{AMM}}]
Let $M$ be a $q$-Hamiltonian $(G \times G \times H)$-space
with moment map $(\mu_1,\, \mu_2,\, \mu_3)$.
Let $(M, \omega)$ be a $q$-Hamiltonian $(G \times G \times H)$-space
 with moment map $(\mu_1,\, \mu_2,\, \mu_3)$, so $\mu_1: M \to G$ and
 $\mu_2: M \to G$, while $\mu_3: M \to H$.  Let $G \times H$ act
by the diagonal embedding $(y,\, z)\, \longmapsto\, (y,\,y,\,z)$.
 Then $M$ with $2$-form 
$$ \omega + (\mu_1^* \theta,\, \mu_2^* \overline{\theta} )/2$$
and moment map $\widetilde{\mu} \,=\, (\mu_1 \cdot \mu_2, \,\mu_3)\,:\, M \,\longrightarrow\, G \times H$
is a $q$-Hamiltonian $(G \times H)$-space.
 Here $ \cdot $ denotes  multiplication  in $G$.

Internal fusion means replacing the $(G \times G \times H)$-action 
on a $q$-Hamiltonian $G\times G \times H$-space with a $(G \times H)$-action. 
The space remains the same but the group that acts on it is different.
The 2-form also changes: see \S3.2 of \cite{Mein}.
\end{theorem}

\begin{example}[{\cite[Example 6.1]{AMM}}] Internal fusion turns 
the $q$-Hamiltonian $G\times G$-space $ D(G)$ into a $q$-Hamiltonian $G$-space denoted 
${\bf D}(G)$. 
\end{example}

The space $D(G)\,=\, G \times G $ has coordinates $(u,\, v) \,\in\, G \times G$. There may be $r$ copies of the double which are indexed by variables $u_j, v_j,$ $ j  = 1, \dots, r$.

 The internal fusion
${\bf D}(G)$ is also $G \times G$ with coordinates $(a,\, b)$ ($a,b \in G$). There may be $g$ copies of the internal fusion, which are denoted $(a_i, b_i)$ $(i = 1,\,  \dots\, , g)$. 
The space $D(G)$ is a $q$-Hamiltonian $(G \times G)$-space with moment maps $(v_j^{-1},\, {\rm Ad}_{u_j} v_j)$,
while ${\bf D}(G) $ is a $G$-space with moment map $[a_i,\, b_i]$.

A quasi-Poisson manifold is a special type of  $q$-Hamiltonian space  (see \cite{AKM}).
This is in fact the type of $q$-Hamiltonian space that we reduce to construct character varieties (see \cite{AKM}).  
For relations between $q$-Hamiltonian spaces and Poisson geometry, we refer the reader to \cite{AMM}.

The fusion product of  $ (D(G))^r $ and $\left ( {\bf D}(G) \right) ^g $ is $G^{2g+2r}$ with the $q$-Hamiltonian 
action of $G^{r+1}$ given by Equation \eqref{thirty-eight} and moment maps given by Equation \eqref{thirty-nine}. If we take the 
$q$-Hamiltonian  quotient of this space with respect to the $G^{r+1}$ action, we  obtain a symplectic manifold.  The $2$-form $\omega$ on 
$D(G)^r \times \left ( {\bf D}(G)\right )^g  $ restricts on the level set of moment maps to a form whose
quotient under the action of $G^{r+1} $ is  a symplectic form.  

\begin{remark} \label{r1}
For example, the $q$-Hamiltonian quotient of the product of the spaces $M_g$ and $M_h$ is the subset $$\{( x,\,y) \,\in\, M_g \times M_h\,\big\vert\, \mu_g(x) \,=\, \mu_h(y)\}/G, $$ where $G$ acts on $M_g \times M_h$ by conjugation and $\mu_g\,:\, M_g \,\longrightarrow\, G$ and $\mu_h\,:\, M_h\,\longrightarrow\, G$ are the respective $q$-Hamiltonian moment maps. Recall that $M_g$ and $M_h$ are equipped with $q$-Hamiltonian actions of $G$. The $q$-Hamiltonian 
quotient of $M_g \times M_h$ with respect to the diagonal action of $G$ is  the character variety of a surface of genus $g+h$. 
The  2-forms on $M_g$ and $M_h$ restrict to the level set
$$A:= \{(x,y) \in M_g \times M_h| \mu_g(x) = \mu_h(y) \},  $$ giving a 2-form which is the pullback of a symplectic form on  the quotient of $A$ by the $G$ action. 
  This procedure is the gluing procedure corresponding to gluing together the boundaries of two different surfaces, each with one boundary component.
\end{remark}

\begin{remark} \label{r2}
To take the $q$-Hamiltonian 
 quotient of the diagonal action of $G$ on the $j_1$-th and $j_2$-th copies of $G$, we set $\mu_{j_1}(a,\, b,\,u,\,v) \,=\, \mu_{j_2}(a,\,b,\,u,\,v)$ and 
then take the quotient by the diagonal action of $G$ on these copies of $G$. In other words, we require that $v_{j_1} \,=\, v_{j_2}.$  This operation corresponds to gluing together the $j_1$-th and  $j_2$-th boundary components of a connected surface with $r+1$ boundary components. This procedure is the gluing procedure corresponding to gluing together two boundary components of a connected surface.
Again, the $2$-form on the level set 
$$B:= \{ a,b,u, v | \mu_{j_1}(a,\, b,\,u,\,v) = \mu_{j_2}(a,\,b,\,u,\,v)\}$$ is the pullback of a symplectic form on the quotient of $B$  by the diagonal $G$ action.
\end{remark}

\subsection{Extended moduli spaces}

For reference, see \cite{J} and also \cite{GHJW,H,H2}).

Let $G \,=\,\SU(2)$. The extended moduli space \cite{J} for a genus $g$ surface $\Sigma_g$ with one boundary component was developed for the same purpose as the quasi-Hamiltonian $G$-space from \cite{AMM}. It has a symplectic structure and a Hamiltonian action of $G$. Its symplectic quotient (at a specific coadjoint orbit of $G$) is the moduli space of parabolic bundles associated to that orbit. It has real dimension $6g$.

\subsubsection{Extended moduli space for general $G$}

Let $G$ be a compact Lie group with maximal torus $T$.

We define 
\begin{equation}{\mathcal N}^{\lieg}\,=\,  \{(a_1,\,b_1,\, \dots,\, a_{g},\,b_g) \,\in\, G^{2g}, \Lambda \,\in\, \lieg\,\Big\vert\, 
 \prod_{j=1}^g [a_j,\, b_j] \,=\, \exp (\Lambda) \}. \label{e:ext}\end{equation}
The real dimension of this space is $2g\dim (G)$. 

\subsubsection{More than one  boundary component}

The generalization to $r$ boundary components (where $r \ge 2$) is given in 
(5.6) of \cite{J}:
\begin{equation} \label{enng}
{\mathcal N}^{\lieg, r} \,=\, \{ (a_1,\, \dots,\,a_g,\, b_1,\, 
\dots,\, b_g,\, k_2,\, \dots,\, k_r,\, \lambda_1,\, 
\dots,\, \lambda_r )\ \Big\vert\ a_j,\,b_j,\, k_j \,\in\, G,
\end{equation}
$$\lambda_i \,\in\, \lieg, 
\prod_{i=1}^g [a_{i}, \,b_{i}]\, =\, e^{\lambda_1} k_2 e^{\lambda_2} k_2^{-1}
\dots k_r e^{\lambda_r} k_r^{-1} \}. $$
There is a two-form defined on this space whose restriction to an  open dense subset is symplectic.
This subject may also be described in terms of connections on the 2-manifold rather than elements of products of $G$ -- see \cite{J}. 
The $(r \ge 2)$-boundary component case of the extended moduli space 
is also treated in \cite{HJ}. 

\subsection{Gauge theory version of extended moduli space}

There is a description of the extended moduli space
in terms of connections.

The space $\mathcal{A}_F^\lieg$ parametrizes connections $A \,\in\, \Omega^1 (\Sigma,\, \lieg)$ with curvature $F_A \,=\, 0 $ and 
$A $ having the form $\lambda ds $ on a neighbourhood of 
the boundary (where $s \,\in\, [0,\, 2 \pi]$ is a coordinate on the boundary and
$\lambda$ is a constant in $\lieg$). 
Then $\mathcal{M}_g(\Sigma)$ consists of the quotient of $A^\lieg_F$
by the group of gauge transformations equal to the identity on a 
neighbourhood of the boundary. 
The space $\mathcal{M}_g(\Sigma)$ is homeomorphic
to ${\mathcal{N}}^\lieg$ (see \cite[Proposition 2.5]{J}).

\subsubsection{Symplectic form of extended moduli space}

For the $\SU(2)$ extended
moduli space, the symplectic form is given in \cite{J}, Section 3.1
(see also Definition 3.2).
It is $$\omega_\Sigma (a,\,b)\, = \,\int_\Sigma Tr (a \wedge b), $$
where $a,\,b \,\in\, \Omega^1(\Sigma) \otimes \lieg$ are flat $G$-connections on $\Sigma.$
For general $G$, 
the symplectic form of this space also has this form (see Section 5.2 of
\cite{J}). 
This is 
a simple generalization of the symplectic form defined by 
Atiyah and Bott \cite{AB}.

Let $G\,=\,\SU(2)$. 
By \cite{J}, Proposition 3.1, 
the symplectic form on ${\mathcal M}^{\lieg}$ is nondegenerate provided 
$\Lambda \,= \,0 $ or $\Lambda \,\notin\, \pi {\Z} H$, 
where $H$ is the diagonal $2 \times 2$ matrix with 
entries $(\sqrt{-1},\,-\sqrt{-1})$. 
(a generator of the weight lattice of $G$). 

\subsubsection{Moment map for extended moduli space}

The moment map for the extended moduli space 
$\mathcal{N}^\lieg$ described above (corresponding to a surface of genus 
$g$ with one boundary component) is the map
$$ ((a_1,\, \dots,\, a_{2g}) \,\in\, G^{2g},\, \Lambda \in \lieg ) \,\longmapsto\,
\Lambda.$$

\subsection{Relation with Poisson geometry}

As above, let $G$ be a compact connected simply connected semisimple Lie group.
According to \S8 of \cite{AMM}, there is a bijective correspondence between $q$-Hamiltonian $G$-spaces
and Hamiltonian $LG$-spaces. The key result is Theorem 8.3 in that paper. The following section of 
\cite{AMM}, \S9, uses this machinery to exhibit
spaces of flat connections on oriented 2-manifold as symplectic quotient of Hamiltonian $LG$-spaces, or equivalently 
$q$-Hamiltonian quotients of $q$-Hamiltonian spaces. 

The extended moduli spaces of \cite{J} represent a choice of a gauge for the $LG$ variables. Any $\lieg$-valued 1-form $\alpha$ on $S^1$ is gauge equivalent to 
a 1-form of a particular type: $ \gamma^* \alpha = \lambda ds$ where $\lambda \in \lieg$ is a constant and $\gamma: S^1 \to G$ is a gauge transformation. 
By means of this construction, one no longer needs to use
 the infinite-dimensional group $LG$.  The price one pays is  the (non-canonical) choice of  a gauge. 

If one takes the symplectic quotient of such a Hamiltonian $LG$-space by $(LG)^{r+1}$ , one recovers the space of gauge
equivalence classes of flat connections on an oriented $2$-manifold, with 
fixed values of the holonomy around each boundary component. This means one fixes a value in $L \lieg$
(the Lie algebra of $LG$) for each of the $r+1$ boundary components of the surface.  This value represents the holonomy 
of the connection around that boundary component.
Taking the quotient by $(LG)^{r+1}$ 
gives a symplectic manifold.    

If one takes the quotient by $(LG)^{r+1} $ but does  not  fix the holonomy for the boundary components, the result is a Poisson manifold.
If one then fixes the holonomies around the $r+1$ boundary components, this gives a restriction map 
 from 
 this Poisson manifold to one of its symplectic leaves.

\section{Proof of Theorem \ref{thm:5holed}}\label{5holed}
The coordinate ring  $\C[\X_{n,g}(\SL(2,\C))]$ is finitely generated by traces of curves in $\Sigma_{n,g}$.   Since the Poisson bracket is a derivation and a Lie bracket, it is determined by the pairings of the generators of the coordinate ring.  

When $G=\SL(n,\C)$, Equation \eqref{eq:poisson} simplifies to:
\begin{align}\label{eq:poissonsl}
\{\tr(\rho(\alpha)),\,\tr(\rho(\beta))\}&\,=\,\sum_{p\in
\alpha\cap\beta}\epsilon(p,\,\alpha,\,\beta)\left(\tr(\rho(\alpha_p\beta_p))-\frac{1}{n}\tr(\rho(\alpha))\tr(\rho(\beta))\right).
\end{align}

As is apparent from this formula, to compute the requisite pairings, one need only draw the required curves and compute the traces of the resulting words. The algorithm in \cite{ABL} that computes traces of words in $\C[\X_{n,g}(\SL(2,\C))]$ is effective.   Therefore, the algorithm to compute the Poisson bracket is likewise effective.

We now demonstrate this algorithm with a new non-trivial, but tractable, example.  In principle, many other such examples could be computed with a fully automated implementation of our algorithm.  However, we do this computation by hand with the aid of a compute program implementing the algorithm in \cite{ABL}, and verify the computation is correct with {\it Mathematica}.

There are three open surfaces with Euler characteristic $-3$: the 5-holed sphere, the 3-holed torus, 
and the 1-holed genus 2 surface. In all three cases, the fundamental group is a free group $F_4\,=\,\langle 
c_1,\,c_2,\,c_3,\,c_4\rangle$ of rank 4. Let $t_{\{i\}}\,=\,\tr_{c_i}$, $t_{\{i,j\}}\,=\,\tr_{c_ic_j}$, and 
$t_{\{i,j,k\}}\,=\,\tr_{c_ic_jc_k}$.

Then, the coordinate ring of $\hm(F_4,\,\SL(2,\C))\git \SL(2,\C)$ is generated by the 14 trace functions 
$$\left\{t_{\{1\}},t_{\{2\}},t_{\{3\}},t_{\{4\}},t_{\{1,2\}},t_{\{1,3\}},t_{\{1,4\}},t_{\{2,3\}},t_{\{2,4\}},t_{\{3,4\}},t_{\{1,2,3\}},t_{\{1,2,4\}},t_{\{1,3,4\}},t_{\{2,3,4\}}\right\},$$ 
and its ideal of relations is generated by 14 polynomials in these variables (see \cite{ABL} for details).

To compute the bi-vector for the 5-holed sphere, we need to compute the $14\times 14\,=\,196$ pairings of the form $\{t_x,\,t_y\}$. Since Poisson brackets are anti-symmetric and $t_{\{1\}},$ $t_{\{2\}},$ $t_{\{3\}},$ $t_{\{4\}}$ are Casimirs, we are left with 45 pairings. These come in three types: (a) 15 of $\{t_{\{i,j\}},\,t_{\{k,l\}}\}$, (b) 24 of $\{t_{\{i,j\}},\,t_{\{k,l,m\}}\}$, and (c) 6 of $\{t_{\{i,j,k\}},\,t_{\{l,m,n\}}\}$.

The topological model of $\Sigma_{5,0}$ we will use is in Figure \ref{5holedmodel}, and
we will use Equation \eqref{eq:poissonsl} in the following subsections without explicit mention.

\begin{figure}[ht!]
\includegraphics[scale=0.15]{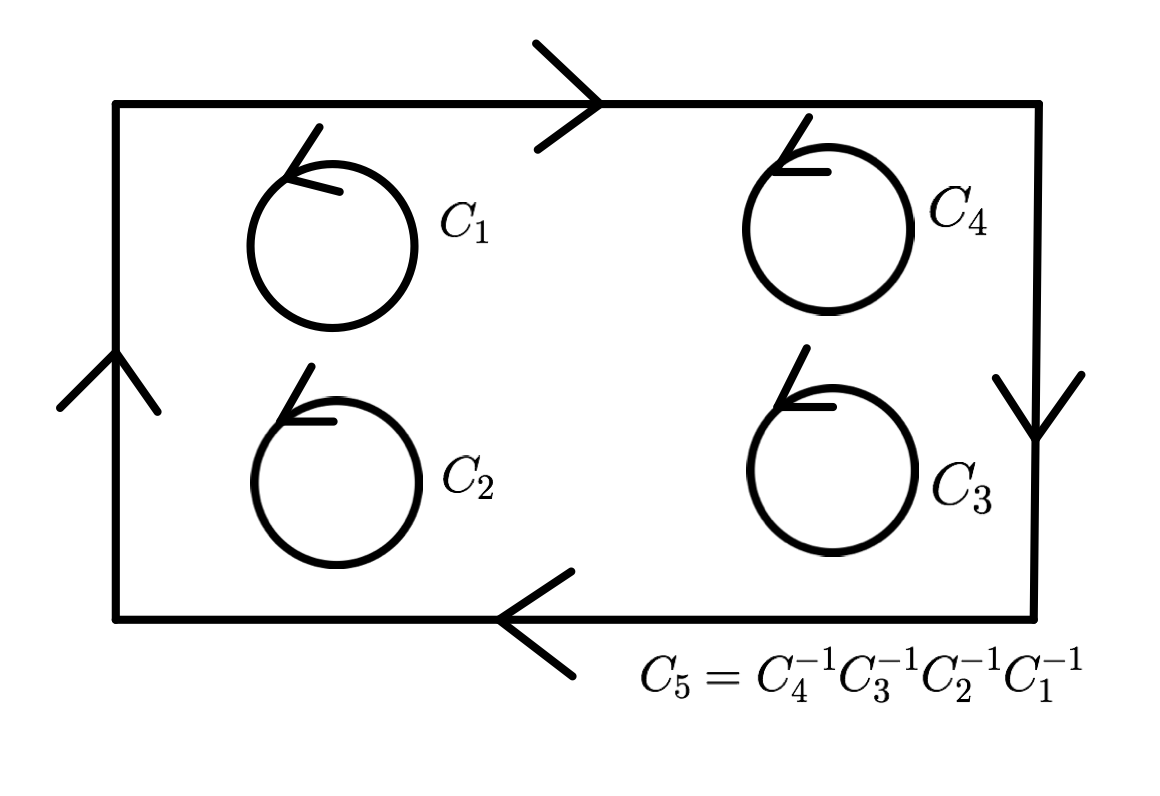} \caption{$\Sigma_{5,0}$ with $c_1c_2c_3c_4c_5=1$.}
\label{5holedmodel} 
\end{figure}

\subsubsection{Type $(a)$ Pairings} There are 15 pairings of type $\{t_{\{i,j\}},\,t_{\{k,l\}}\}$.

Since $c_1c_2$ and $c_3c_4$ are disjoint, $\{t_{\{1,2\}},\,t_{\{3,4\}}\}\,=\,0$.
Likewise, $\{t_{\{1,4\}},\,t_{\{2,3\}}\}\,=\,0$.

The first non-trivial\footnote{One might be tempted to thinking that these curves can be drawn disjoint by drawing 
$c_2c_4$ ``going out around'' $c_3$, but the resulting curve would be homotopic to $c_2c_3c_4c_3^{-1}$ which is 
{\it not} homotopic to $c_2c_4$.} computation will be for $\{t_{\{1,3\}},\,t_{\{2,4\}}\}$. We draw the curves 
$c_1c_3$ and $c_2c_4$ in the 5-holed sphere (see Figure \ref{1324}). Then we simplify the trace functions in the 
formula using the algorithm in \cite{ABL} implemented in {\it Mathematica}.

\begin{figure}[ht!]
\includegraphics[scale=0.15]{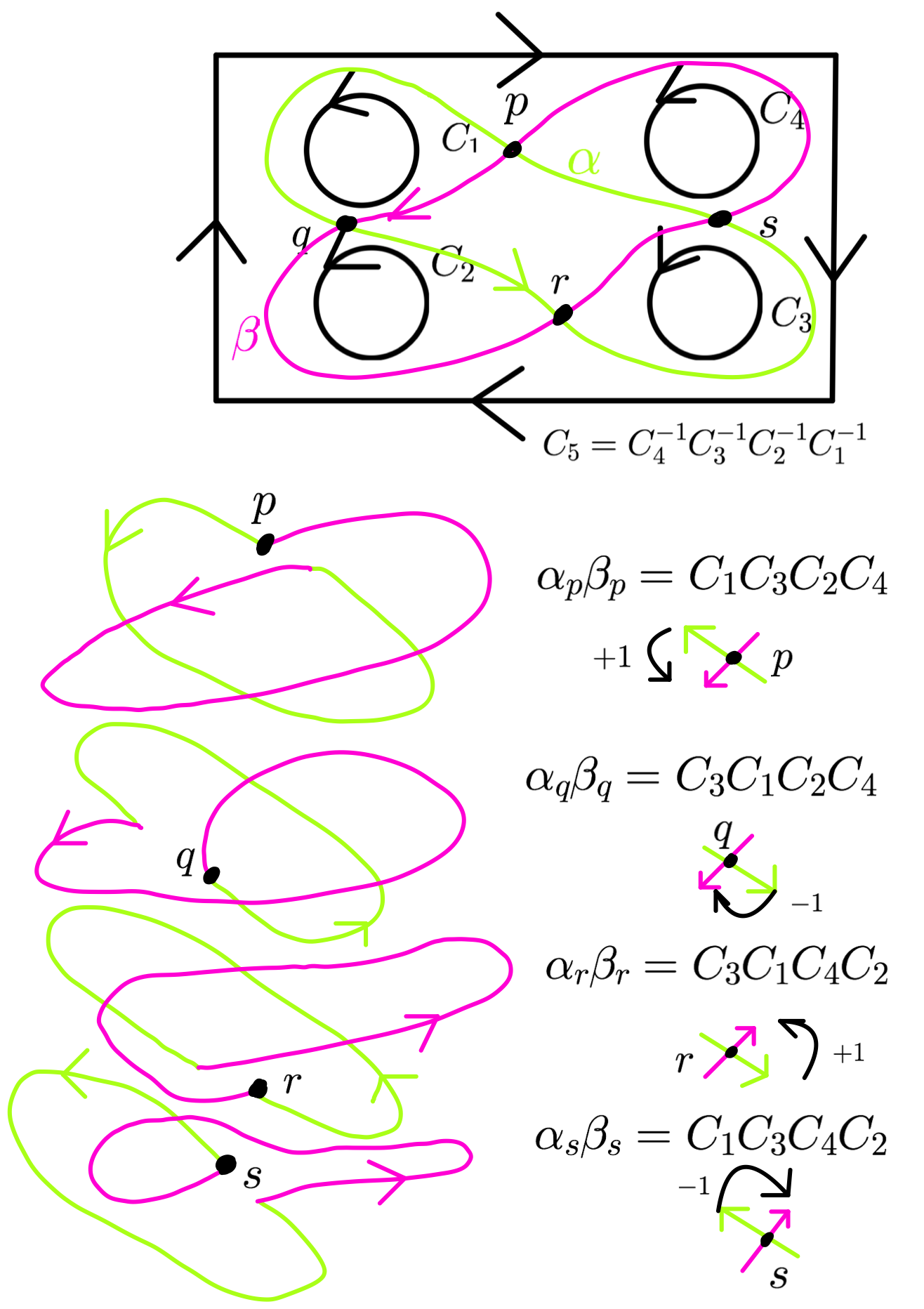} \caption{$\{t_{\{1,3\}},t_{\{2,4\}}\}$}
\label{1324} 
\end{figure}

\begin{eqnarray*}
\{t_{\{1,3\}},\,t_{\{2,4\}}\}&=&(+1)(\tr(c_1c_3c_2c_4)-\frac{1}{2}\tr(c_1c_3)\tr(c_2c_4))\\&+&(-1)(\tr(c_3c_1c_2c_4)-\frac{1}{2}\tr(c_1c_3)\tr(c_2c_4))\\
\\&+&(+1)(\tr(c_3c_1c_4c_2)-\frac{1}{2}\tr(c_1c_3)\tr(c_2c_4))\\
\\&+&(-1)(\tr(c_1c_3c_4c_2)-\frac{1}{2}\tr(c_1c_3)\tr(c_2c_4))\\
&=& t_{\{3\}} t_{\{4\}}
 t_{\{1,2\}}-2 t_{\{3,4\}}
 t_{\{1,2\}}-t_{\{2\}}
 t_{\{3\}}
 t_{\{1,4\}}-t_{\{1\}}
 t_{\{4\}} t_{\{2,3\}}\\&+&2
 t_{\{1,4\}}
 t_{\{2,3\}}+t_{\{1\}}
 t_{\{2\}} t_{\{3,4\}}.
\end{eqnarray*}

The next non-trivial computation will be for $\{t_{\{1,2\}},\,t_{\{1,4\}}\}$. We draw the curves $c_1c_2$ and $c_1c_4$ in Figure \ref{1214}.

\begin{figure}[ht!]
\includegraphics[scale=0.15]{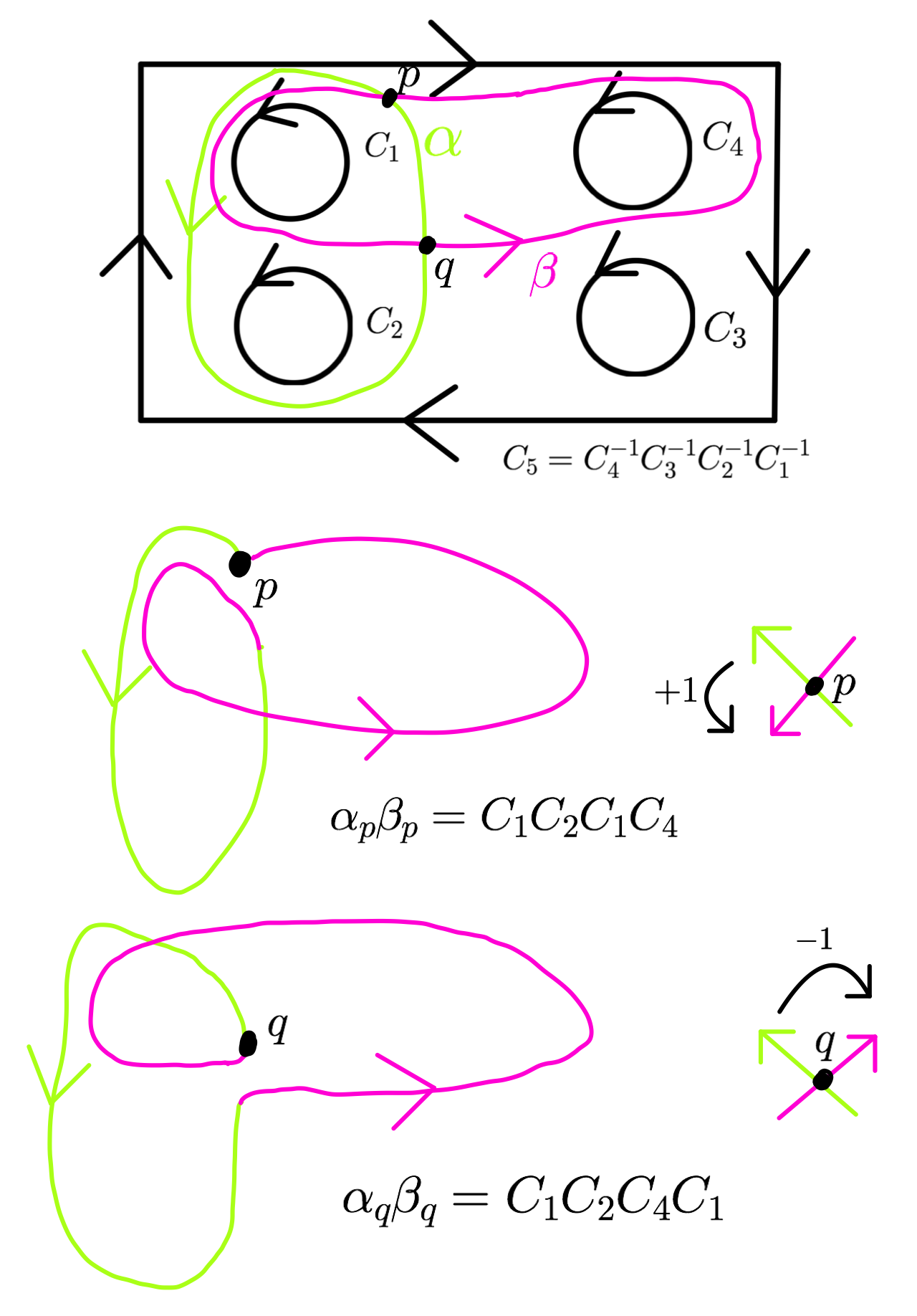} \caption{$\{t_{\{1,2\}},t_{\{1,4\}}\}$}
\label{1214} 
\end{figure}

\begin{eqnarray*}
\{t_{\{1,2\}},\,t_{\{1,4\}}\}&=&(+1)(\tr(c_1c_2c_1c_4)-\frac{1}{2}\tr(c_1c_2)\tr(c_1c_4))\\&+&(-1)(\tr(c_1c_2c_4c_1)-\frac{1}{2}\tr(c_1c_2)\tr(c_1c_4))\\
&=\,& t_{\{1,2\}} t_{\{1,4\}}+2 t_{\{2,4\}}-t_{\{1\}}
 t_{\{1,2,4\}}-t_{\{2\}} t_{\{4\}}.
\end{eqnarray*}

Likewise, by permuting the indices 3 and 4, we have: $$\{t_{\{1,2\}},\,t_{\{1,3\}}\}\,=\,t_{\{1,2\}} t_{\{1,3\}}+
2 t_{\{2,3\}}-t_{\{1\}} t_{\{1,2,3\}}-t_{\{2\}} t_{\{3\}},$$ and by permuting 2 and 3, we have:
$$ \{t_{\{1,3\}},\,t_{\{1,4\}}\}\,=\,t_{\{1,3\}} t_{\{1,4\}}+2 t_{\{3,4\}}-t_{\{1\}}
 t_{\{1,3,4\}}-t_{\{3\}} t_{\{4\}}.$$

We next draw the curves $c_1c_2$ and $c_2c_3$ in Figure \ref{1223} and compute the resulting bracket.

\begin{figure}[ht!]
\includegraphics[scale=0.15]{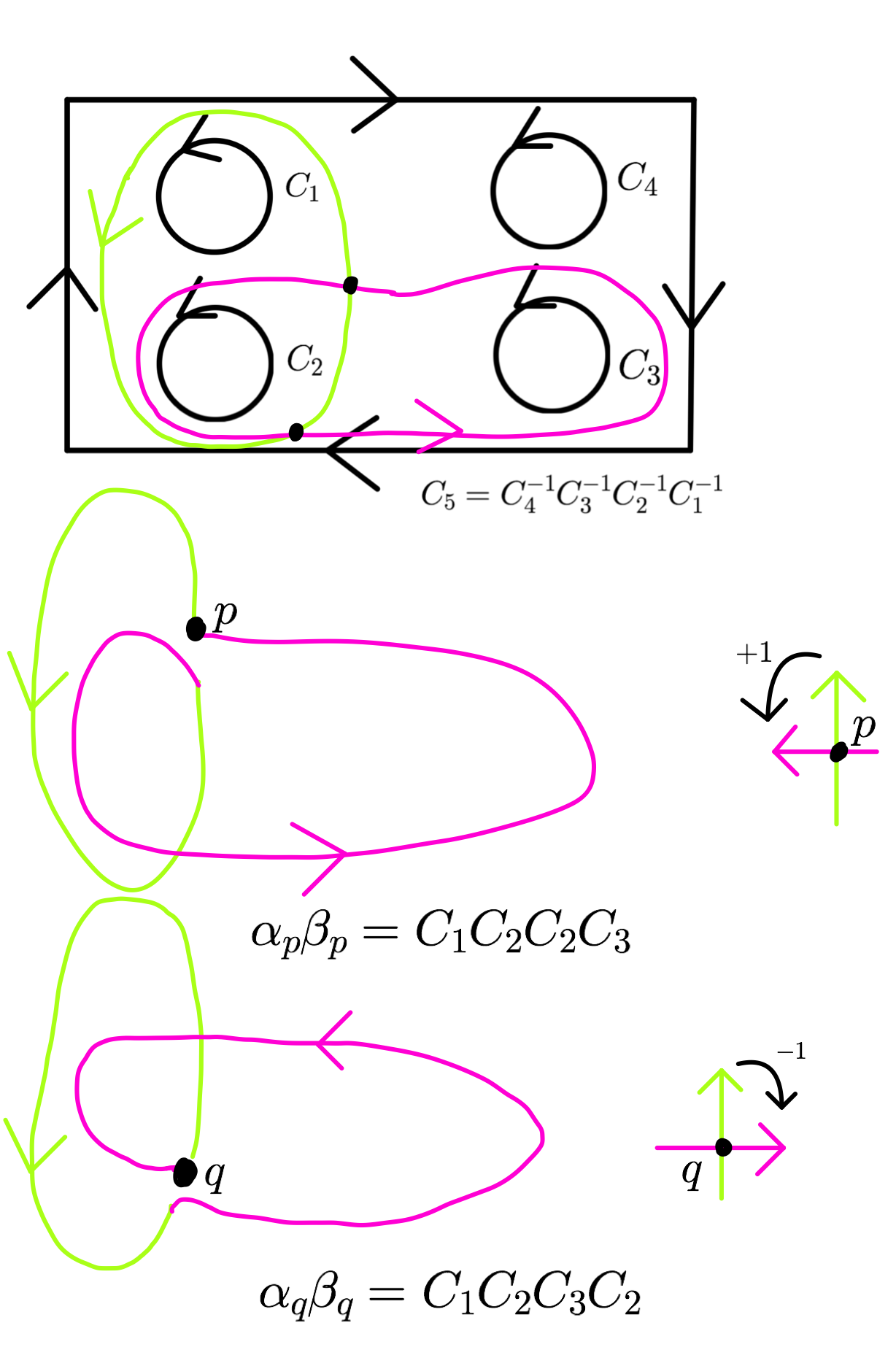} \caption{$\{t_{\{1,2\}},t_{\{2,3\}}\}$}
\label{1223} 
\end{figure}

\begin{eqnarray*}
\{t_{\{1,2\}},\,t_{\{2,3\}}\}&=&(+1)(\tr(c_1c_2c_2c_3)-\frac{1}{2}\tr(c_1c_2)\tr(c_2c_3))\\&+&(-1)(\tr(c_1c_2c_3c_2)-\frac{1}{2}\tr(c_1c_2)\tr(c_2c_3))\\
&=\,& -2 t_{\{1,3\}}-t_{\{1,2\}} t_{\{2,3\}}+t_{\{2\}}
 t_{\{1,2,3\}}+t_{\{1\}} t_{\{3\}}.
\end{eqnarray*}

Again, permuting the indices 3 and 4 we obtain $$\{t_{\{1,2\}},\,t_{\{2,4\}}\}\,=\,-2 t_{\{1,4\}}-t_{\{1,2\}}
t_{\{2,4\}}+t_{\{2\}} t_{\{1,2,4\}}+t_{\{1\}} t_{\{4\}},$$ and by permuting 1 and 4 we obtain
$$\{t_{\{2,4\}},\,t_{\{2,3\}}\}\,=\,-t_{\{2,3\}} t_{\{2,4\}}-2 t_{\{3,4\}}+t_{\{2\}}
t_{\{2,3,4\}}+t_{\{3\}} t_{\{4\}}.$$

We next draw the curves $c_1c_4$ and $c_3c_4$ in Figure \ref{1434} and compute the resulting bracket.

\begin{figure}[ht!]
\includegraphics[scale=0.15]{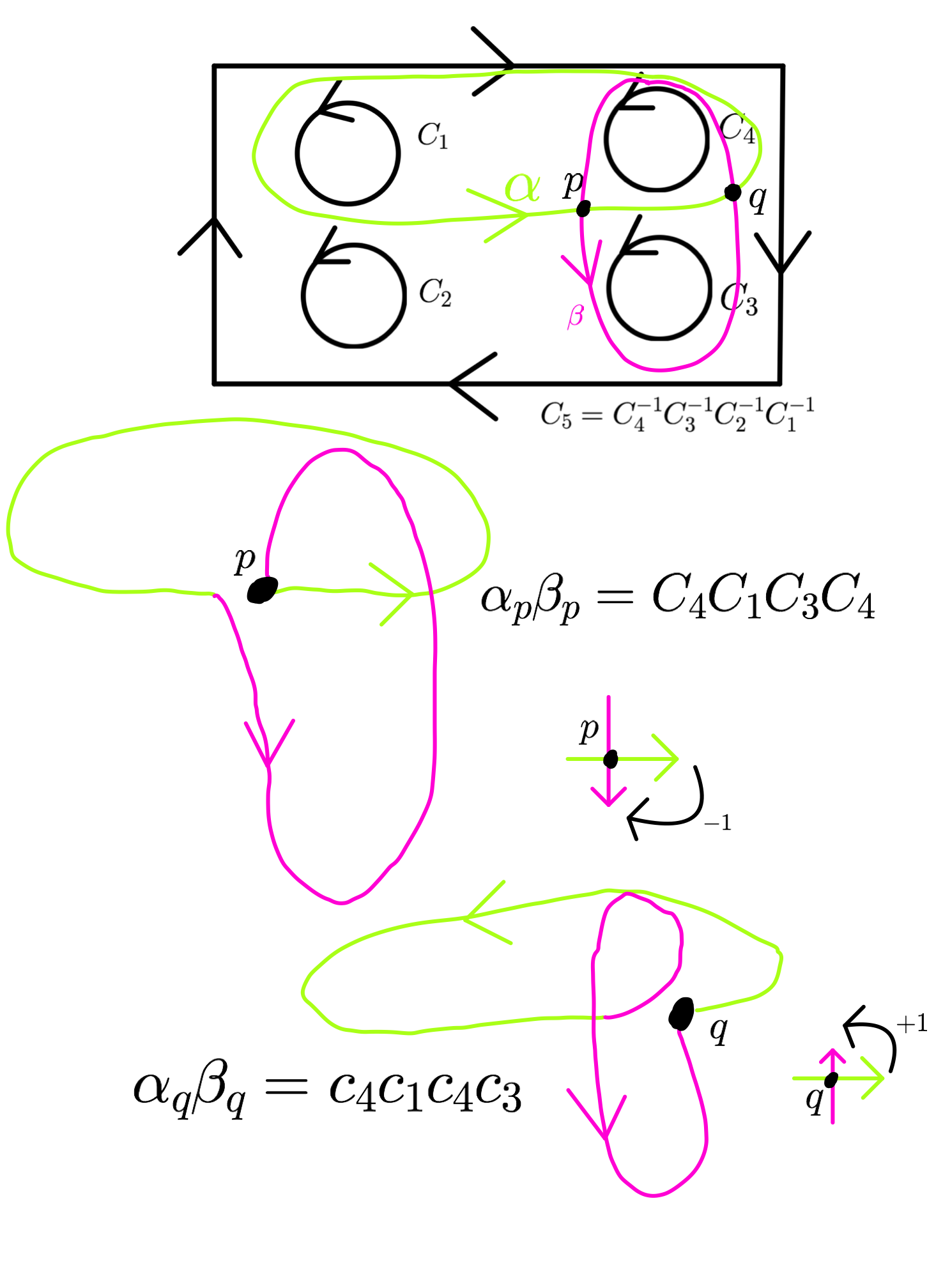} \caption{$\{t_{\{1,4\}},t_{\{3,4\}}\}$}
\label{1434} 
\end{figure}

\begin{eqnarray*}
\{t_{\{1,4\}},\,t_{\{3,4\}}\}&=&(-1)(\tr(c_4c_1c_3c_4)-\frac{1}{2}\tr(c_1c_4)\tr(c_3c_4))\\&+&(+1)(\tr(c_4c_1c_4c_3)-\frac{1}{2}\tr(c_1c_4)\tr(c_3c_4))\\
&=\,& 2 t_{\{1,3\}}+t_{\{1,4\}} t_{\{3,4\}}-t_{\{4\}}
 t_{\{1,3,4\}}-t_{\{1\}} t_{\{3\}}.
\end{eqnarray*}

Permuting the indices 1 and 2 we obtain $$\{t_{\{2,4\}},\,t_{\{3,4\}}\}\,=\,2 t_{\{2,3\}}+t_{\{2,4\}} t_{\{3,4\}}-
t_{\{4\}} t_{\{2,3,4\}}-t_{\{2\}} t_{\{3\}},$$ and by permuting 2 and 3 we obtain $$\{t_{\{1,4\}},\,t_{\{2,4\}}\}
\,=\,2 t_{\{1,2\}}+t_{\{1,4\}} t_{\{2,4\}}-t_{\{4\}}
t_{\{1,2,4\}}-t_{\{1\}} t_{\{2\}}.$$

We next draw the curves $c_2c_3$ and $c_3c_4$ in Figure \ref{2334} and compute the resulting bracket.

\begin{figure}[ht!]
\includegraphics[scale=0.15]{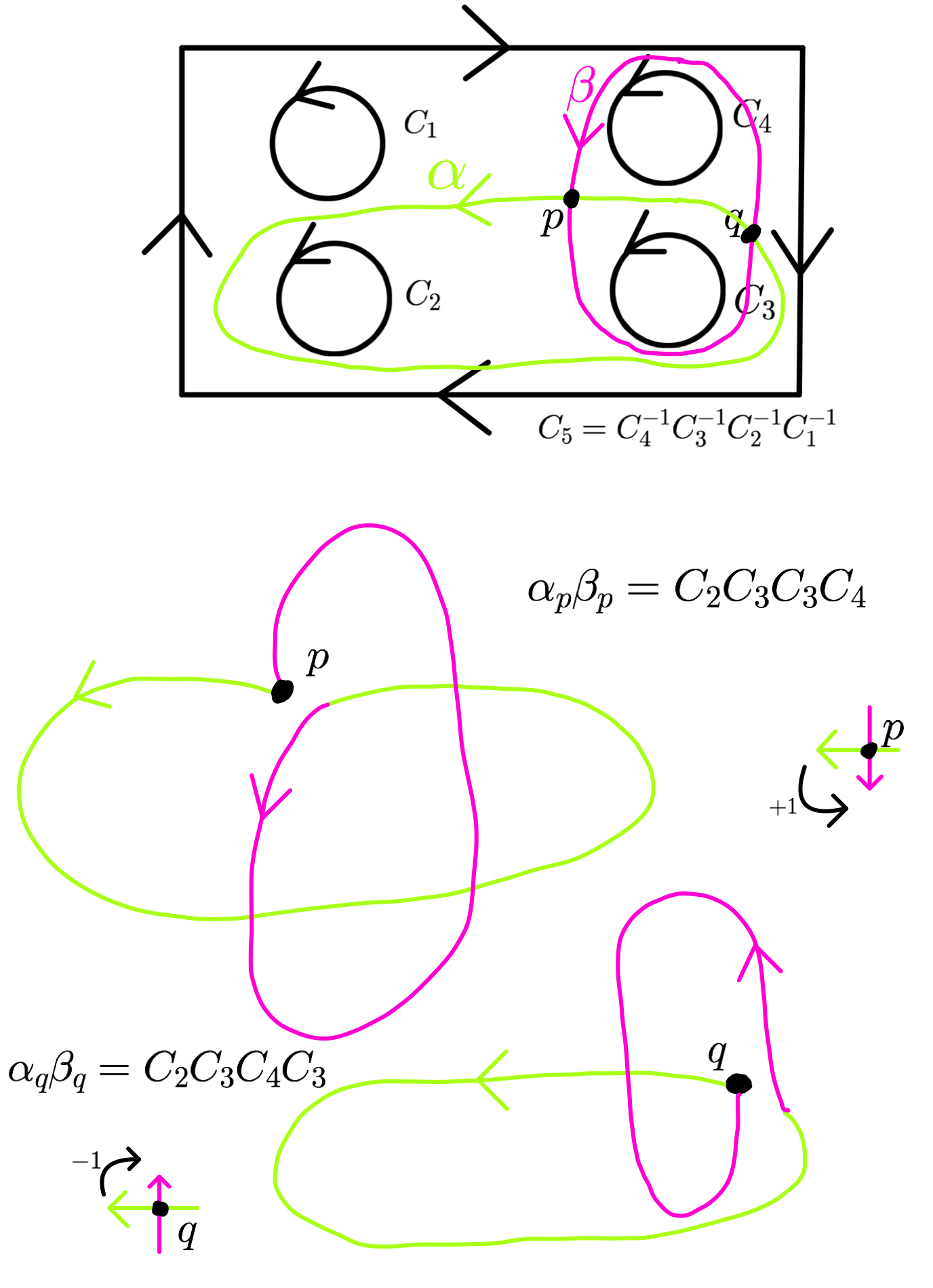} \caption{$\{t_{\{2,3\}},t_{\{3,4\}}\}$}
\label{2334} 
\end{figure}

\begin{eqnarray*}
\{t_{\{2,3\}},\,t_{\{3,4\}}\}&=&(+1)(\tr(c_2c_3c_3c_4)-\frac{1}{2}\tr(c_2c_3)\tr(c_3c_4))\\&+&(-1)(\tr(c_2c_3c_4c_3)-\frac{1}{2}\tr(c_2c_3)\tr(c_3c_4))\\
&=\,& -2 t_{\{2,4\}}-t_{\{2,3\}} t_{\{3,4\}}+t_{\{3\}}
 t_{\{2,3,4\}}+t_{\{2\}} t_{\{4\}}.
\end{eqnarray*}

Permuting the indices 1 and 2 we obtain $$\{t_{\{1,3\}},\,t_{\{3,4\}}\}\,=\,-2 t_{\{1,4\}}-t_{\{1,3\}} t_{\{3,4\}}+t_{\{3\}}
 t_{\{1,3,4\}}+t_{\{1\}} t_{\{4\}},$$ and by permuting 1 and 4 we obtain
$$\{t_{\{2,3\}},\,t_{\{1,3\}}\}\,=\,-2 t_{\{1,2\}}-t_{\{1,3\}} t_{\{2,3\}}+t_{\{3\}}
 t_{\{1,2,3\}}+t_{\{1\}} t_{\{2\}}.$$

\subsubsection{Type $(b)$ Pairings} There are 24 pairings of type $\{t_{\{i,j\}},\,t_{\{k,l,m\}}\}$.

\begin{figure}[ht!]
\includegraphics[scale=0.15]{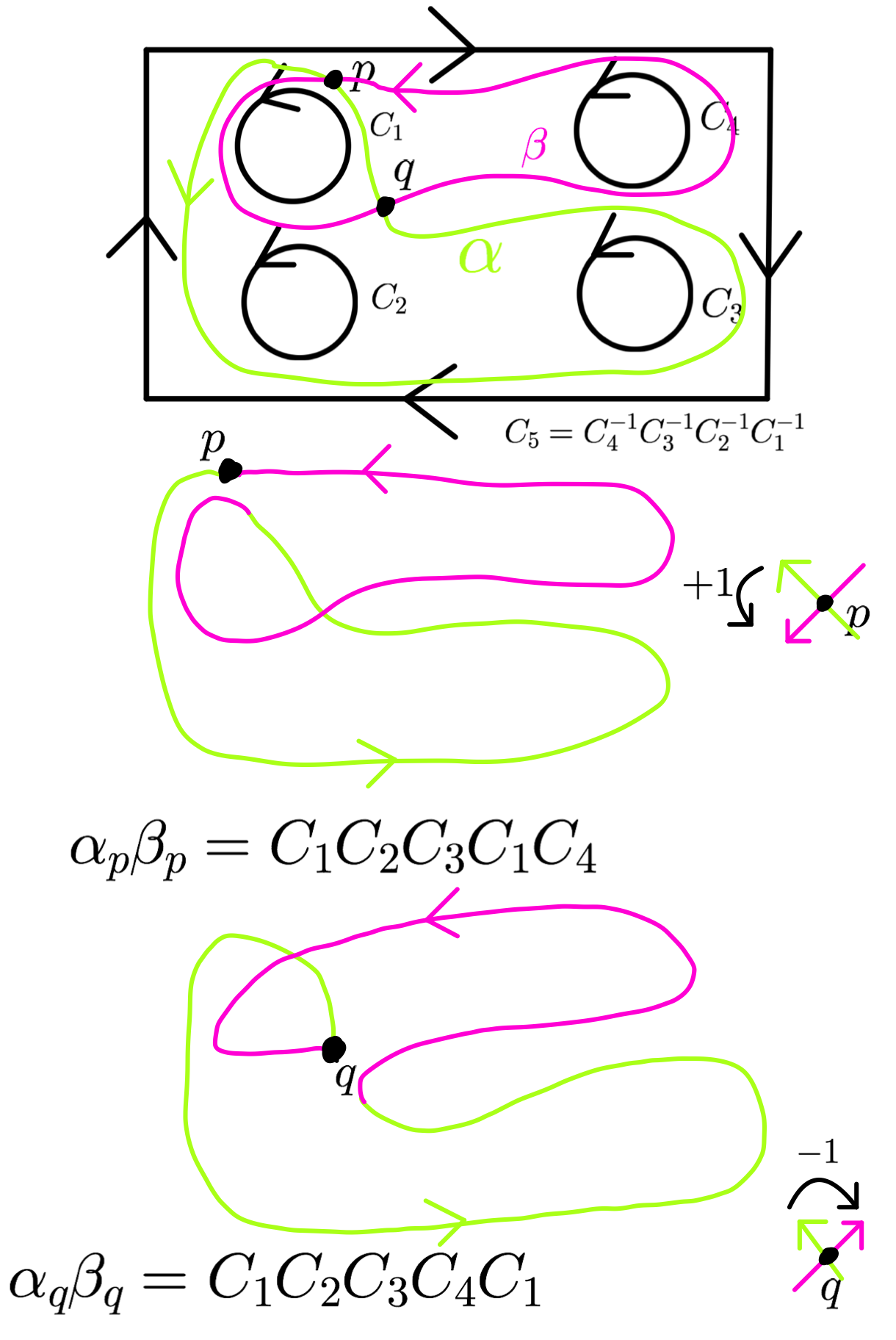} \caption{$\{t_{\{1,2,3\}},t_{\{1,4\}}\}$}
\label{12314} 
\end{figure}

\begin{eqnarray*}
\{t_{\{1,2,3\}},\,t_{\{1,4\}}\}&=&(+1)(\tr(c_1c_2c_3c_1c_4)-\frac{1}{2}\tr(c_1c_2c_3)\tr(c_1c_4))\\&+&(-1)(\tr(c_1c_2c_3c_4c_1)-\frac{1}{2}\tr(c_1c_2c_3)\tr(c_1c_4))\\
&=\,&\frac{1}{2} t_{\{4\}} t_{\{1\}}^2 t_{\{2,3\}}+\frac{1}{2} t_{\{2\}}
 t_{\{1\}}^2 t_{\{3,4\}}-\frac{1}{2} t_{\{1\}}^2
 t_{\{2,3,4\}}+\frac{1}{2} t_{\{3\}} t_{\{4\}} t_{\{1\}}
 t_{\{1,2\}}\\&+&\frac{1}{2} t_{\{2\}} t_{\{3\}} t_{\{1\}}
 t_{\{1,4\}}-\frac{1}{2} t_{\{1\}} t_{\{1,4\}}
 t_{\{2,3\}}+\frac{1}{2} t_{\{1\}} t_{\{1,3\}}
 t_{\{2,4\}}\\&-&\frac{1}{2} t_{\{1\}} t_{\{1,2\}}
 t_{\{3,4\}}-\frac{1}{2} t_{\{4\}} t_{\{1\}}
 t_{\{1,2,3\}}-\frac{1}{2} t_{\{3\}} t_{\{1\}}
 t_{\{1,2,4\}}-\frac{1}{2} t_{\{2\}} t_{\{1\}}
 t_{\{1,3,4\}}\\&-&t_{\{4\}} t_{\{2,3\}}+t_{\{1,4\}} t_{\{1,2,3\}}+2
 t_{\{2,3,4\}}-\frac{1}{2} t_{\{2\}} t_{\{3\}} t_{\{4\}}
 t_{\{1\}}^2.
\end{eqnarray*}

One finds similar diagrams for $\{t_{\{1,2,3\}},\,t_{\{2,4\}}\}$, and $\{t_{\{1,2,3\}},\,t_{\{3,4\}}\}.$

Indeed, we have:

\begin{eqnarray*}
\{t_{\{1,2,3\}},\,t_{\{2,4\}}\}&=&(+1)(\tr(c_1c_2c_3c_2c_4)-\frac{1}{2}\tr(c_1c_2c_3)\tr(c_2c_4))\\&+&(-1)(\tr(c_1c_2c_3c_4c_2)-\frac{1}{2}\tr(c_1c_2c_3)\tr(c_2c_4))\\
&=\,&-t_{\{3\}} t_{\{1,4\}}+t_{\{1\}} t_{\{3,4\}}+t_{\{2,3\}}
 t_{\{1,2,4\}}-t_{\{1,2\}} t_{\{2,3,4\}},
\end{eqnarray*}

and

\begin{eqnarray*}
\{t_{\{1,2,3\}},\,t_{\{3,4\}}\}&=&(+1)(\tr(c_1c_2c_3c_3c_4)-\frac{1}{2}\tr(c_1c_2c_3)\tr(c_3c_4))\\&+&(-1)(\tr(c_1c_2c_3c_4c_3)-\frac{1}{2}\tr(c_1c_2c_3)\tr(c_3c_4))\\
&=\,& -\frac{1}{2} t_{\{4\}} t_{\{3\}}^2 t_{\{1,2\}}-\frac{1}{2} t_{\{2\}}
 t_{\{3\}}^2 t_{\{1,4\}}+\frac{1}{2} t_{\{3\}}^2
 t_{\{1,2,4\}}-\frac{1}{2} t_{\{1\}} t_{\{4\}} t_{\{3\}}
 t_{\{2,3\}}\\&+&\frac{1}{2} t_{\{3\}} t_{\{1,4\}}
 t_{\{2,3\}}-\frac{1}{2} t_{\{3\}} t_{\{1,3\}}
 t_{\{2,4\}}-\frac{1}{2} t_{\{1\}} t_{\{2\}} t_{\{3\}}
 t_{\{3,4\}}\\&+&\frac{1}{2} t_{\{3\}} t_{\{1,2\}}
 t_{\{3,4\}}+\frac{1}{2} t_{\{4\}} t_{\{3\}}
 t_{\{1,2,3\}}+\frac{1}{2} t_{\{2\}} t_{\{3\}}
 t_{\{1,3,4\}}+\frac{1}{2} t_{\{1\}} t_{\{3\}}
 t_{\{2,3,4\}}\\&+&t_{\{4\}} t_{\{1,2\}}-t_{\{3,4\}} t_{\{1,2,3\}}-2
 t_{\{1,2,4\}}+\frac{1}{2} t_{\{1\}} t_{\{2\}} t_{\{4\}}
 t_{\{3\}}^2.
\end{eqnarray*}

Since the curves are disjoint (Figure \ref{12312}), we have:
$\{t_{\{1,2,3\}},\,t_{\{1,2\}}\}\,=\,0,$\, $\{t_{\{1,2,3\}},\,t_{\{1,3\}}\}\,=\,0$ and $\{t_{\{1,2,3\}},\,t_{\{2,3\}}\}
\,=\,0$.

\begin{figure}[ht!]
\includegraphics[scale=0.15]{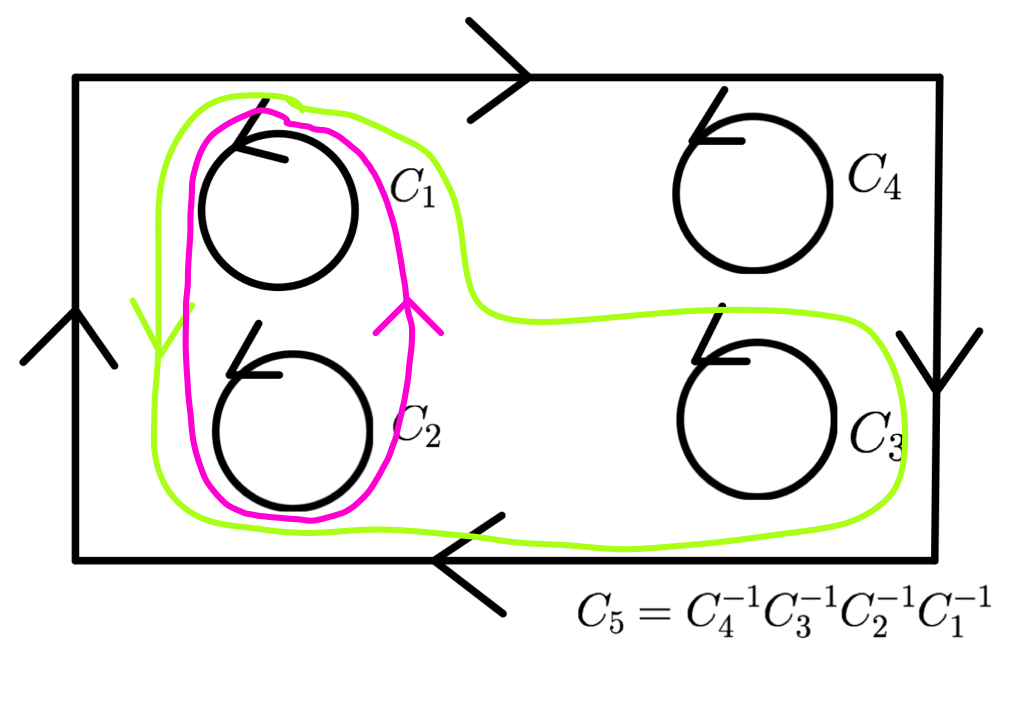} \caption{$\{t_{\{1,2,3\}},t_{\{1,2\}}\}=0$}
\label{12312} 
\end{figure}

So we see, that for each diagram like Figure \ref{12314}, which there are 4, we obtain the data for 6 pairings.

Now on to the next diagram (Figure \ref{12413}).

\begin{figure}[ht!]
\includegraphics[scale=0.15]{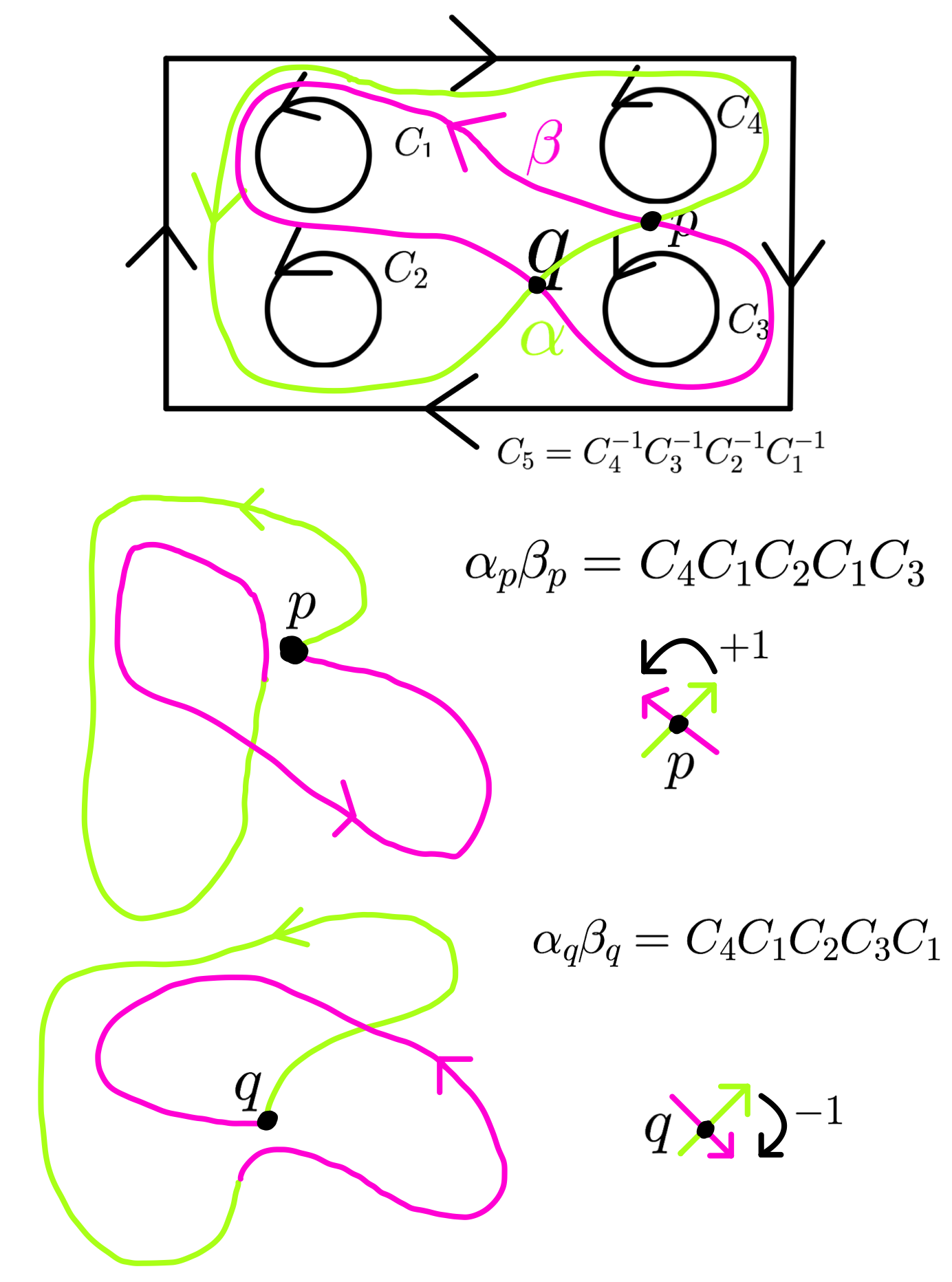} \caption{$\{t_{\{1,2,4\}},t_{\{1,3\}}\}$}
\label{12413} 
\end{figure}

\begin{eqnarray*}
\{t_{\{1,2,4\}},\,t_{\{1,3\}}\}&=&(+1)(\tr(c_4c_1c_2c_1c_3)-\frac{1}{2}\tr(c_1c_2c_4)\tr(c_1c_3))\\&+&(-1)(\tr(c_4c_1c_2c_3c_1)-\frac{1}{2}\tr(c_1c_2c_4)\tr(c_1c_3))\\
&=&
t_{\{4\}} t_{\{2,3\}}-t_{\{2\}} t_{\{3,4\}}-t_{\{1,4\}}t_{\{1,2,3\}}+t_{\{1,2\}} t_{\{1,3,4\}}.
\end{eqnarray*}

Again, one finds similar diagrams for $\{t_{\{1,2,4\}},\,t_{\{2,3\}}\}$, and $\{t_{\{1,2,4\}},\,t_{\{3,4\}}\}.$

Indeed, we have:

\begin{eqnarray*}
\{t_{\{1,2,4\}},\,t_{\{2,3\}}\}&=&(+1)(\tr(c_4c_1c_2c_2c_3)-\frac{1}{2}\tr(c_1c_2c_4)\tr(c_2c_3))\\&+&(-1)(\tr(c_4c_1c_2c_3c_2)-\frac{1}{2}\tr(c_1c_2c_4)\tr(c_2c_3))\\
&=\,&-\frac{1}{2} t_{\{3\}} t_{\{2\}}^2 t_{\{1,4\}}-\frac{1}{2} t_{\{1\}}
 t_{\{2\}}^2 t_{\{3,4\}}+\frac{1}{2} t_{\{2\}}^2
 t_{\{1,3,4\}}-\frac{1}{2} t_{\{3\}} t_{\{4\}} t_{\{2\}}
 t_{\{1,2\}}\\&-&\frac{1}{2} t_{\{1\}} t_{\{4\}} t_{\{2\}}
 t_{\{2,3\}}+\frac{1}{2} t_{\{2\}} t_{\{1,4\}}
 t_{\{2,3\}}-\frac{1}{2} t_{\{2\}} t_{\{1,3\}}
 t_{\{2,4\}}\\&+&\frac{1}{2} t_{\{2\}} t_{\{1,2\}}
 t_{\{3,4\}}+\frac{1}{2} t_{\{4\}} t_{\{2\}}
 t_{\{1,2,3\}}+\frac{1}{2} t_{\{3\}} t_{\{2\}}
 t_{\{1,2,4\}}\\&+&\frac{1}{2} t_{\{1\}} t_{\{2\}}
 t_{\{2,3,4\}}+t_{\{3\}} t_{\{1,4\}}-t_{\{2,3\}} t_{\{1,2,4\}}-2
 t_{\{1,3,4\}}+\frac{1}{2} t_{\{1\}} t_{\{3\}} t_{\{4\}}
 t_{\{2\}}^2,
\end{eqnarray*}

and

\begin{eqnarray*}
\{t_{\{1,2,4\}},\,t_{\{3,4\}}\}&=&(+1)(\tr(c_4c_1c_2c_4c_3)-\frac{1}{2}\tr(c_1c_2c_4)\tr(c_3c_4))\\&+&(-1)(\tr(c_4c_1c_2c_3c_4)-\frac{1}{2}\tr(c_1c_2c_4)\tr(c_3c_4))\\
&=\,& \frac{1}{2} t_{\{3\}} t_{\{4\}}^2 t_{\{1,2\}}+\frac{1}{2} t_{\{1\}}
 t_{\{4\}}^2 t_{\{2,3\}}-\frac{1}{2} t_{\{4\}}^2
 t_{\{1,2,3\}}+\frac{1}{2} t_{\{2\}} t_{\{3\}} t_{\{4\}}
 t_{\{1,4\}}\\&-&\frac{1}{2} t_{\{4\}} t_{\{1,4\}}
 t_{\{2,3\}}+\frac{1}{2} t_{\{4\}} t_{\{1,3\}}
 t_{\{2,4\}}+\frac{1}{2} t_{\{1\}} t_{\{2\}} t_{\{4\}}
 t_{\{3,4\}}\\&-&\frac{1}{2} t_{\{4\}} t_{\{1,2\}}
 t_{\{3,4\}}-\frac{1}{2} t_{\{3\}} t_{\{4\}}
 t_{\{1,2,4\}}-\frac{1}{2} t_{\{2\}} t_{\{4\}}
 t_{\{1,3,4\}}-\frac{1}{2} t_{\{1\}} t_{\{4\}}
 t_{\{2,3,4\}}\\&-&t_{\{3\}} t_{\{1,2\}}+2 t_{\{1,2,3\}}+t_{\{3,4\}}
 t_{\{1,2,4\}}-\frac{1}{2} t_{\{1\}} t_{\{2\}} t_{\{3\}}
 t_{\{4\}}^2.
\end{eqnarray*}

Since the curves are disjoint, we have:
$\{t_{\{1,2,4\}},\,t_{\{1,2\}}\}\,=\,0,$\, $\{t_{\{1,2,4\}},\,t_{\{1,4\}}\}\,=\,0$ and $\{t_{\{1,2,4\}},
\,t_{\{2,4\}}\}\,=\,0$.

Next we have Figure \ref{13423}.
\begin{figure}[ht!]
\includegraphics[scale=0.15]{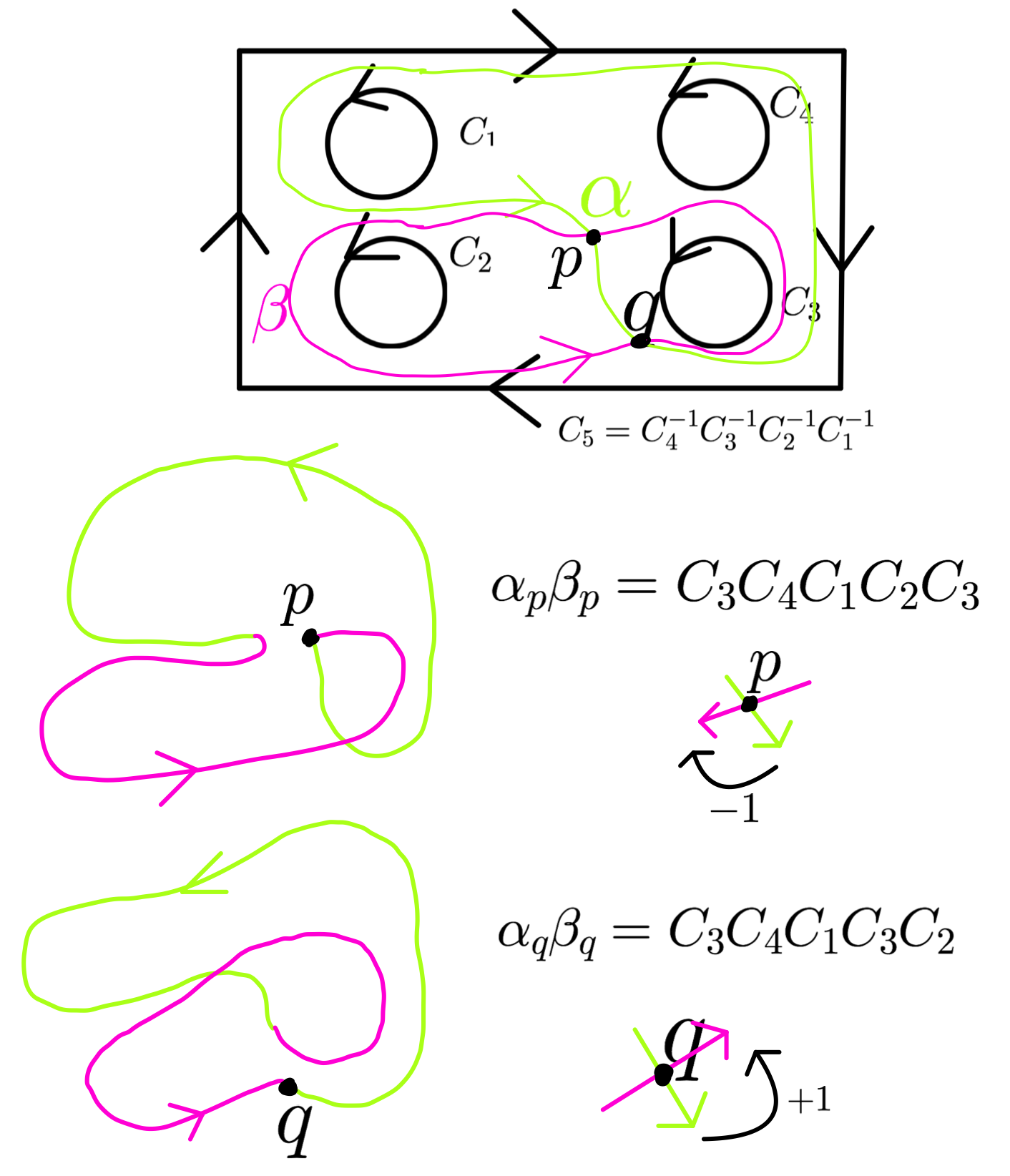} \caption{$\{t_{\{1,3,4\}},t_{\{2,3\}}\}$}
\label{13423} 
\end{figure}

\begin{eqnarray*}
\{t_{\{1,3,4\}},\,t_{\{2,3\}}\}&=&(-1)(\tr(c_3c_4c_1c_2c_3)-\frac{1}{2}\tr(c_1c_3c_4)\tr(c_2c_3))\\&+&(+1)(\tr(c_3c_4c_1c_3c_2)-\frac{1}{2}\tr(c_1c_3c_4)\tr(c_2c_3))\\
&=\,&\frac{1}{2} t_{\{4\}} t_{\{3\}}^2 t_{\{1,2\}}+\frac{1}{2} t_{\{2\}}
 t_{\{3\}}^2 t_{\{1,4\}}-\frac{1}{2} t_{\{3\}}^2
 t_{\{1,2,4\}}+\frac{1}{2} t_{\{1\}} t_{\{4\}} t_{\{3\}}
 t_{\{2,3\}}\\&-&\frac{1}{2} t_{\{3\}} t_{\{1,4\}}
 t_{\{2,3\}}+\frac{1}{2} t_{\{3\}} t_{\{1,3\}}
 t_{\{2,4\}}+\frac{1}{2} t_{\{1\}} t_{\{2\}} t_{\{3\}}
 t_{\{3,4\}}-\frac{1}{2} t_{\{3\}} t_{\{1,2\}}
 t_{\{3,4\}}\\&-&\frac{1}{2} t_{\{4\}} t_{\{3\}}
 t_{\{1,2,3\}}-\frac{1}{2} t_{\{2\}} t_{\{3\}}
 t_{\{1,3,4\}}-\frac{1}{2} t_{\{1\}} t_{\{3\}}
 t_{\{2,3,4\}}-t_{\{2\}} t_{\{1,4\}}\\&+&2 t_{\{1,2,4\}}+t_{\{2,3\}}
 t_{\{1,3,4\}}-\frac{1}{2} t_{\{1\}} t_{\{2\}} t_{\{4\}}
 t_{\{3\}}^2.
\end{eqnarray*}

Again, one finds similar diagrams for $\{t_{\{1,3,4\}},\,t_{\{2,4\}}\}$ and $\{t_{\{1,3,4\}},\,t_{\{1,2\}}\}.$

Indeed, we have:

\begin{eqnarray*}
\{t_{\{1,3,4\}},\,t_{\{2,4\}}\}&=&(-1)(\tr(c_3c_4c_1c_2c_4)-\frac{1}{2}\tr(c_1c_3c_4)\tr(c_2c_4))\\&+&(+1)(\tr(c_3c_4c_1c_4c_2)-\frac{1}{2}\tr(c_1c_3c_4)\tr(c_2c_4))\\
&=\,&t_{\{3\}} t_{\{1,2\}}-t_{\{1\}} t_{\{2,3\}}-t_{\{3,4\}}
 t_{\{1,2,4\}}+t_{\{1,4\}} t_{\{2,3,4\}},
\end{eqnarray*}

and

\begin{eqnarray*}
\{t_{\{1,3,4\}},\,t_{\{1,2\}}\}&=&(-1)(\tr(c_3c_4c_1c_2c_1)-\frac{1}{2}\tr(c_1c_3c_4)\tr(c_1c_2))\\&+&(+1)(\tr(c_3c_4c_1c_1c_2)-\frac{1}{2}\tr(c_1c_3c_4)\tr(c_1c_2))\\
&=\,&-\frac{1}{2} t_{\{4\}} t_{\{1\}}^2 t_{\{2,3\}}-\frac{1}{2} t_{\{2\}}
 t_{\{1\}}^2 t_{\{3,4\}}+\frac{1}{2} t_{\{1\}}^2
 t_{\{2,3,4\}}-\frac{1}{2} t_{\{3\}} t_{\{4\}} t_{\{1\}}
 t_{\{1,2\}}\\&-&\frac{1}{2} t_{\{2\}} t_{\{3\}} t_{\{1\}}
 t_{\{1,4\}}+\frac{1}{2} t_{\{1\}} t_{\{1,4\}}
 t_{\{2,3\}}-\frac{1}{2} t_{\{1\}} t_{\{1,3\}}
 t_{\{2,4\}}+\frac{1}{2} t_{\{1\}} t_{\{1,2\}}
 t_{\{3,4\}}\\&+&\frac{1}{2} t_{\{4\}} t_{\{1\}}
 t_{\{1,2,3\}}+\frac{1}{2} t_{\{3\}} t_{\{1\}}
 t_{\{1,2,4\}}+\frac{1}{2} t_{\{2\}} t_{\{1\}}
 t_{\{1,3,4\}}+t_{\{2\}} t_{\{3,4\}}\\&-&t_{\{1,2\}} t_{\{1,3,4\}}-2
 t_{\{2,3,4\}}+\frac{1}{2} t_{\{2\}} t_{\{3\}} t_{\{4\}}
 t_{\{1\}}^2.
\end{eqnarray*}

Since the curves are disjoint, we have:
$\{t_{\{1,3,4\}},\,t_{\{1,4\}}\}\,=\,0,$\, $\{t_{\{1,3,4\}},\,t_{\{1,3\}}\}\,=\,0$ and
$\{t_{\{1,3,4\}},\,t_{\{3,4\}}\}\,=\,0$.

Figure \ref{23414} depicts the final diagram for this type of pairing.

\begin{figure}[ht!]
\includegraphics[scale=0.15]{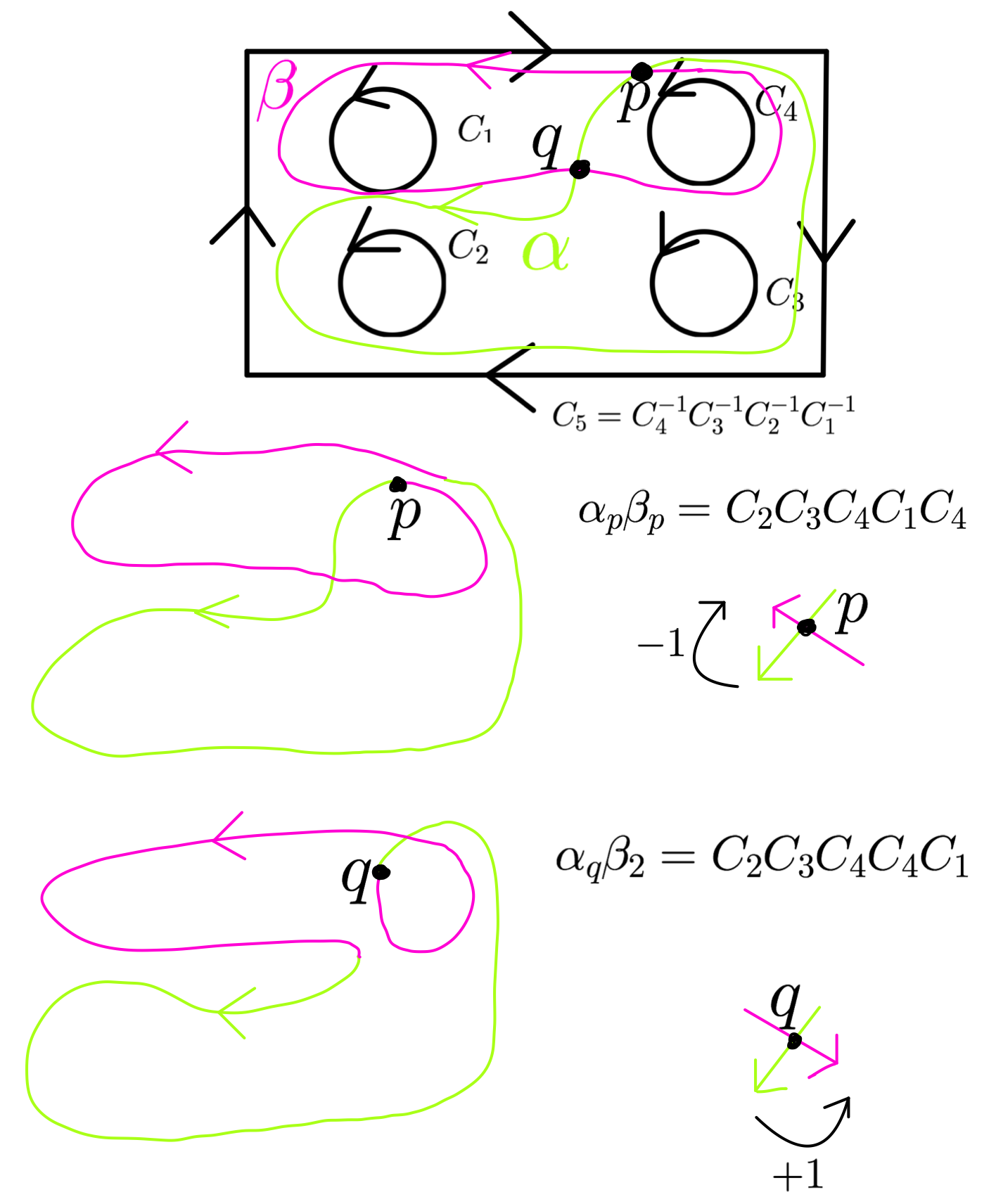} \caption{$\{t_{\{2,3,4\}},t_{\{1,4\}}\}$}
\label{23414} 
\end{figure}

\begin{eqnarray*}
\{t_{\{2,3,4\}},\,t_{\{1,4\}}\}&=&(-1)(\tr(c_2c_3c_4c_1c_4)-\frac{1}{2}\tr(c_2c_3c_4)\tr(c_1c_4))\\&+&(+1)(\tr(c_2c_3c_4c_4c_1)-\frac{1}{2}\tr(c_2c_3c_4)\tr(c_1c_4))\\
&=\,&
-\frac{1}{2} t_{\{3\}} t_{\{4\}}^2 t_{\{1,2\}}-\frac{1}{2} t_{\{1\}}
 t_{\{4\}}^2 t_{\{2,3\}}+\frac{1}{2} t_{\{4\}}^2
 t_{\{1,2,3\}}-\frac{1}{2} t_{\{2\}} t_{\{3\}} t_{\{4\}}
 t_{\{1,4\}}\\&+&\frac{1}{2} t_{\{4\}} t_{\{1,4\}}
 t_{\{2,3\}}-\frac{1}{2} t_{\{4\}} t_{\{1,3\}}
 t_{\{2,4\}}-\frac{1}{2} t_{\{1\}} t_{\{2\}} t_{\{4\}}
 t_{\{3,4\}}+\frac{1}{2} t_{\{4\}} t_{\{1,2\}}
 t_{\{3,4\}}\\&+&\frac{1}{2} t_{\{3\}} t_{\{4\}}
 t_{\{1,2,4\}}+\frac{1}{2} t_{\{2\}} t_{\{4\}}
 t_{\{1,3,4\}}+\frac{1}{2} t_{\{1\}} t_{\{4\}}
 t_{\{2,3,4\}}+t_{\{1\}} t_{\{2,3\}}\\&-&2 t_{\{1,2,3\}}-t_{\{1,4\}}
 t_{\{2,3,4\}}+\frac{1}{2} t_{\{1\}} t_{\{2\}} t_{\{3\}}
 t_{\{4\}}^2
\end{eqnarray*}

Again, one finds similar diagrams for $\{t_{\{2,3,4\}},\,t_{\{1,3\}}\}$, and $\{t_{\{2,3,4\}},\,t_{\{1,2\}}\}.$

Indeed, we have:

\begin{eqnarray*}
\{t_{\{2,3,4\}},\,t_{\{1,3\}}\}&=&(-1)(\tr(c_2c_3c_4c_1c_3)-\frac{1}{2}\tr(c_2c_3c_4)\tr(c_1c_3))\\&+&(+1)(\tr(c_2c_3c_4c_3c_1)-\frac{1}{2}\tr(c_2c_3c_4)\tr(c_1c_3))\\
&=\,&-t_{\{4\}} t_{\{1,2\}}+t_{\{2\}} t_{\{1,4\}}+t_{\{3,4\}}
 t_{\{1,2,3\}}-t_{\{2,3\}} t_{\{1,3,4\}},
\end{eqnarray*}

and

\begin{eqnarray*}
\{t_{\{2,3,4\}},\,t_{\{1,2\}}\}&=&(-1)(\tr(c_2c_3c_4c_1c_2)-\frac{1}{2}\tr(c_2c_3c_4)\tr(c_1c_2))\\&+&(+1)(\tr(c_2c_3c_4c_2c_1)-\frac{1}{2}\tr(c_2c_3c_4)\tr(c_1c_2))\\
&=\,& \frac{1}{2} t_{\{3\}} t_{\{2\}}^2 t_{\{1,4\}}+\frac{1}{2} t_{\{1\}}
 t_{\{2\}}^2 t_{\{3,4\}}-\frac{1}{2} t_{\{2\}}^2
 t_{\{1,3,4\}}+\frac{1}{2} t_{\{3\}} t_{\{4\}} t_{\{2\}}
 t_{\{1,2\}}\\&+&\frac{1}{2} t_{\{1\}} t_{\{4\}} t_{\{2\}}
 t_{\{2,3\}}-\frac{1}{2} t_{\{2\}} t_{\{1,4\}}
 t_{\{2,3\}}+\frac{1}{2} t_{\{2\}} t_{\{1,3\}}
 t_{\{2,4\}}-\frac{1}{2} t_{\{2\}} t_{\{1,2\}}
 t_{\{3,4\}}\\&-&\frac{1}{2} t_{\{4\}} t_{\{2\}}
 t_{\{1,2,3\}}-\frac{1}{2} t_{\{3\}} t_{\{2\}}
 t_{\{1,2,4\}}-\frac{1}{2} t_{\{1\}} t_{\{2\}}
 t_{\{2,3,4\}}-t_{\{1\}} t_{\{3,4\}}\\&+&2 t_{\{1,3,4\}}+t_{\{1,2\}}
 t_{\{2,3,4\}}-\frac{1}{2} t_{\{1\}} t_{\{3\}} t_{\{4\}}
 t_{\{2\}}^2.
\end{eqnarray*}

Since the curves are disjoint, we have:
$\{t_{\{2,3,4\}},\,t_{\{3,4\}}\}\,=\,0,$\, $\{t_{\{2,3,4\}},\,t_{\{2,3\}}\}\,=\,0$ and $\{t_{\{2,3,4\}},
\,t_{\{2,4\}}\}\,=\,0$.

\subsubsection{Type $(c)$ Pairings}

There are 6 pairings of type $\{t_{\{i,j,k\}},\,t_{\{l,m,n\}}\}$. Figure \ref{123134} is the first one of this type.

\begin{figure}[ht!]
\includegraphics[scale=0.15]{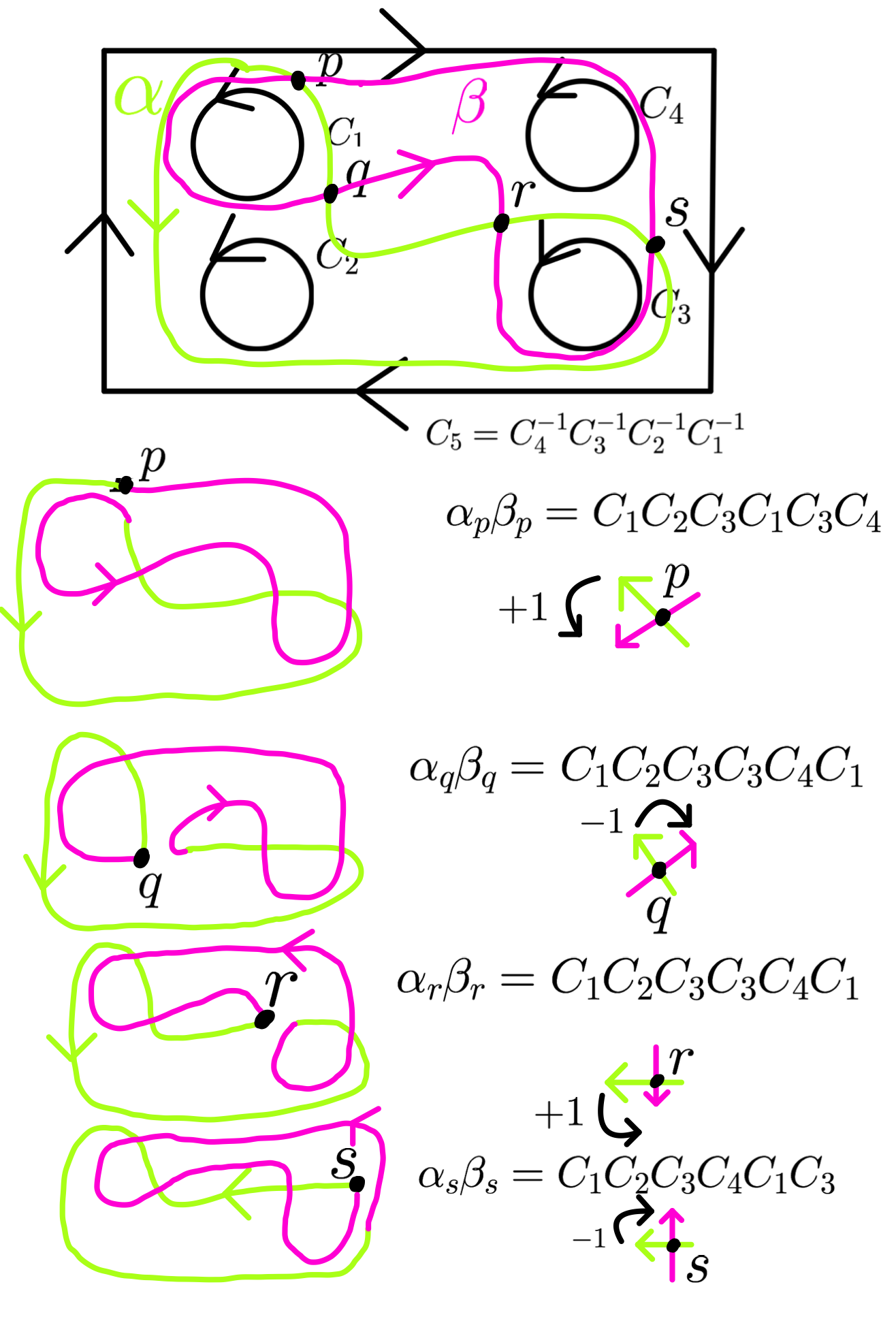} \caption{$\{t_{\{1,2,3\}},t_{\{1,3,4\}}\}$}
\label{123134} 
\end{figure}

\begin{eqnarray*}
\{t_{\{1,2,3\}},\,t_{\{1,3,4\}}\}&=&(+1)(\tr(c_1c_2c_3c_1c_3c_4)-\frac{1}{2}\tr(c_1c_2c_3)\tr(c_1c_3c_4))\\&+&(-1)(\tr(c_1c_2c_3c_3c_4c_1)-\frac{1}{2}\tr(c_1c_2c_3)\tr(c_1c_3c_4))\\&+&(+1)(\tr(c_1c_2c_3c_3c_4c_1)-\frac{1}{2}\tr(c_1c_2c_3)\tr(c_1c_3c_4))\\&+&(-1)(\tr(c_1c_2c_3c_4c_1c_3)-\frac{1}{2}\tr(c_1c_2c_3)\tr(c_1c_3c_4))\\
&=\,&\frac{1}{2} t_{\{2,4\}} t_{\{1,3\}}^2-\frac{1}{2} t_{\{1\}}
 t_{\{2\}} t_{\{3\}} t_{\{4\}} t_{\{1,3\}}+\frac{1}{2} t_{\{3\}}
 t_{\{4\}} t_{\{1,2\}} t_{\{1,3\}}\\&+&\frac{1}{2} t_{\{2\}} t_{\{3\}}
 t_{\{1,4\}} t_{\{1,3\}}+\frac{1}{2} t_{\{1\}} t_{\{4\}}
 t_{\{2,3\}} t_{\{1,3\}}-\frac{1}{2} t_{\{1,4\}} t_{\{2,3\}}
 t_{\{1,3\}}\\&+&\frac{1}{2} t_{\{1\}} t_{\{2\}} t_{\{3,4\}}
 t_{\{1,3\}}-\frac{1}{2} t_{\{1,2\}} t_{\{3,4\}}
 t_{\{1,3\}}-\frac{1}{2} t_{\{4\}} t_{\{1,2,3\}}
 t_{\{1,3\}}\\&-&\frac{1}{2} t_{\{3\}} t_{\{1,2,4\}}
 t_{\{1,3\}}-\frac{1}{2} t_{\{2\}} t_{\{1,3,4\}}
 t_{\{1,3\}}-\frac{1}{2} t_{\{1\}} t_{\{2,3,4\}} t_{\{1,3\}}-2
 t_{\{2,4\}}\\&-&2 t_{\{2,3\}} t_{\{3,4\}}+t_{\{1,2,3\}}
 t_{\{1,3,4\}}+2 t_{\{3\}} t_{\{2,3,4\}}+t_{\{2\}} t_{\{4\}}.
\end{eqnarray*}

Similarly, we have:

\begin{eqnarray*}
\{t_{\{1,2,3\}},\,t_{\{1,2,4\}}\}&=&(+1)(\tr(c_1c_2c_3c_1c_2c_4)-\frac{1}{2}\tr(c_1c_2c_3)\tr(c_1c_2c_4))\\&+&(-1)(\tr(c_1c_2c_3c_4c_1c_2)-\frac{1}{2}\tr(c_1c_2c_3)\tr(c_1c_2c_4))\\
&=\,&\frac{1}{2} t_{\{3\}} t_{\{4\}} t_{\{1,2\}}^2-\frac{1}{2}
 t_{\{3,4\}} t_{\{1,2\}}^2-\frac{1}{2} t_{\{1\}} t_{\{2\}}
 t_{\{3\}} t_{\{4\}} t_{\{1,2\}}\\&+&\frac{1}{2} t_{\{2\}} t_{\{3\}}
 t_{\{1,4\}} t_{\{1,2\}}+\frac{1}{2} t_{\{1\}} t_{\{4\}}
 t_{\{2,3\}} t_{\{1,2\}}-\frac{1}{2} t_{\{1,4\}} t_{\{2,3\}}
 t_{\{1,2\}}\\&+&\frac{1}{2} t_{\{1,3\}} t_{\{2,4\}}
 t_{\{1,2\}}+\frac{1}{2} t_{\{1\}} t_{\{2\}} t_{\{3,4\}}
 t_{\{1,2\}}-\frac{1}{2} t_{\{4\}} t_{\{1,2,3\}}
 t_{\{1,2\}}\\&-&\frac{1}{2} t_{\{3\}} t_{\{1,2,4\}}
 t_{\{1,2\}}-\frac{1}{2} t_{\{2\}} t_{\{1,3,4\}}
 t_{\{1,2\}}-\frac{1}{2} t_{\{1\}} t_{\{2,3,4\}} t_{\{1,2\}}+2
 t_{\{3,4\}}\\&+&t_{\{1,2,3\}} t_{\{1,2,4\}}-t_{\{3\}} t_{\{4\}},
\end{eqnarray*}
and

\begin{eqnarray*}
\{t_{\{1,2,3\}},\,t_{\{2,3,4\}}\}&=&(+1)(\tr(c_1c_2c_3c_2c_3c_4)-\frac{1}{2}\tr(c_1c_2c_3)\tr(c_2c_3c_4))\\&+&(-1)(\tr(c_1c_2c_3c_4c_2c_3)-\frac{1}{2}\tr(c_1c_2c_3)\tr(c_2c_3c_4))\\
&=\,&-\frac{1}{2} t_{\{1\}} t_{\{4\}} t_{\{2,3\}}^2+\frac{1}{2}
 t_{\{1,4\}} t_{\{2,3\}}^2+\frac{1}{2} t_{\{1\}} t_{\{2\}}
 t_{\{3\}} t_{\{4\}} t_{\{2,3\}}\\&-&\frac{1}{2} t_{\{3\}} t_{\{4\}}
 t_{\{1,2\}} t_{\{2,3\}}-\frac{1}{2} t_{\{2\}} t_{\{3\}}
 t_{\{1,4\}} t_{\{2,3\}}\\&-&\frac{1}{2} t_{\{1,3\}} t_{\{2,4\}}
 t_{\{2,3\}}-\frac{1}{2} t_{\{1\}} t_{\{2\}} t_{\{3,4\}}
 t_{\{2,3\}}\\&+&\frac{1}{2} t_{\{1,2\}} t_{\{3,4\}}
 t_{\{2,3\}}+\frac{1}{2} t_{\{4\}} t_{\{1,2,3\}}
 t_{\{2,3\}}+\frac{1}{2} t_{\{3\}} t_{\{1,2,4\}}
 t_{\{2,3\}}\\&+&\frac{1}{2} t_{\{2\}} t_{\{1,3,4\}}
 t_{\{2,3\}}+\frac{1}{2} t_{\{1\}} t_{\{2,3,4\}} t_{\{2,3\}}-2
 t_{\{1,4\}}-t_{\{1,2,3\}} t_{\{2,3,4\}}\\&+&t_{\{1\}} t_{\{4\}}.
\end{eqnarray*}

Next, and very similarly, we have Figure \ref{124134}.

\begin{figure}[ht!]
\includegraphics[scale=0.15]{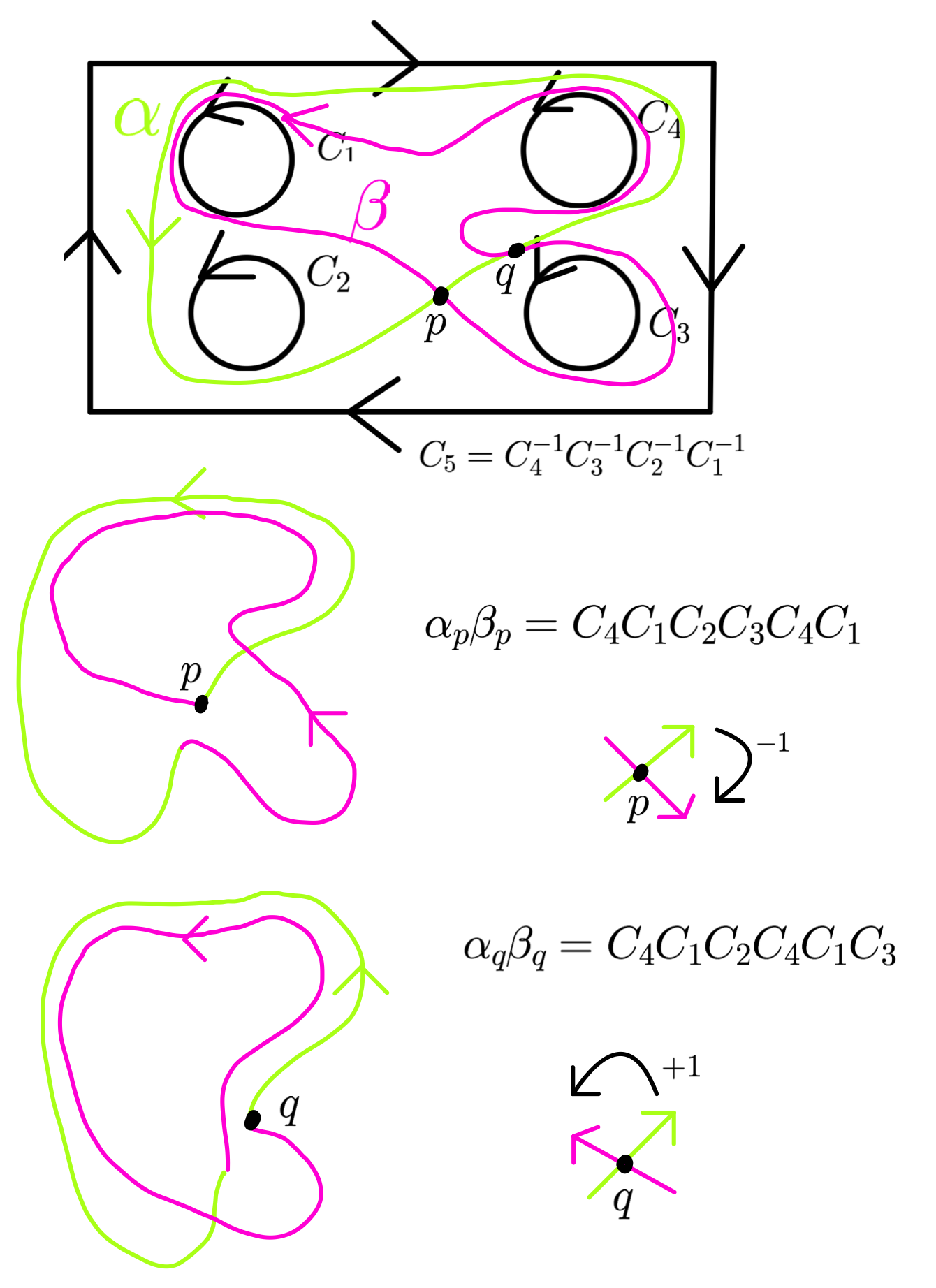} \caption{$\{t_{\{1,2,4\}},t_{\{1,3,4\}}\}$}
\label{124134} 
\end{figure}

\begin{eqnarray*}
\{t_{\{1,2,4\}},\,t_{\{1,3,4\}}\}&=&(-1)(\tr(c_4c_1c_2c_3c_4c_1)-\frac{1}{2}\tr(c_1c_2c_4)\tr(c_1c_3c_4))\\&+&(+1)(\tr(c_4c_1c_2c_4c_1c_3)-\frac{1}{2}\tr(c_1c_2c_4)\tr(c_1c_3c_4))\\
&=\,&\frac{1}{2} t_{\{2\}} t_{\{3\}} t_{\{1,4\}}^2-\frac{1}{2}
 t_{\{2,3\}} t_{\{1,4\}}^2-\frac{1}{2} t_{\{1\}} t_{\{2\}}
 t_{\{3\}} t_{\{4\}} t_{\{1,4\}}\\&+&\frac{1}{2} t_{\{3\}} t_{\{4\}}
 t_{\{1,2\}} t_{\{1,4\}}+\frac{1}{2} t_{\{1\}} t_{\{4\}}
 t_{\{2,3\}} t_{\{1,4\}}+\frac{1}{2} t_{\{1,3\}} t_{\{2,4\}}
 t_{\{1,4\}}\\&+&\frac{1}{2} t_{\{1\}} t_{\{2\}} t_{\{3,4\}}
 t_{\{1,4\}}-\frac{1}{2} t_{\{1,2\}} t_{\{3,4\}}
 t_{\{1,4\}}-\frac{1}{2} t_{\{4\}} t_{\{1,2,3\}}
 t_{\{1,4\}}\\&-&\frac{1}{2} t_{\{3\}} t_{\{1,2,4\}}
 t_{\{1,4\}}-\frac{1}{2} t_{\{2\}} t_{\{1,3,4\}}
 t_{\{1,4\}}-\frac{1}{2} t_{\{1\}} t_{\{2,3,4\}} t_{\{1,4\}}+2
 t_{\{2,3\}}\\&+&t_{\{1,2,4\}} t_{\{1,3,4\}}-t_{\{2\}} t_{\{3\}}.
\end{eqnarray*}

Similarly,
\begin{eqnarray*}
\{t_{\{1,2,4\}},\,t_{\{2,3,4\}}\}&=&(-1)(\tr(c_4c_1c_2c_3c_4c_2)-\frac{1}{2}\tr(c_1c_2c_4)\tr(c_2c_3c_4))\\&+&(+1)(\tr(c_4c_1c_2c_4c_2c_3)-\frac{1}{2}\tr(c_1c_2c_4)\tr(c_2c_3c_4))\\
&=\,&\frac{1}{2} t_{\{1,3\}} t_{\{2,4\}}^2-\frac{1}{2} t_{\{1\}}
 t_{\{2\}} t_{\{3\}} t_{\{4\}} t_{\{2,4\}}+\frac{1}{2} t_{\{3\}}
 t_{\{4\}} t_{\{1,2\}} t_{\{2,4\}}\\&+&\frac{1}{2} t_{\{2\}} t_{\{3\}}
 t_{\{1,4\}} t_{\{2,4\}}+\frac{1}{2} t_{\{1\}} t_{\{4\}}
 t_{\{2,3\}} t_{\{2,4\}}-\frac{1}{2} t_{\{1,4\}} t_{\{2,3\}}
 t_{\{2,4\}}\\&+&\frac{1}{2} t_{\{1\}} t_{\{2\}} t_{\{3,4\}}
 t_{\{2,4\}}-\frac{1}{2} t_{\{1,2\}} t_{\{3,4\}}
 t_{\{2,4\}}-\frac{1}{2} t_{\{4\}} t_{\{1,2,3\}}
 t_{\{2,4\}}\\&-&\frac{1}{2} t_{\{3\}} t_{\{1,2,4\}}
 t_{\{2,4\}}-\frac{1}{2} t_{\{2\}} t_{\{1,3,4\}}
 t_{\{2,4\}}-\frac{1}{2} t_{\{1\}} t_{\{2,3,4\}} t_{\{2,4\}}\\&-&2
 t_{\{1,3\}}-2 t_{\{1,2\}} t_{\{2,3\}}+2 t_{\{2\}}
 t_{\{1,2,3\}}-\frac{1}{2} t_{\{1,2,3\}} t_{\{2,3,4\}}\\&+&\frac{3}{2}
 t_{\{1,2,4\}} t_{\{2,3,4\}}+t_{\{1\}} t_{\{3\}}.
\end{eqnarray*}

Lastly, we consider Figure \ref{134234}.

\begin{figure}[ht!]
\includegraphics[scale=0.15]{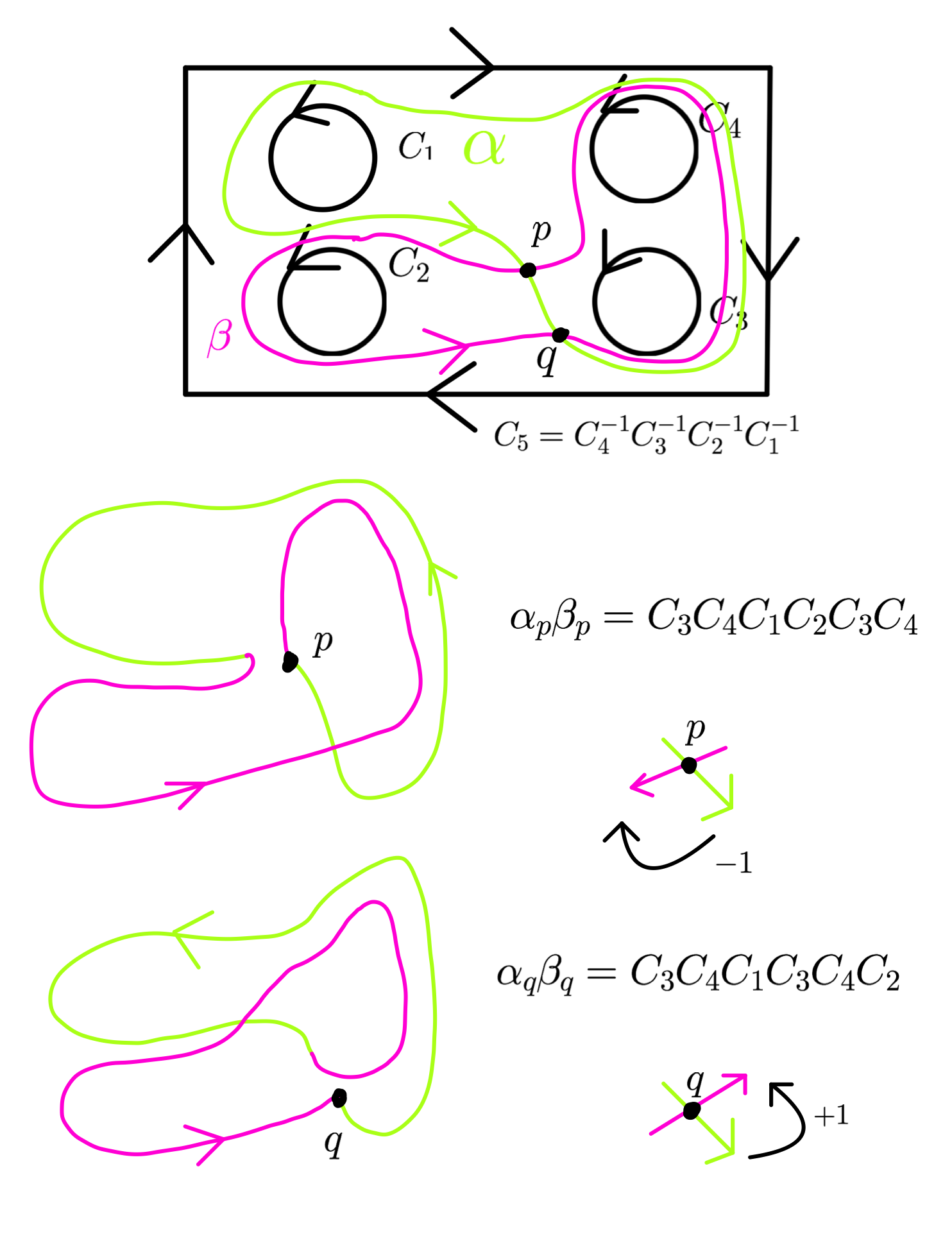} \caption{$\{t_{\{1,3,4\}},t_{\{2,3,4\}}\}$}
\label{134234} 
\end{figure}

\begin{eqnarray*}
\{t_{\{1,3,4\}},\,t_{\{2,3,4\}}\}&=&(-1)(\tr(c_3c_4c_1c_2c_3c_4)-\frac{1}{2}\tr(c_1c_3c_4)\tr(c_2c_3c_4))\\&+&(+1)(\tr(c_3c_4c_1c_3c_4c_2)-\frac{1}{2}\tr(c_1c_3c_4)\tr(c_2c_3c_4))\\
&=\,&\frac{1}{2} t_{\{1\}} t_{\{2\}} t_{\{3,4\}}^2-\frac{1}{2}
 t_{\{1,2\}} t_{\{3,4\}}^2-\frac{1}{2} t_{\{1\}} t_{\{2\}}
 t_{\{3\}} t_{\{4\}} t_{\{3,4\}}\\&+&\frac{1}{2} t_{\{3\}} t_{\{4\}}
 t_{\{1,2\}} t_{\{3,4\}}+\frac{1}{2} t_{\{2\}} t_{\{3\}}
 t_{\{1,4\}} t_{\{3,4\}}+\frac{1}{2} t_{\{1\}} t_{\{4\}}
 t_{\{2,3\}} t_{\{3,4\}}\\&-&\frac{1}{2} t_{\{1,4\}} t_{\{2,3\}}
 t_{\{3,4\}}+\frac{1}{2} t_{\{1,3\}} t_{\{2,4\}}
 t_{\{3,4\}}-\frac{1}{2} t_{\{4\}} t_{\{1,2,3\}}
 t_{\{3,4\}}\\&-&\frac{1}{2} t_{\{3\}} t_{\{1,2,4\}}
 t_{\{3,4\}}-\frac{1}{2} t_{\{2\}} t_{\{1,3,4\}}
 t_{\{3,4\}}-\frac{1}{2} t_{\{1\}} t_{\{2,3,4\}} t_{\{3,4\}}+2
 t_{\{1,2\}}\\&+&t_{\{1,3,4\}} t_{\{2,3,4\}}-t_{\{1\}} t_{\{2\}}.
\end{eqnarray*}

\subsubsection{The Bi-vector and Symmetry}

Let $\mu_{\sigma}$ be the mapping class that corresponds to the permutation $\sigma$ of the boundary components of 
$\Sigma_{5,0}$; we will use cycle notation for permutations. We consider formal sums of such permutations in the 
integral group ring associated to the mapping class group of $\Sigma_{5,0}$, and observe that elements in the group 
ring acts on the coordinate ring of $\X_{5,0}(\SL(2,\C))$ since the mapping class group acts on 
$\X_{5,0}(\SL(2,\C))$.

Let $\Sigma_1\,=\,\mu_{(1)}+\mu_{(34)}+\mu_{(23)}+\mu_{(123)}+\mu_{(124)}+\mu_{(142)}+\mu_{(143)}+ 
\mu_{(13)(24)}+\mu_{(1234)}+\mu_{(1342)}+\mu_{(1324)}+\mu_{(1432)}$, 
$\Sigma_2\,=\,\mu_{(1)}+\mu_{(1234)}+\mu_{(13)(24)}+\mu_{(1432)}$ and $\Sigma_3\,=\,\mu_{(1)}+\mu_{(1432)}$.

The above calculations, after observing symmetry (verified using {\it Mathematica}), establish the following form 
for the bi-vector:
\begin{eqnarray*}&&\mathfrak{a}_{5,0}(\SL(2,\C))\,=\,\mathfrak{a}_{1324}\frac{\partial}{\partial t_{\{1,3\}}}\wedge
\frac{\partial}{\partial t_{\{2,4\}}}+\Sigma_1\left(\mathfrak{a}_{1214}\frac{\partial}{\partial t_{\{1,2\}}}\wedge
\frac{\partial}{\partial t_{\{1,4\}}}\right)\\&+&\Sigma_2\left(\mathfrak{a}_{12314}
\frac{\partial}{\partial t_{\{1,2,3\}}}\wedge \frac{\partial}{\partial t_{\{1,4\}}}+\mathfrak{a}_{12324}
\frac{\partial}{\partial t_{\{1,2,3\}}}\wedge \frac{\partial}{\partial t_{\{2,4\}}}+\mathfrak{a}_{12334}
\frac{\partial}{\partial t_{\{1,2,3\}}}\wedge \frac{\partial}{\partial t_{\{3,4\}}}\right)\\&+
&\Sigma_2\left(\mathfrak{a}_{123124}\frac{\partial}{\partial t_{\{1,2,3\}}}\wedge
\frac{\partial}{\partial t_{\{1,2,4\}}}\right)+\Sigma_3\left(\mathfrak{a}_{123134}
\frac{\partial}{\partial t_{\{1,2,3\}}}\wedge \frac{\partial}{\partial t_{\{1,3,4\}}}\right),
\end{eqnarray*} 
where:
\begin{eqnarray*} 
\mathfrak{a}_{1324}&=&t_{\{3\}} t_{\{4\}} t_{\{1,2\}}-2 t_{\{3,4\}} t_{\{1,2\}}-t_{\{2\}}t_{\{3\}} t_{\{1,4\}}-t_{\{1\}} t_{\{4\}} t_{\{2,3\}}+2t_{\{1,4\}} t_{\{2,3\}}\\&+&t_{\{1\}} t_{\{2\}} t_{\{3,4\}},\\ 
\mathfrak{a}_{1214}&=&t_{\{1,2\}} t_{\{1,4\}}+2 t_{\{2,4\}}-t_{\{1\}}t_{\{1,2,4\}}-t_{\{2\}} t_{\{4\}},\\
\mathfrak{a}_{12314}&=&\frac{1}{2} t_{\{4\}} t_{\{1\}}^2 t_{\{2,3\}}+\frac{1}{2} t_{\{2\}}t_{\{1\}}^2 t_{\{3,4\}}-\frac{1}{2} t_{\{1\}}^2
t_{\{2,3,4\}}+\frac{1}{2} t_{\{3\}} t_{\{4\}} t_{\{1\}}t_{\{1,2\}}\\&+&\frac{1}{2} t_{\{2\}} t_{\{3\}} t_{\{1\}}t_{\{1,4\}}-\frac{1}{2} t_{\{1\}} t_{\{1,4\}}t_{\{2,3\}}+\frac{1}{2} t_{\{1\}} t_{\{1,3\}}t_{\{2,4\}}-\frac{1}{2} t_{\{1\}} t_{\{1,2\}}t_{\{3,4\}}\\&-&\frac{1}{2} t_{\{4\}} t_{\{1\}}t_{\{1,2,3\}}-\frac{1}{2} t_{\{3\}} t_{\{1\}}t_{\{1,2,4\}}-\frac{1}{2} t_{\{2\}} t_{\{1\}}t_{\{1,3,4\}}-t_{\{4\}} t_{\{2,3\}}+t_{\{1,4\}} t_{\{1,2,3\}}\\&+&2t_{\{2,3,4\}}-\frac{1}{2} t_{\{2\}} t_{\{3\}} t_{\{4\}}t_{\{1\}}^2,\\
\mathfrak{a}_{12324}&=& -t_{\{3\}} t_{\{1,4\}}+t_{\{1\}} t_{\{3,4\}}+t_{\{2,3\}}t_{\{1,2,4\}}-t_{\{1,2\}} t_{\{2,3,4\}},\\
\mathfrak{a}_{12334}&=& -\frac{1}{2} t_{\{4\}} t_{\{3\}}^2 t_{\{1,2\}}-\frac{1}{2} t_{\{2\}}t_{\{3\}}^2 t_{\{1,4\}}+\frac{1}{2} t_{\{3\}}^2t_{\{1,2,4\}}-\frac{1}{2} t_{\{1\}} t_{\{4\}} t_{\{3\}}t_{\{2,3\}}\\&+&\frac{1}{2} t_{\{3\}} t_{\{1,4\}}t_{\{2,3\}}-\frac{1}{2} t_{\{3\}} t_{\{1,3\}}t_{\{2,4\}}-\frac{1}{2} t_{\{1\}} t_{\{2\}} t_{\{3\}}t_{\{3,4\}}+\frac{1}{2} t_{\{3\}} t_{\{1,2\}}t_{\{3,4\}}\\&+&\frac{1}{2} t_{\{4\}} t_{\{3\}}t_{\{1,2,3\}}+\frac{1}{2} t_{\{2\}} t_{\{3\}}t_{\{1,3,4\}}+\frac{1}{2} t_{\{1\}} t_{\{3\}}t_{\{2,3,4\}}+t_{\{4\}} t_{\{1,2\}}-t_{\{3,4\}} t_{\{1,2,3\}}\\&-&2t_{\{1,2,4\}}+\frac{1}{2} t_{\{1\}} t_{\{2\}} t_{\{4\}}t_{\{3\}}^2,\\
\mathfrak{a}_{123124}&=& \frac{1}{2} t_{\{3\}} t_{\{4\}} t_{\{1,2\}}^2-\frac{1}{2}t_{\{3,4\}} t_{\{1,2\}}^2-\frac{1}{2} t_{\{1\}} t_{\{2\}}t_{\{3\}} t_{\{4\}} t_{\{1,2\}}+\frac{1}{2} t_{\{2\}} t_{\{3\}}t_{\{1,4\}} t_{\{1,2\}}\\&+&\frac{1}{2} t_{\{1\}} t_{\{4\}}t_{\{2,3\}} t_{\{1,2\}}-\frac{1}{2} t_{\{1,4\}} t_{\{2,3\}}t_{\{1,2\}}+\frac{1}{2} t_{\{1,3\}} t_{\{2,4\}}t_{\{1,2\}}+\frac{1}{2} t_{\{1\}} t_{\{2\}} t_{\{3,4\}}t_{\{1,2\}}\\&-&\frac{1}{2} t_{\{4\}} t_{\{1,2,3\}}t_{\{1,2\}}-\frac{1}{2} t_{\{3\}} t_{\{1,2,4\}}t_{\{1,2\}}-\frac{1}{2} t_{\{2\}} t_{\{1,3,4\}}t_{\{1,2\}}-\frac{1}{2} t_{\{1\}} t_{\{2,3,4\}} t_{\{1,2\}}\\&+&2t_{\{3,4\}}+t_{\{1,2,3\}} t_{\{1,2,4\}}-t_{\{3\}} t_{\{4\}},\\
\mathfrak{a}_{123134}&=&\frac{1}{2} t_{\{2,4\}} t_{\{1,3\}}^2-\frac{1}{2} t_{\{1\}}t_{\{2\}} t_{\{3\}} t_{\{4\}} t_{\{1,3\}}+\frac{1}{2} t_{\{3\}}
t_{\{4\}} t_{\{1,2\}} t_{\{1,3\}}+\frac{1}{2} t_{\{2\}} t_{\{3\}}t_{\{1,4\}} t_{\{1,3\}}\\&+&\frac{1}{2} t_{\{1\}} t_{\{4\}}t_{\{2,3\}} t_{\{1,3\}}-\frac{1}{2} t_{\{1,4\}} t_{\{2,3\}}t_{\{1,3\}}+\frac{1}{2} t_{\{1\}} t_{\{2\}} t_{\{3,4\}}t_{\{1,3\}}-\frac{1}{2} t_{\{1,2\}} t_{\{3,4\}}t_{\{1,3\}}\\&-&\frac{1}{2} t_{\{4\}} t_{\{1,2,3\}}t_{\{1,3\}}-\frac{1}{2} t_{\{3\}} t_{\{1,2,4\}}t_{\{1,3\}}-\frac{1}{2} t_{\{2\}} t_{\{1,3,4\}}t_{\{1,3\}}-\frac{1}{2} t_{\{1\}} t_{\{2,3,4\}} t_{\{1,3\}}\\&-&2t_{\{2,4\}}-2 t_{\{2,3\}} t_{\{3,4\}}+t_{\{1,2,3\}}t_{\{1,3,4\}}+2 t_{\{3\}} t_{\{2,3,4\}}+t_{\{2\}} t_{\{4\}} .
\end{eqnarray*}

Moreover, using all boundary permutations in the mapping class group (including the permutations of the fifth boundary $C_5$), we show there do not exist any further symmetries of this type. In other words, the above symmetry is sharp. We used a {\it 
Mathematica} notebook for a proof by exhaustion (we checked all 120 induced mappings explicitly).

\bibstyle{alpha}

\end{document}